\documentclass{amsart}
\usepackage{amssymb, amsmath, amscd, latexsym, mathrsfs, appendix}
\usepackage{graphics}
\usepackage{amsmath}
\usepackage{amsxtra}
\usepackage{amsfonts}
\usepackage{appendix}
\usepackage[all]{xy}

\usepackage{color}

\usepackage{tikz}
\usetikzlibrary{matrix,arrows}

\usepackage{hyperref}
\hypersetup{
    colorlinks,
    citecolor=black,
    filecolor=black,
    linkcolor=black,
    urlcolor=black
}

\numberwithin{equation}{section}
\newtheorem{thm}{Theorem}[section]
\newtheorem{cor}[thm]{Corollary}
\newtheorem{lem}[thm]{Lemma}

\newtheorem{prop}[thm]{Proposition}
\newtheorem{defn}[thm]{Definition}

\newtheorem{rem}[thm]{Remark}
\newtheorem{expl}[thm]{Example}

\newcommand{\lra}{\longrightarrow}
\newcommand{\vlra}[1]{\xrightarrow{\makebox[#1]{}}}
\newcommand{\xvlra}[2]{\xrightarrow{\makebox[#1]{\footnotesize #2}}}
\newcommand{\vlla}[1]{\xleftarrow{\makebox[#1]{}}}

\newcommand{\co}{\colon\!}

\newcommand{\smin}{\smallsetminus}

\newcommand{\heq}{\textup{he}}

\newcommand{\id}{\textup{id}}
\newcommand{\im}{\textup{im}}

\newcommand{\holim}{\textup{holim}}
\newcommand{\hocolim}{\textup{hocolim}}
\newcommand{\colim}{\textup{colim}}
\newcommand{\hofiber}{\textup{hofiber}}
\newcommand{\mor}{\textup{mor}}
\newcommand{\ob}{\textup{ob}}
\newcommand{\map}{\textup{map}}
\newcommand{\imap}{\textup{imap}}
\newcommand{\rmap}{\mathbb R\textup{map}}

\newcommand{\TOP}{\textup{TOP}}

\newcommand{\Or}{\textup{O}}

\newcommand{\emb}{\textup{emb}}
\newcommand{\config}{\mathsf{con}} 

\newcommand{\conloc}[1]{\config^\loc(#1)}

\newcommand{\loc}{\textup{loc}}
\newcommand{\imm}{\textup{imm}}
\newcommand{\hodim}{\textup{hdim}}
\newcommand{\man}{\mathsf{Man}}
\newcommand{\disc}{\mathsf{Disc}}
\newcommand{\fin}{\mathsf{Fin}}
\newcommand{\finplus}{{\fin_*}}
\newcommand{\fino}{{\mathsf{Fino}_*}}

\newcommand{\finplusk}{\fin^k_*}
\newcommand{\finok}{\mathsf{Fino}^k_*}
\newcommand{\dend}{\mathsf{Tree}}
\newcommand{\dendca}{\dend^c}
\newcommand{\dendroca}{\dend^{rc}}
\newcommand{\rel}{\textup{ver}}
\newcommand{\simp}{\mathsf{simp}}
\newcommand{\nsimp}{\mathsf{nsimp}}
\newcommand{\holink}{\textup{holink}}

\newcommand{\sA}{\mathcal A}
\newcommand{\sB}{\mathcal B}
\newcommand{\sC}{\mathcal C}
\newcommand{\sD}{\mathcal D}
\newcommand{\sE}{\mathcal E}

\newcommand{\sS}{\mathcal S}

\newcommand{\op}{\textup{op}}

\renewcommand{\sphat}{\hat{\rule{1.5mm}{0mm}}}

\newcommand{\ms}{\mathcal}

\newcommand{\LL}{\mathbb L}
\newcommand{\NN}{\mathbb N}
\newcommand{\RR}{\mathbb R}
\newcommand{\ZZ}{\mathbb Z}

\newcommand{\xlongleftarrow}[1]{\overset{#1}{\longleftarrow}}
\newcommand{\xlongrightarrow}[1]{\overset{#1}{\longrightarrow}}

\newcommand{\twosub}[2]{\begin{array}{cc}
\scriptstyle{#1} \\  [-1mm] \scriptstyle{#2}  \end{array}}

\newcommand{\colimsub}[1]{\begin{array}[t]{cc} \textup{colim} \\
[-1.7mm] \scriptstyle{#1} \end{array}}

\newcommand{\holimsub}[1]{\begin{array}[t]{cc} \textup{holim} \\ [-1mm]
\scriptstyle{#1} \end{array}}

\newcommand{\hocolimsub}[1]{\begin{array}[t]{cc} \textup{hocolim} \\
[-1.2mm] \scriptstyle{#1} \end{array}}

\newcommand{\uli}{\underline}

\newcommand{\cyl}{\textup{Cyl}}

\newcommand{\PSh}{\mathsf{PSh}}
\newcommand{\Hom}{\textup{Hom}}

\newcommand{\sSp}{s\mathcal S}

\newcommand{\LK}[1]{{#1}_!}
\newcommand{\RK}[1]{{#1}_*}

\newcommand{\hLK}[1]{\LL \LK{#1}}
\newcommand{\hRK}[1]{\RR \RK{#1}}

\title[Embeddings and configuration categories]{Spaces of smooth embeddings and configuration categories}
\author{Pedro Boavida de Brito and Michael Weiss}%

\address{Mathematics Institute, WWU M\"{u}nster, Einsteinstrasse 62, 48149 M\"{u}nster \\
   Germany } \email{m.weiss@uni-muenster.de}
\address{Dept. of Mathematics, Instituto Superior Tecnico, U Lisboa, Av. Rovisco Pais, 1049-001 Lisboa \\
   Portugal } \email{pedrobbrito@tecnico.ulisboa.pt}

\thanks{Both authors gratefully acknowledge the support of the Bundesministerium f\"ur Bildung und
Forschung through the A.v.Humboldt foundation (Humboldt professorship for Michael Weiss, 2012-2017). The first author was also partly supported
by FCT grants SFRH/BD/61499/2009 and SFRH/BPD/99841/2014.}

\begin{document}
\maketitle
\begin{abstract} In the homotopical study of spaces of smooth embeddings,
the functor calculus method (Goodwillie-Klein-Weiss manifold calculus) has
opened up important connections to operad theory.
Using this and a few simplifying observations, we arrive at an operadic description
of the obstructions to deforming smooth immersions into
smooth embeddings. We give an application which in some respects improves on recent
results of Arone-Turchin and Dwyer-Hess concerning high-dimensional variants
of spaces of long knots.
\end{abstract}

\setcounter{tocdepth}{1}
\tableofcontents


\section{Introduction}
 \label{sec-intro}
 \subsection{Embeddings and configuration categories}

Let $M^m$ and $N^n$ be smooth manifolds and let $\emb(M,N)$ be the space of smooth embeddings from $M$ to $N$.
In the standard formulation of manifold calculus applied to spaces of smooth embeddings, $\emb(M,N)$ is approximated by
(the homotopy limit of) a diagram made up of spaces $\emb(S\times\RR^m,N)$, one for each embedding $e\co S\times \RR^m\to M$
where $S$ is a variable finite set. We refer the reader to \cite{WeissEmb} and \cite{BoavidaWeiss} for more details.

Our first main result, Theorem \ref{thm-mainintro} below, builds on this model but in addition separates the linear algebra (i.e., differential) data from the configuration space data, and also reveals the mechanism by which they are held together.

More precisely, let $\config(M)$ denote the $\infty$-category whose objects are configurations of points in $M$ and whose morphisms are paths of configurations where collisions are allowed but cannot be undone. This category, which we call the \emph{configuration category} of $M$, comes equipped with a forgetful reference functor to $\fin$, essentially the category of finite sets.

\begin{thm}\label{thm-mainintro}
If $n - m \geq 3$, there is a homotopy cartesian square
\begin{equation} \label{eq:squaremain}
	\begin{tikzpicture}[descr/.style={fill=white}, baseline=(current bounding box.base)] ]
	\matrix(m)[matrix of math nodes, row sep=2.5em, column sep=2.5em,
	text height=1.5ex, text depth=0.25ex]
	{
	\emb(M,N) & \RR \map_{\fin}(\config(M), \config(N)) \\
	\imm(M,N) & \Gamma \\
	};
	\path[->,font=\scriptsize]
		(m-1-1) edge node [auto] {} (m-1-2);
	\path[->,font=\scriptsize]
		(m-2-1) edge node [auto] {} (m-2-2);
	\path[->,font=\scriptsize]
		(m-1-1) edge node [left] {} (m-2-1);
	\path[->,font=\scriptsize] 		
		(m-1-2) edge node [auto] {} (m-2-2);
	\end{tikzpicture}
\end{equation}
where $\Gamma$ denotes the space of sections of a fiber bundle with base $M$ and fiber over $x \in M$ given by
the space of pairs $(y,F)$ with $y\in N$ and $F$ a derived map from $\config(T_x M)$ to $\config(T_{y} N)$ over $\fin$.
\end{thm}

The vertical maps are forgetful maps. The horizontal maps are essentially restriction maps. To understand the lower one, we have the
theorem of Smale-Hirsch which says that $\imm(M,N)$ is weakly equivalent to the space of sections of a fiber bundle
with base $M$ and fiber over $x\in M$ consisting of pairs $(y,A)$ where $y\in N$ and $A$ is a linear embedding $T_x M \to T_yN$. The codimension restriction in Theorem \ref{eq:squaremain} comes from the deep multiple disjunction lemmas of Goodwillie \cite{GoThes} and Goodwillie-Klein \cite{GK1}, \cite{GK2} as explained in Goodwillie-Weiss \cite{GoWeEmb}.

\subsection{Configuration categories and the little disks operad}
In the last ten years or so, a connection to operad theory has emerged in the study of spaces of smooth embeddings. A key
step in our investigation is a description of the derived mapping space between little disks operads in terms of configuration categories:

\begin{thm}\label{thm-operintro}
$\rmap(E_m, E_n) \xrightarrow{\simeq} \rmap_{\fin}(\config(\RR^m), \config(\RR^n))$
\end{thm}

This theorem and its proof are discussed in section \ref{sec-oper}.

\subsection{Long knots and higher variants} There is a variant of theorem \ref{thm-mainintro} for smooth manifolds
$M$, $N$ with boundary and smooth embeddings $M\to N$ which extend a prescribed embedding $\partial M\to \partial N$. In the special case when $M$ and $N$ are disks, we obtain a sharper statement which epitomizes the relation between the
little disks operad and spaces of smooth embeddings.
\begin{thm}\label{thm-main2intro}
For $n-m\ge 3$, there is a homotopy fiber sequence
\begin{equation}\label{eqn-op}
\emb_{\partial}(D^m,D^n) \to \imm_{\partial}(D^m,D^n) \to \Omega^m \rmap(E_m, E_n)
\end{equation}
of $m$-fold loop spaces, where $E_i$ denotes the little $i$-disks operad.
\end{thm}
The symbol $\partial$ signals the constraint that embeddings (and immersions) agree with the standard inclusion near the boundary.
This theorem is more a corollary of other theorems, as will be explained in a moment. It improves on some corollaries of related theorems in the work of Dwyer-Hess \cite{DwyerHess1}, \cite{DwyerHess2}, Arone-Turchin \cite{AroneTurchin} and Turchin \cite{Turchin} (see also remark \ref{aroneturchin}).

\medskip
We deduce theorem \ref{thm-main2intro} from a version of theorem \ref{thm-mainintro} for manifolds with boundary, plus
theorem \ref{thm-operintro} and the following Alexander-trick type result whose proof occupies
sections~\ref{sec-face}, \ref{sec-shadow} and~\ref{sec-knottynew}.

\begin{thm}\label{thm-alex}
$\rmap_{\finplus}^{\partial}(\config(D^m), \config(D^n)) \simeq *$\,.
\end{thm}

\subsection{Further applications}

\subsubsection*{Truncations and low codimensions.}
Theorem \ref{thm-main2intro} has a sharper version in terms of the Taylor tower for the embedding functor, which holds even when
the codimension is less than three. Namely, for $k\ge 1$ there is a homotopy fiber sequence
\begin{equation}\label{eqn-truncated}
T_k \emb_{\partial}(D^m,D^n) \to \imm_{\partial}(D^m,D^n) \to \Omega^m \rmap_\fin(\config(\RR^m;k),\config(\RR^n))
\end{equation} 
where $\imm_{\partial}(D^m,D^n)\simeq T_1\emb_{\partial}(D^m,D^n)$ and where $\config(\RR^m;k)$ is the version of $\config(\RR^m)$ truncated
at cardinality $k$. (There is also a corresponding truncated version of Theorem \ref{thm-mainintro}.) Therefore the homotopy
fiber of the left-hand map in (\ref{eqn-truncated}), also commonly denoted
\[ T_k \overline{\emb}_\partial(D^m, D^n), \]
is an $(m+1)$-fold loop space.

Let us now concentrate on the case $m = 1$. For $m=1$ and $k=2$ the right-hand map in (\ref{eqn-truncated}) is a weak equivalence. The source $\imm_\partial(D^1,D^n)$ can be identified with $\Omega S^{n-1}$ by means of the Smale-Hirsch
$h$-principle. As for the target, by restriction to the subcategory of $\config(\RR;2)$ consisting of the configurations of cardinality two and morphisms
without collisions, we obtain a map from $\rmap_\fin(\config(\RR;2),\config(\RR^n))$ to the space of $\Sigma_2$-equivariant maps from $S^0$ to $\emb(\{1,2\}, \RR^n)$. This map is a weak homotopy equivalence and the target space is identified with $S^{n-1}$.

Therefore if $m=1$ and $k\ge 2$, the right-hand map in (\ref{eqn-truncated}), call it $f$,
has a left homotopy inverse given by $\Omega g$ where
\[  g\co\rmap_\fin(\config(\RR^m;k),\config(\RR^n))\to \rmap_\fin(\config(\RR^m;2),\config(\RR^n)) \]
is the restriction. It follows that
\[
T_k \emb_{\partial}(D^1,D^n) \simeq \hofiber(f)\simeq\Omega\hofiber(\Omega g)=\Omega^2\hofiber(g).
\]
Hence
\begin{cor}
For $k > 1$ and $n > 0$, the space $T_k \emb_{\partial}(D^1,D^n)$ is a double loop space, with double delooping given by $\hofiber(g)$.
\end{cor}
In particular, if $n > 0$, we see that $\pi_0 T_{k} \emb_{\partial}(D^1, D^n)$ is an abelian group (cf.~\cite{BCSK}). If $n \geq 4$ we also recover the result of Salvatore \cite{Salvatore} that $\emb_{\partial}(D^1,D^n)$ is a
double loop space.

\subsubsection*{Homeomorphisms.}
Let $\TOP(n)$ be the topological group of homeomorphisms of $\RR^n$ and $\TOP(n,m)$ the subgroup of those homeomorphisms fixing $\RR^m$ pointwise. Write $\TOP(n) / \TOP(n, m)$ for the homotopy fiber of the map $B\TOP(n,m) \to B\TOP(n)$.
\begin{thm} Suppose $n - m \geq 3$ and $n \geq 5$. Fix the standard inclusion of $\RR^m$ in $\RR^n$ (and hence a map of operads from $E_m$ to $E_n$). There is a based map
\[
\TOP(n) / \TOP(n, m) \to \rmap(E_m, E_n)
\]
which induces an isomorphism on $\pi_j$ for $j > m$.
\end{thm}

\begin{proof} (Compare \cite{Sakai}). There is a commutative diagram
\[
\begin{tikzpicture}[descr/.style={fill=white}, baseline=(current bounding box.base)] ]
	\matrix(m)[matrix of math nodes, row sep=2.5em, column sep=2.5em,
	text height=1.5ex, text depth=0.25ex]
	{
	\emb_\partial(D^m,D^n) & \emb_\partial^\TOP(D^m,D^n) & \rmap^\partial_{\fin}(\config(D^m),\config(D^n))  \\
	\imm_\partial(D^m,D^n) & {\imm^\TOP_\partial(D^m,D^n)} & \Omega^m\rmap(E_m, E_n) \\
	};
	\path[->,font=\scriptsize]
		(m-1-1) edge node [auto] {} (m-1-2)
		(m-1-2) edge node [auto] {} (m-1-3)
		(m-2-1) edge node [auto] {} (m-2-2)
		(m-2-2) edge node [auto] {} (m-2-3)
		(m-1-1) edge node [auto] {} (m-2-1)
		(m-1-2) edge node [auto] {} (m-2-2)
		(m-1-3) edge node [auto] {} (m-2-3);
	\end{tikzpicture}
\]
Here $\emb_\partial^\TOP(D^m,D^n)$ and $\imm_{\partial}^{\TOP}(D^m, D^n)$ denote the spaces of locally flat topological
embeddings and locally flat topological immersions, respectively. By the topological form
of the Smale-Hirsch $h$-principle (Lees-Lashof), the immersion space is weakly equivalent to
$\Omega^m \TOP(n)/\TOP(n,m)$. The outer rectangle is homotopy
cartesian (special case of theorem~\ref{thm-mainintro} in a formulation for manifolds with boundary).
By \cite{Lashof}, the left-hand square is homotopy cartesian and $\emb_\partial^\TOP(D^m,D^n)$ is contractible. Therefore, if we modify the lower middle term $\imm^\TOP_\partial(D^m,D^n)$ by deleting the path components which do not contain any smooth immersions, then Lemma \ref{lem-hococomp} applies
and we conclude that the right-hand lower horizontal arrow induces isomorphisms in $\pi_j$ for $j > 0$.
\end{proof}

\subsection*{Outline}
We now give a brief outline of the paper. Throughout this text, we make extensive use of $\infty$-categories in the framework of Rezk, that of complete Segal spaces. We briefly recall all the necessary background in section \ref{sec-topcat} with a focus on examples. The reader familiar with $\infty$-categories may wish to jump directly to section \ref{sec-topcatFC} where we give a mild generalization of Rezk's model which we call \emph{fiberwise} complete Segal spaces.

The core of the paper begins in section \ref{sec-conman}, where we introduce the notion of \emph{configuration category} and discuss several equivalent models for it. We then show that the configuration category of a manifold $M$ has a multi-local to global property, i.e. it is a homotopy cosheaf as a functor in the variable $M$.

In section \ref{sec-loc}, we define a local variant of a configuration category and prove a (stronger) local to global property for these. In section \ref{sec-mainthm}, we state and prove the main theorem. The boundary version of the previous sections is spelled out in section \ref{sec-bdry}.

The operadic interpretations and the relationship with fiberwise complete Segal spaces and configuration categories, of which theorem \ref{thm-operintro} is the high point, are explained in section \ref{sec-oper}.

In the long sections~\ref{sec-face} and~\ref{sec-shadow} and the short section~\ref{sec-knottynew}, we prove
theorem \ref{thm-alex} from scratch. We believe that the ideas developed in sections~\ref{sec-face} and~\ref{sec-shadow}
may also have independent interest.

We also include four appendices. Appendix A collects a few known results about derived mapping spaces and
homotopy Kan extensions which are used throughout the text. Appendix B contains new results. Namely, we provide a simplicial
model structure on simplicial spaces over a fixed simplicial space which models fiberwise complete Segal spaces; the key result
is the identification of the fibrations between Segal spaces (not necessarily complete) in Rezk's model structure for complete Segal spaces.
Appendix C collects some results which are useful in the computation of derived mapping spaces.
In Appendix D, we include proofs for some postponed technical lemmas.

\section{Topological categories versus complete Segal spaces} \label{sec-topcat}
\subsection{Segal spaces} By the nerve of a discrete small category $\ms C$ we mean the simplicial set $N\ms C$ whose set of $n$-simplices
is the set of strings of composable morphisms
\[   c_0\leftarrow c_1\leftarrow \cdots \leftarrow c_n \]
in $\ms C$. In particular the set of $0$-simplices is the set of objects of $\ms C$,
the set of $1$-simplices is the set of all morphisms
in $\ms C$, and $d_0$ on $1$-simplices is the operator \emph{source} while $d_1$ on $1$-simplices is
the operator \emph{target}. With these conventions we follow Bousfield-Kan \cite{BousfieldKan}.

Grothendieck observed that the \emph{nerve} functor from the category of (small, discrete) categories to
the category of simplicial sets is fully faithful. That is to say, a small discrete category can be
viewed as a simplicial set $X$ with the extra property that for each $n\ge 2$ the following map of sets
is a bijection:
\begin{equation} \label{eqn-nerve}
(u_1^*,u_2^*,\dots,u_n^*)\co X_n \lra X_1\times_{X_0} X_1\times_{X_0}\cdots\times_{X_0} X_1~.
\end{equation}
The $u_i^*$ are iterated face operators corresponding to
the weakly order-preserving maps $u_i\co\{0,1\}\to\{0,1,2,\dots,n\}$ defined by $u_i(0)=i-1$ and $u_i(1)=i$.

Following Grothendieck
we might take the view that a small topological category is by definition the same thing as a simplicial \emph{space} $X$
with the extra property that for each $n\ge 2$ the map of spaces
\begin{equation} \label{eqn-spacenerve}
(u_1^*,u_2^*,\dots,u_n^*)\co X_n \lra X_1\times_{X_0} X_1\times_{X_0}\cdots\times_{X_0} X_1
\end{equation}
which generalizes the one in~(\ref{eqn-nerve})
is a homeomorphism. In this situation the target of $(u_1^*,u_2^*,\dots,u_n^*)$ is of course
the topological limit (alias iterated pullback) of the diagram
\[
X_1 \xrightarrow{ \; d_0 \;} X_0 \xleftarrow{ \; d_1 \;} X_1  \xrightarrow{ \; d_0 \;} \cdots \xrightarrow{ \; d_0 \;} X_0 \xleftarrow{ \; d_1 \;} X_1
\]
Unfortunately, topological categories defined like that come with many
pitfalls, especially for those who expect reasonable homotopy behavior. It is natural to look for alternative conditions which guarantee good homotopy behavior. Such a set of conditions has been proposed by Charles Rezk \cite{Rezk}. The resulting concept, of a simplicial space satisfying additional conditions of a homotopy theoretic nature, is what he calls a \emph{complete Segal space}.

\begin{defn} \label{defn-Segalspace} {\rm A \emph{Segal space} is a simplicial space $X$ satisfying condition ($\sigma$) below. If condition ($\kappa$) below is also satisfied, then $X$ is a \emph{complete Segal space}.}
\end{defn}

\begin{itemize}
\item[($\sigma$)] For each $n\ge 2$ the map $(u_1^*,u_2^*,\dots,u_n^*)$ from $X_n$ to the homotopy limit of
the diagram
\[
X_1 \xrightarrow{ \; d_0 \;} X_0 \xleftarrow{ \; d_1 \;} X_1  \xrightarrow{ \; d_0 \;} \cdots \xrightarrow{ \; d_0 \;} X_0 \xleftarrow{ \; d_1 \;} X_1
\]
 is a weak homotopy equivalence.
\end{itemize}
In order to formulate condition ($\kappa$) we introduce some vocabulary based on ($\sigma$). We call an element
$z$ of $\pi_0X_1$ \emph{homotopy left invertible} if there is an element $x$ of $\pi_0X_2$ such that $d_0x=z$
and $d_1x$ is in the image of $s_0\co \pi_0X_0\to \pi_0 X_1$. (In such a case $d_2x$ can loosely be thought of as
a left inverse for $z=d_0x$. Indeed $d_1x$ can loosely be thought of as the composition $d_2x\circ d_0x$, and by
assuming that this is in the image of $s_0$ we are saying that it is in the path component of an identity morphism.
We have written $d_0$, $d_1$, $s_0$ etc.~for maps induced on $\pi_0$ by the face and degeneracy operators.)
We call $z$ \emph{homotopy right invertible} if there is
an element $y$ of $\pi_0X_2$ such that $d_2y=z$
and $d_1x$ is in the image of $s_0\co \pi_0X_0\to \pi_0 X_1$.
Finally $z\in \pi_0X_1$ is \emph{homotopy invertible} if it is both homotopy left invertible and homotopy right invertible.
Let $X_1^{\heq}$ be the union of the homotopy invertible path components of $X_1$\,. It is a subspace of $X_1$\,.

\smallskip
\begin{itemize}
\item[($\kappa$)] The map $d_0$ restricts to a weak homotopy equivalence from $X_1^{\heq}$ to $X_0$\,.
\end{itemize}

\smallskip
The meaning of this condition will be illustrated later on.

\medskip
\emph{Examples and non-examples of complete Segal spaces.} (i) Let $\ms C$ be a small (discrete) category. The nerve of $\ms C$ is a simplicial set which
we can also view as a simplicial (discrete) space. It satisfies ($\sigma$). It satisfies ($\kappa$) \emph{if and only if}
every isomorphism in $\ms C$ is an identity morphism. \newline\newline
 (ii) Let $\ms C$ be a small category enriched over the category of topological spaces. That is,
we assume that $\ms C$ has a set (discrete space) of objects, but every morphism set $\mor(x,y)$ in $\ms C$
is equipped with the structure of a topological space, and composition of morphisms is continuous. Then the nerve $N\ms C$,
formed in the usual way, satisfies ($\sigma$). \newline\newline
(iii) More generally, let $\ms C$ be a topological category (category object in the category of topological spaces). Suppose
that at least one of the maps \emph{source, target} from the space of morphisms of $\ms C$ to the space of objects of $\ms C$ is a fibration. Then the nerve $N\ms C$, formed in the usual way, satisfies ($\sigma$). Indeed, the fibration property ensures that the map from the limit of the diagram
\[
(N\ms C)_1 \xrightarrow{ \; d_0 \;} (N\ms C)_0 \xleftarrow{ \; d_1 \;} (N\ms C)_1  \xrightarrow{ \; d_0 \;} \cdots \xrightarrow{ \; d_0 \;} (N\ms C)_0 \xleftarrow{ \; d_1 \;} (N\ms C)_1
\]
to its homotopy limit is a homotopy equivalence.\newline\newline
(iv) Let $Y$ be any space. Make a simplicial space $Z$ by setting $Z_n=Y$
for all $n$. This $Z$ is a complete Segal space. We can also view it
as (the nerve of) a topological category with space of objects $Y$ and space of morphisms $Y$. \newline\newline
(v) (See also \cite{Woolf}, \cite{Treumann}) Let $Z$ be a stratified space; that is, $Z$ is a space equipped
with a locally finite partition into locally closed subsets $Z_\alpha$ (the strata) such
that the closure in $Z$ of each stratum is a union of strata. An \emph{exit path} in $Z$ is a pair consisting of
a non-negative $a\in\RR$ and a continuous map
$\gamma\co [0,a]\to Z$
such that for all $s,t\in[0,a]$ with $s\le t$, the stratum containing $\gamma(s)$ is contained in the
closure of the stratum containing $\gamma(t)$. The \emph{exit path category} $\ms C_Z$ of $Z$ has object space
\[  \coprod_{\alpha} Z_\alpha  \]
(the topological disjoint union of the strata, each stratum being equipped with the topology that it has as a subspace of $X$)
while the morphism space is the topological disjoint union, over all $(\alpha,\beta)$, of the space of exit paths starting somewhere
in $Z_\alpha$ and ending somewhere in $Z_\beta$. Composition of morphisms is Moore composition of exit paths. It is evident that the
maps \emph{source} and \emph{target} from the space of all morphisms to the space of all objects are Serre fibrations.
Therefore the (topological) nerve $N\ms C_Z$ of $\ms C_Z$ has property ($\sigma$). Furthermore, an exit path $\gamma$ in $Z$
is in a homotopy invertible component of the space of morphisms of $\ms C_Z$ if and only if it proceeds in a single stratum
$Z_\alpha$ of $Z$~; this implies immediately that property ($\kappa$) is satisfied by $N\ms C_Z$. Therefore
$N\ms C_Z$ is a complete Segal space. Miller \cite{Miller2} investigates this example, discussing specifically whether and how the stratified homotopy type of $Z$ can be reconstructed from the exit path category of $Z$.  \newline \newline
(vi) Let $\ms D$ be the category determined by the poset $\{0 < 1\}$, and let $\ms C$
be the subcategory determined by the sub-poset $\{0\}$. Let $F\co \ms C\to\ms D$ be the inclusion functor.
Then $N\ms C$ and $N\ms D$ are complete Segal spaces.
The functor $F$ induces a simplicial map $N\ms C \to N\ms D$ which is not a degreewise weak equivalence but whose induced map on
geometric realizations is a weak equivalence. We mention this in order to stress that the weak homotopy type of a complete Segal space typically carries
much more information than the weak homotopy type of its geometric realization.

\begin{defn} \label{defn-Rezkmorphism}
Let $X$ be a Segal space. Given $a$ and $b$ in $X_0$, form the homotopy fiber of
\[  (d_0,d_1)\co X_1 \lra X_0\times X_0 \]
over $(a,b)$. We call this the \emph{space of morphisms from $a$ to $b$} in $X$ and denote it by $\mor^h_X(a,b)$.
\end{defn}

Given $a$ and $b$ in $Y_0$\,, their path components $[a],[b]\in \pi_0Y_0$ are considered \emph{weakly equivalent}
if there is an element of $\pi_0(Y_1^{\heq})$ mapped to $[a]$ by $d_0$ and to $[b]$ by $d_1$.

\medskip

In order to shed some more light on condition ($\kappa$) we suppose that $f\co X\to Y$
is a map between simplicial spaces both of which satisfy ($\sigma$) but not necessarily ($\kappa$).

Following Rezk, we say that $f\co X\to Y$ is a \emph{Dwyer-Kan equivalence} if
\begin{itemize}
\item the map $\pi_0X_0\to \pi_0Y_0/\!\!\sim$ induced by $f$ is onto, where $\sim$ denotes the relation
of weak equivalence;
\item for every pair of points $x,x'$ in $X_0$ with images $y,y'$ in $Y_0$, the map $f$
induces a weak homotopy equivalence from the space of morphisms $x\to x'$ to the space of morphisms $y\to y'$.
\end{itemize}
Clearly $f$ can be a Dwyer-Kan equivalence without being a degreewise weak homotopy equivalence
of simplicial spaces. (For example, a functor between categories which induces an equivalence of categories does not have to be bijective on objects.) However, if both $X$ and $Y$ are complete Segal spaces, then it does hold that $f$ is a Dwyer-Kan equivalence if and only if it is a degreewise weak equivalence.

In \cite{Rezk}, Rezk provides a simplicial model structure on the category of simplicial spaces in which fibrant objects are the complete Segal spaces and the weak equivalences between Segal spaces are the Dwyer-Kan equivalences. A morphism in this model structure is a cofibration if it is a cofibration in the underlying (degreewise) model structure of simplicial spaces.

\begin{defn}
Let $X$ be a Segal space. A (Rezk) \emph{completion} of $X$ is a complete Segal space $X\sphat \,$ together with a Dwyer-Kan equivalence (or possibly a zigzag of them) relating $X$ and $X\sphat \,$.
\end{defn}

Rezk constructs an explicit completion \emph{functor}, $X \mapsto X\sphat \,$ (a fibrant replacement in the model structure of complete Segal spaces) with good formal properties. This highlights 
one essential feature of the completeness property ($\kappa$), namely that a map $X \to Y$ between Segal spaces is a Dwyer-Kan equivalence if and only if the induced map $X\sphat \, \to Y\sphat \,$ between completions is a degreewise weak equivalence.

Beware, however, that the completion process can have a drastic effect in all degrees, not just degree $0$. We illustrate this with an
example below.

\medskip
\emph{Examples.}
(vii) Let $G$ be a topological group which we view as a category in the usual manner, with one object $e$ whose
space of endomorphisms is $G$, and whose nerve is a Segal space $X$. It is easy to see that $X$ satisfies ($\kappa$) only if the group $G$ is weakly contractible. Indeed $X_0$ is a singleton while $X_1=X_1^{\heq}$
is the underlying space of  $G$.

It turns out that the completion of $X$ is a constant simplicial space as in example (iv), up to degreewise homotopy equivalence. In order to see this we assume for simplicity that the group $G$ was discrete. Let $Y$ be the constant simplicial space such that $Y_n=BG$ for all $n\ge 0$, where $BG$ is the geometric realization of $X$.
Let $\bar Y$ be the simplicial space such that $\bar Y_n$ is the space of maps from the geometric simplex $\Delta^n$ to $BG$. Then $Y$ and $\bar Y$
also satisfy ($\sigma$).
There are simplicial maps
\[  X \lra \bar Y \longleftarrow Y~. \]
One of these simplicial maps, $X\to \bar Y$, uses the adjoints
of characteristic maps $\Delta^n \times X_n\to BG=Y_n$. The other is, in degree $n$, the inclusion of the constant maps from $\Delta^n$
to $BG$. The map $X\to \bar Y$ is a fine example of a Dwyer-Kan equivalence. Indeed, there is only one element $*\in X_0$,
with image in $Y_0$ equal to the base point (again denoted $*$). The space of endomorphisms of $*$ in
$\bar Y$ is $\simeq\Omega(BG)\simeq G$,
which is in good agreement with the space of endomorphisms of $*$ in $X$. The inclusion map from $Y$ to $\bar Y$ is a
degreewise equivalence. So we are justified in saying that the Rezk completion of $X$ is $Y$.

\begin{rem} \label{rem-Deltas} {\rm We make an effort to distinguish in notation between the geometric $n$-simplex $\Delta^n$ and the simplicial discrete space $\Delta[n]$ whose space of $k$-simplices is the set of monotone maps $[k]\to[n]$. With our conventions for nerves, $\Delta[n]$ is also canonically isomorphic to the nerve of the poset $[n]^\op$.
}
\end{rem}

\subsection{Fiberwise complete Segal spaces}\label{sec-topcatFC}
We finish the section with a mild generalization of the concept of complete Segal space.
Let $Y$ and $B$ be simplicial spaces, both satisfying ($\sigma$).  Let $f\co Y\to B$
be a simplicial map. The following condition is an obvious variation on condition ($\kappa$) above.
\begin{itemize}
\item[($\kappa_{f}$)] The square
\[
	\begin{tikzpicture}[descr/.style={fill=white}]
	\matrix(m)[matrix of math nodes, row sep=2.5em, column sep=2.5em,
	text height=1.5ex, text depth=0.25ex]
	{
	Y_1^{\heq} & B_1^{\heq} \\
	Y_0 & B_0 \\
	};
	\path[->,font=\scriptsize]
		(m-1-1) edge node [auto] {} (m-1-2);
	\path[->,font=\scriptsize]
		(m-2-1) edge node [auto] {} (m-2-2);
	\path[->,font=\scriptsize]
		(m-1-1) edge node [left] {$d_1$} (m-2-1);
	\path[->,font=\scriptsize] 		
		(m-1-2) edge node [auto] {$d_1$} (m-2-2);
	\end{tikzpicture}
\]
is homotopy cartesian.
\end{itemize}

\begin{defn} \label{defn-rezkfunctorover} {\rm A map between Segal spaces $f\co Y\to B$ is
a \emph{fiberwise complete Segal space over $B$} if it satisfies ($\kappa_{f}$).
}
\end{defn}

In Theorem \ref{thm-fiberwisecpl}, we give a simplicial model structure describing the homotopy theory of fiberwise complete Segal spaces. The underlying category is the category of simplicial spaces over $B$; a fibrant object is a fiberwise complete Segal space over $B$ and a weak equivalence $X \to Y$ between Segal spaces over $B$ is a Dwyer-Kan equivalence (forgetting the reference maps to $B$). If $X$ and $Y$ are fiberwise complete Segal spaces then a map is a Dwyer-Kan equivalence if and only if it is a degreewise weak equivalence. 

\medskip
A map $X\to Z$ between Segal spaces
has a fiberwise Rezk completion over $B$. More precisely there is a commutative diagram (or a zigzag) of
simplicial spaces
\[
	\begin{tikzpicture}[descr/.style={fill=white}]
	\matrix(m)[matrix of math nodes, row sep=1.25em, column sep=1.25em,
	text height=1.5ex, text depth=0.25ex]
	{
	X & & Y \\
	 & B & \\
	};
		\path[->,font=\scriptsize]
		(m-1-1) edge node [auto] {} (m-1-3);
	\path[->,font=\scriptsize]
		(m-1-1) edge node [auto] {} (m-2-2);
	\path[->,font=\scriptsize]
		(m-1-3) edge node [left] {} (m-2-2);

	\end{tikzpicture}
\]
where $Y$ is fiberwise complete over $B$ and the map $X\to Y$ is a
Dwyer-Kan equivalence. For more details, see Appendix \ref{sec-model}.

\medskip
\emph{Example.} (ix)
Let $\ms C$ be a small category enriched over spaces (with a set of objects).
Let $F$ be an enriched functor from $\ms C$ to spaces. Let $\int F$ be the Grothendieck construction
alias transport category of $F$. Then the projection functor $\int F\to \ms C$ induces a map
of nerves $N\int F\to N\ms C$ which is a fiberwise complete Segal space over $N\ms C$.

\section{The configuration category of a manifold} \label{sec-conman}
\subsection{Descriptions of the configuration category}
\begin{defn}
{\rm For an integer $m\ge 0$ let $\uli m=\{1,\dots,m\}$. In particular $\uli 0$ is the empty set.
Let $\fin$ be the category with objects $\uli m$ for $m\ge 0$,
where a morphism from $\uli m$ to $\uli n$ is a map from $\uli m$ to $\uli n$.}
\end{defn}

The configuration category of a manifold $M$ is, for us, a fiberwise complete Segal space $\config(M)$ over $N\fin$. It is well
known in several incarnations, most of them related to operad theory. We try to keep the operad theory at arm's length, for as long as we can,
but even so there are several equivalent descriptions. One of the descriptions we give (we call it the
\emph{particle model}) is due to Andrade \cite{Andrade}. A similar construction was introduced by Lurie at about the same time in the context of factorization homology \cite{Lurie}.

\medskip
The manifold $M$ can be smooth or just topological; in some models of $\config(M)$ we need a smooth structure.
We assume for a start that $M$ has empty boundary. It is not required to be compact.

\bigskip
\emph{The multipatch model.}
Suppose for simplicity that $M$ is smooth and $m$-dimen\-sional.
Let $\disc$ be the category enriched over topological spaces with objects $\uli k\times\RR^m$ for $k$ a non-negative integer, where
the space of morphisms from $U=\uli k\times\RR^m$ to $V=\uli\ell\times\RR^m$ is
the space of smooth embeddings from $U$
to $V$ with the weak $C^\infty$ topology. See \cite{BoavidaWeiss} for more details.
Let $E=E_M$ be the contravariant functor on $\disc$ defined by $E(U)=\emb(U,M)$,
again with the weak $C^\infty$ topology. Form $\smallint E$, the Grothendieck construction. 
To clarify, $\smallint E$ is a category object in spaces (alias internal category). The space of objects of $\smallint E$ is the space of pairs $(k,f)$ where $k\ge 0$ and $f\co\uli k\times\RR^m\to M$ is a (codimension zero) smooth embedding. The space of \emph{all} morphisms is the coproduct
\[
\coprod_{k, \ell \geq 0} \emb(\uli k \times \RR^m, \uli \ell \times \RR^m) \times \emb(\uli \ell \times \RR^m, M)
\]
The source and target maps to the object space are composition and projection, respectively.

As explained in the previous section, the projection
functor $\smallint E\to \disc$ induces a map of nerves
\[ N\smallint E \lra N\disc \]
which is a fiberwise complete Segal space over $N\disc$. Taking path components gives a functor $\pi_0$ from $\disc$ to $\fin$
and we compose $N\smallint E \lra N\disc \lra  N\fin$. The composition
\[ N\smallint E \lra N\fin \]
is typically not a fiberwise complete Segal space over $N\fin$.

\begin{defn}
$\config(M)$ is the fiberwise completion of $N\smallint E$ over $N \fin$.
\end{defn}

For ease of notation, write $X:=N\smallint E$. It is easy to give an explicit description of fiberwise Rezk completion (over $N\fin$) applied to the simplicial space $X$.
Clearly $X_r$ is a coproduct of spaces $X(k_0, \dots, k_r)$ defined by the formula
\begin{equation} \label{eqn-stringy}
\emb(\uli k_r \times \RR^m, \uli k_{r-1} \times \RR^m) \times \cdots \times \emb(\uli k_1 \times \RR^m, \uli k_{0} \times \RR^m) \times \emb(\uli k_{0} \times \RR^m, M)
\end{equation}
The Lie group
$\Or(m)^{\sum k_i}$
acts on the space~(\ref{eqn-stringy}) as follows. Given $(g_0, \dots, g_r) \in \prod_{i = 0}^r O(m)^{k_i}$ and $(f_r, \dots, f_0) \in X(k_0, \dots, k_r)$, the action product is
\[
(g_{r-1} f_r g_r^{-1}, \dots, g_0 f_1 g_{1}^{-1}, f_0 g_0^{-1}) \in X(k_0, \dots, k_r) \; .
\]
We form the homotopy orbit space of this action. Let $Y_r$ be the coproduct of these homotopy orbit spaces.
Then $Y=(Y_r)$ is a simplicial space over $N\fin$ which satisfies ($\sigma$) and ($\kappa_{f}$). The inclusion $X\to Y$ is a Dwyer-Kan equivalence. Therefore $Y$ so defined is a fiberwise Rezk completion of $X$ over the nerve of $\fin$ -- we refer to $Y$ as the \emph{multipatch model} of the configuration category of $M$.

\medskip
Based on the explicit description $Y$ of $\config(M)$, we can make some observations on objects and morphisms,
that is, the spaces $Y_0$ and $Y_1$. It is clear that
\begin{equation} \label{eqn-multipatchob} Y_0~\simeq\coprod_{k\ge 0} \emb(\uli k\,,M),
\end{equation}
that is to say, $Y_0$ has the homotopy type of the disjoint union of the
ordered configuration spaces of $M$ for the various finite cardinalities. Now fix a point in $Y_0$
represented by some embedding $f\co \uli k\times\RR^m\to M$.
The (homotopy type of the) homotopy fiber $\Phi_f$ of the map \emph{target} (alias $d_1$) from $Y_1$ to $Y_0$
over that point \emph{depends only on the integers $k$ and $m$}, not on $M$. Indeed we have
\begin{equation} \label{eqn-multipatchmor}
\Phi_f~~\simeq~~ \coprod_{i \ge 0} \, \emb(\uli i, \uli k \times \RR^m) \cong \coprod_{i\ge 0}~\coprod_{g\co \uli i\to\uli k}~\prod_{x=1}^k \emb(g^{-1}(x),\RR^m)~.
\end{equation}
There is a more coordinate-free way to write this. Instead of expressing $\Phi_f$ in terms of the source
of $f$, we could express it in terms of $\im(f)$ which is a standard tubular neighborhood of an ordered
configuration of $k$ points in $M$.

\bigskip
\emph{The particle model.} See also \cite{Andrade}, \cite{Woolf}, \cite{Treumann}, \cite{Lurie} and \cite{AFT1}\footnote{The last two references use a similar but non-equivalent definition.}. Here we assume only that $M$ is a topological manifold.

Let $k\in\NN$. The space of maps from $\uli k$ to $M$
comes with an obvious stratification. There is one stratum for each equivalence relation $\eta$ on $\uli k$~. The
points of that stratum are precisely the maps $\uli k\to M$ which can be factorized as projection from $\uli k$ to
$\uli k\,/\eta$ followed by an injection of $\uli k\,/\eta$ into $M$.

Now we construct a topological category whose object space is
\begin{equation} \label{eqn-appfact1}  X_0:=\coprod_{k\ge 0} \emb(\uli k\,,M), \end{equation}
that is, the topological disjoint union of the ordered configuration spaces of $M$ for each cardinality $k\ge 0$.
By a morphism from $f\in \emb(\uli k\,,M)$ to $g\in\emb(\uli\ell\,,M)$ we mean a pair consisting of a map $v\co \uli k\to \uli\ell$
and a (reverse) exit path $\gamma$ from $f$ to $gv$ in the stratified space of \emph{all} maps from $\uli k$ to $M$. See section~\ref{sec-topcat},
example (v). The space
of all morphisms is therefore a coproduct
\begin{equation} \label{eqn-appfact2} X_1:=\coprod_{\twosub{k,\ell\ge 0}{v\co\uli k\to\uli\ell}} X_1(v) \end{equation}
where $X_1(v)$ consists of triples $(f,g,\gamma)$ as above: $f\in\emb(\uli k\,,M)$,
$g\in\emb(\uli\ell\,,M)$ and $\gamma$ is a (reverse) exit path from $f$ to $gv$. Composition of morphisms is obvious.
The nerve of this category, which we denote by $X$\,, is a fiberwise complete Segal space
over $N\fin$. No further completion is necessary.

For clarification we mention a technical result on stratified spaces (which we apply here
to the stratified space of all maps from $\uli k$ to $M$). Let $Z$ be a stratified
space, $Z_j$ a stratum of $Z$ and $z$ a point in the closure of $Z_j$. Let $\wp(z,Z,Z_j)$ be the space of exit paths in $Z$ starting at $z$ and ending somewhere in $Z_j$. Let $\holink(z,Z,Z_j)$ be the subspace of $\wp(z,Z,Z_j)$ consisting of those exit paths $\gamma\co[0,a]\to Z$ having $\gamma(0)=z$ and $\gamma(t)\in Z_j$ for all other $t\in[0,a]$.

\begin{prop} \label{prop-miller} \cite[Theorem 4.9]{Miller} If $Z$ is a homotopically stratified space in the sense of Quinn, then the inclusion of $\holink(z,Z,Z_j)$ in $\wp(z,Z,Z_j)$ is a homotopy equivalence.
\end{prop}

\smallskip We use this in the case $Z=\map(\uli k\,,M)$ which is a homotopically stratified space in the sense of Quinn \cite{Quinn}. This is so because the condition of being \emph{homotopically stratified} can be verified \emph{locally} \cite[Lemma 2.5]{Quinn}. In the smooth case, this is enough; in the topological case, use that a topological manifold can be smoothed \emph{locally}.

The remarkable thing in Proposition \ref{prop-miller} is that $\holink(z,Z,Z_j)$ is a local concept. That is,
for any open neighborhood $U$ of $z$ in $Z$ the inclusion
\begin{equation}\label{eqn-strat}
\holink(z,U,Z_j\cap U) \to \holink(z,Z,Z_j)
\end{equation}
is a homotopy equivalence. To see this note that $\holink(z,Z,Z_j)$ is metrizable, therefore paracompact. Then use partitions of unity to construct a function $\varphi$ from $\holink(z,Z,Z_j)$ to $(0,1]$ such that for every $\gamma : [0,a] \to Z$ in $\holink(z,Z,Z_j)$, its restriction to $[0, \varphi(\gamma) \cdot a)$ is in $\holink(z,U,Z_j \cap U)$. The corresponding locality statement is not easy to prove for $\wp(z,Z,Z_j)$ directly, without appealing to Proposition~\ref{prop-miller}.

\bigskip
\emph{The Riemannian model.} In this description we assume that $M$ is smooth and equipped with a
Riemannian metric. Define a topological category with objects the pairs
$(S,\rho)$ where $S$ is a finite subset of $M$ with a total ordering, and $\rho$
is a function from $S$ to the set of positive real numbers, subject to two conditions.
\begin{itemize}
\item For each $s\in S$, the exponential map $\exp_s$ at $s$ is defined and regular on the open ball
$B_{\rho(s)}(0_s)$ of radius $\rho(s)$ about the origin $0_s$ in $T_sM$;
\item The images $\exp_s(B_{\rho(s)}(0_s))$ for $s\in S$ are disjoint. We write $U_\rho(S)$ for their union, an
open subset of $M$ which is diffeomorphic to $S\times\RR^m$.
\end{itemize}
Given objects $(S,\rho)$ and $(R,\omega)$, we declare there to be exactly one morphism from $(S,\rho)$ to $(R,\omega)$ if $U_\rho(S)\subset U_\omega(R)$. (Beware that the category so defined is not a
topological poset because there are isomorphisms which are not identity morphisms. Namely,
objects $(S,\rho)$ and $(R,\rho)$ are isomorphic, but not equal, if $S$ and $R$ agree as subsets of $M$ but their total
orderings differ.) Let $X$ be the nerve of that category, a simplicial space.
It is easy to verify that $X$ is a fiberwise complete Segal space over $N\fin$.
We will prove later in this section that it is another incarnation of $\config(M)$.

\bigskip
\emph{The screen completion model of Fulton-MacPherson, Axelrod-Singer and Sinha.}
This is a beautiful edition of $\config(M)$ with many
fascinating differential-geometric features. Since we do not need all of that, we only give a superficial description which
emphasizes the categorical aspects. For more details, see \cite{Sinha} and \cite[\S5]{AxSing94}. In the words of
Axelrod and Singer, their construction was made by taking the definitions
in the algebro-geometric context of \cite{FultonMacPherson} and replacing algebraic-geometric blowups
with differential geometric blowups.

Let $M$ be an $m$-dimensional Riemannian manifold without boundary. Axelrod and Singer \cite{AxSing94} construct a \emph{properification} of the inclusion
\[   \emb(\uli k,M) \lra \map(\uli k, M), \]
that is to say a factorization
\[ \emb(\uli k\,,M) \to M[k] \to \map(\uli k\,, M) \]
where the second arrow is a proper map, $M[k]$
is a smooth manifold with boundary and multiple corners, and the first arrow identifies
$\emb(\uli k\,,M)$ with $M[k]\smin\partial(M[k])$. (In \cite{AxSing94},
the manifold $M$ is assumed to be closed and 3-dimensional. One of the chief merits of \cite{Sinha}
is to clarify that it can be an arbitrary Riemannian manifold. Sinha frequently describes the inclusion
$\emb(\uli k\,,M)\to M[k]$ as a compactification,
but he does point out that $M[k]$ for $k>0$ is compact only if $M$ is compact.) The manifold $M[k]$ is stratified with manifold strata $M(k,\Phi)$, where $\Phi$ is a \emph{nest} of subsets of $\uli k$~. More precisely,
$\Phi$ is a set of subsets of $\uli k$ such that
\begin{itemize}
\item every $S\in\Phi$ has cardinality $>1$~;
\item $S,T\in\Phi$ implies that $S\cap T=\emptyset$ or $S\subset T$ or $T\subset S$.
\end{itemize}
The view taken here is that distinct strata in a stratification must be disjoint; in particular $M[k]$ as a set
is the \emph{disjoint} union of the subsets $M(k,\Phi)$.
Let $M[k,\Phi]$ be the closure of the stratum $M(k,\Phi)$
in $M[k]$. It is the union of the strata $M(k,\Phi^\prime)$ where $\Phi\supset\Phi^\prime$. 
It is also a smooth manifold with boundary
(and a complicated corner structure) whose interior is precisely $M(k,\Phi)$. \emph{Examples}: For $\Phi=\emptyset$, the minimal choice,
we have $M(k,\Phi)=\emb(\uli k\,,M)$ and $M[k,\Phi]=M[k]$. For $k=2$ and $\Phi=\{\uli 2\}$, the stratum
$M(2,\Phi)$ is diffeomorphic to the total space of the unit tangent bundle of $M$.

There is the following description of the stratum $M(k,\Phi)$ in general. For $S\in\Phi$ let $Q_{S,\Phi}$ be the
quotient of the set $S$ by the smallest equivalence relation such that every $T\in\Phi$ which is a proper subset of $S$
is contained in an equivalence class. Let $Q_{\Phi}$ be the quotient of $\uli k$ by the smallest
equivalence relation such that every $T\in\Phi$ is contained in an equivalence class. Then $M(k,\Phi)$ is
the total space of a bundle on the space $\emb(Q_\Phi,M)$
such that the fiber over $f\co Q_\Phi\hookrightarrow M$ is
\[  \prod_{S\in\Phi} \Big(\emb(Q_{S,\Phi}\,,T_{f(S)}M)/\sim\Big) \]
Here $f(S)$ is short for the value of $f$ on the element of $Q_\Phi$ represented by any element of $S$, and
$T_{f(S)}M$ is the tangent space of $M$ at $f(S)$. The equivalence relation indicated by $\sim$
is the orbit relation under an obvious action of the group of self-maps of $T_{f(S)}M$ of the form $v\mapsto av+w$,
where $a$ is a positive real number and $w\in T_{f(S)}M$.
That action is free and so passage to orbits reduces dimension by $m+1$. ---
An element of $\emb(Q_{S,\Phi}\,,T_{f(S)}M)/\sim$ as in the above description of $M(k,\Phi)$
is called a \emph{screen} in \cite{AxSing94}, following \cite{FultonMacPherson}
where the same word is used for the analogous concept in the complex setting. Sinha \cite{Sinha}
has developed a rather nice way to represent elements of $M(k,\Phi)$ by pictures.

The spaces $M[k]$ and their strata depend functorially on $k$ in two ways. An \emph{injective} map $g\co\uli k\to\uli\ell$
induces contravariantly a map
\[ g^*\co M[\ell]\to M[k] \]
which extends the map $\emb(\uli \ell\,,M)\to \emb(\uli k\,,M)$ given by pre-composition with $g$. This map $g^*$ takes the
stratum $M(\ell,\Phi)$ to $M(k,g^*\Phi)$, in self-explanatory notation. For a \emph{surjective} map $g\co \uli k\to\uli\ell$
let $\Phi_g$ be the nest of subsets of $\uli k$ whose elements are all sets $g^{-1}(j)$ of cardinality at least $2$,
where $j\in\uli\ell$\,. For a nest $\Psi$ of subsets of $\uli k$ which contains $\Phi_g$~, let $g_*\Psi$ be the nest
of subsets of $\uli\ell$ consisting of all subsets of $\uli\ell$ which have cardinality $>1$ and whose preimage
under $g$ belongs to $\Psi$. There is a map
\[  g_*\co M[k,\Phi_g] \lra M[\ell] \]
which maps a stratum $M(k,\Psi)$, where $\Psi\supset\Phi_g$~, to the stratum $M(\ell,g_*\Psi)$.

We use this two-way naturality to construct a topological category with object space equal to $\coprod_{k\ge 0} M[k]$. Let $g\co \uli k\to \uli\ell$ be any map. There is a
unique factorization of $g$ as
\[
\uli k \xrightarrow{ \; p \;} \uli r \xrightarrow{ \; e \; } \uli \ell
\]
where $p$ is surjective while $e$ is injective and order-preserving. For $x\in M[k]$
and $y\in M[\ell]$, we regard the triple $(x,g,y)$ as a morphism from $x$ to $y$ if $x\in M[k,\Phi_p]$ and
$p_*x=e^*y\in M[r]$. Composition is straightforward: For $w\in M[j]$, $x\in M[k]$, $y\in M[\ell]$ and maps
$h\co\uli j\to \uli k$~, $g\co \uli k\to\uli\ell$ such that the triples $(w,h,x)$ and $(x,g,y)$ are morphisms,
the triple $(w,gh,y)$ is also a morphism, by inspection.

\subsection{Equivalence of models} For this section, we denote by $X$ the Grothendieck construction of $\emb(-,M)$
(formerly denoted $\smallint E$) and $X^{m}$ (multipatch), $X^{p}$ (particle), $X^{R}$ (Riemannian) and $X^{s}$ (screen)
the different proposed models for $\config(M)$ from the previous section. By construction, $X^m$ is a model for $\config(M)$.
In this section, we show that $X^p$, $X^R$ and $X^s$ are models for $\config(M)$ too, i.e., they are fiberwise completions
of $X$ over $\fin$.

\emph{Particle model $X^p$.}
We construct a zigzag $X^p \xleftarrow{i} W \xrightarrow{j} X$ of Dwyer-Kan equivalences over $\fin$ (in fact, $j$ will be a degreewise equivalence). Let $W$  be the nerve of the category with the same objects as $X$; the space of morphisms from an embedding $f: \uli k \times \RR^m \to M$ to an embedding $g : \uli \ell \times \RR^m \to M$ is empty if $\im(f)$ is not contained in $\im(g)$ and, otherwise, it is the space of morphisms
in $\config(\im(g))$ from $f_0$ to $g_0$ (where $f_0$ is the composition of $f$ with the inclusion $\uli k\cong \uli k\times 0\to \uli k\times\RR^m$).

The maps $i:W \rightarrow{} X^p$ and $j\co W\to X$ are both forgetful; $i$ restricts embeddings $\uli k \times \RR^m \to M$
to $\uli k\cong \uli k\times 0$, and $j$ forgets the
reverse exit paths. Both $i$ and $j$ are essentially surjective on objects
because they are actually surjective on objects. We have to show that both $i$ and $j$ are fully faithful.

We begin with $i$. Let $W^{\sharp}$ be the fiberwise completion of $W$ over $N\fin$. We have an explicit model for $W^\sharp$ similar to the explicit model for $X^\sharp$, the fiberwise completion of $X$ over $N\fin$. So there is a pullback square
\[
	\begin{tikzpicture}[descr/.style={fill=white}, baseline=(current bounding box.base)] ]
	\matrix(m)[matrix of math nodes, row sep=2.5em, column sep=2.5em,
	text height=1.5ex, text depth=0.25ex]
	{
 	W & W^{\sharp}  \\
	X & X^\sharp \\
	};
	\path[->,font=\scriptsize]
		(m-1-1) edge node [auto] {$$} (m-1-2);
	\path[->,font=\scriptsize]
		(m-2-1) edge node [auto] {$$} (m-2-2);
	\path[->,font=\scriptsize]
		(m-1-1) edge node [left] {$$} (m-2-1);
	\path[->,font=\scriptsize] 		
		(m-1-2) edge node [auto] {$$} (m-2-2);
	\end{tikzpicture}
\]
The map $i$ from $W$ to $X^p$ evidently factors through $W^\sharp$. (Such a factorization exists for abstract reasons since $X^p$ is fiberwise complete, but we also have an explicit description of this factorization.) Now we check that the map from $W^\sharp$ to $X^p$ is an equivalence in degree $0$ and $1$. This implies that $i$ is a Dwyer-Kan equivalence since $W^\sharp$ is by definition Dwyer-Kan equivalent to $W$.  In degree $0$, the situation is that $(X^p)_0$ is the space of all ordered configurations in $M$ whereas $W_0$ is weakly equivalent to the space of all ordered, framed configurations in $M$. Passing from $W_0$ to $(W^\sharp)_0$ we lose the framing. In degree $1$: a point in $W^{\sharp}_1$ consists of an inclusion of a tubular neighborhood $U_x$ of an ordered configuration $x$ in a tubular neighborhood $V_y$ of another ordered configuration $y$, together with a reverse exit path from $x$ to $y$ which proceeds in $V_y$. The map from $(W^\sharp)_1$ to $(X^p)_1$ forgets the tubular neighborhoods, and it is a weak equivalence.

It remains to show that $j : W \to X$ is fully faithful, i.e. that the square
\[
	\begin{tikzpicture}[descr/.style={fill=white}, baseline=(current bounding box.base)] ]
	\matrix(m)[matrix of math nodes, row sep=2.5em, column sep=2.5em,
	text height=1.5ex, text depth=0.25ex]
	{
	W_1 & W_0 \times W_0  \\
	X_1 & X_0 \times X_0 \\
	};
	\path[->,font=\scriptsize]
		(m-1-1) edge node [auto] {$(d_0, d_1)$} (m-1-2);
	\path[->,font=\scriptsize]
		(m-2-1) edge node [auto] {$(d_0, d_1)$} (m-2-2);
	\path[->,font=\scriptsize]
		(m-1-1) edge node [left] {$j_1$} (m-2-1);
	\path[->,font=\scriptsize] 		
		(m-1-2) edge node [auto] {$j_0$} (m-2-2);
	\end{tikzpicture}
\]
is homotopy cartesian. The vertical maps are fibrations (note that $j_0$ is the identity). We have to show that the fibers of $j_1$ are contractible.
Let $\mathcal{W}$ denote the fiber of $j_1$ over $f \to g$ in $X_1$, i.e. the space of morphisms in $\config(\im(g))$ from
$f_0$ to $g_0$. It is easy to see that this is a product of factors indexed by the connected components of
$\im(g)$. That is, we can assume that $g$ has the form $g\co \uli\ell\times\RR^m\to M$ where $\ell=1$.
Then we use the lemma below.

\begin{lem} \label{lem-exicontr}
Let $f$ be a configuration of $k$ points in $\RR^m$. The space of exit paths starting at $f$ and ending at the origin is contractible.
\end{lem}

\begin{proof}
Let us denote the space of these exit paths by $Y$ (which, for convenience, we reparametrize so that their domain is the interval $[0,1]$). Then $Y$ has a distinguished element $\gamma\co[0,1]\to \map(\uli k\,,\RR^m)$ given by $\gamma(t)(j)=t\cdot f(j)$ for $j\in\uli k$ and $t\in[0,1]$, where we use the action of the positive real numbers on $\RR^m$ by scalar multiplication. A homotopy
$(h_s\co Y\to Y)_{s\in[0,1]}$ from the constant map with value $\gamma$ to the identity map of $Y$
can be defined as follows: $h_s(\alpha)$ is given by
\[  (t,j)\mapsto \left\{ \begin{array}{rl} s\cdot \alpha(t/s)(j) & \textup{ if }t< s \\
\gamma(t)(j) & \textup{ if }t\ge s
\end{array} \right.
\]
for $t\in[0,1]$ and $j\in\uli k$. Note that for $s=t>0$ we have $s\cdot \alpha(t/s)(j)=s\cdot \alpha(1)(j)=s\cdot f(j)
=t\cdot f(j)=\gamma(t)(j)$.
\end{proof}

\emph{Riemannian model $X^R$.}
 Let $W$ be the category defined as follows: $\ob(W)=\ob(X^R)$, a morphism from $(S,\rho)$ to $(T,\tau)$ in $W$ exists only if $U_\rho(S)\subset U_\tau(T)$, and in that case it is a
morphism in $X^p$ from the object $\uli k\cong S\hookrightarrow M$ to
the object $\uli\ell \cong T\hookrightarrow M$ whose track is contained in $U_\tau(T)$.
(Here we need to choose order-preserving bijections $\uli k\to S$ and $\uli\ell\to T$.)
There are obvious forgetful functors $W\to X^p$ and $W \to X^{R}$ over $\fin$. Both induce degreewise weak equivalences of the nerves. (In the more difficult case $W\to X^p$, show first that it is a Dwyer-Kan equivalence by reasoning as in the
previous comparison of models; but then note that the underlying map of object spaces is also a weak equivalence.)

\medskip
\emph{Screen completion model $X^s$. } Recall the category $X^s$ of the screen completion model, with
\[ \ob(X^s)=\coprod_{k\ge 0} M[k]~. \]
We need an auxiliary category $\ms A$ containing $X^s$ as a subcategory.
The object space of $\ms A$ agrees with the object space of $X^s$; a morphism from $x\in M[k]$ to
$y\in M[\ell]$ in $\ms A$ is a morphism $(x',g,y)$ in $X^s$ from some $x'\in M[k]$ to $y$,
together with a reverse exit path $\gamma\co [0,a]\to M[k]$ from $x$ to $x'$.
(Since $\gamma$ is a reverse exit path, it proceeds in the closure of the stratum of $M[k]$ which contains $x=\gamma(0)$. This makes it easy to define the composition of morphisms in $\ms A$.) The inclusion $X^s \hookrightarrow \ms A$ induces a degreewise weak equivalence of the nerves, which are fiberwise complete Segal spaces over $N\fin$. There is also an inclusion
$X^p \hookrightarrow \ms A$, as the full
subcategory of $\ms A$ obtained by keeping only the objects in
\[ \coprod_{k\ge 0} \emb(\uli k\,,M)\subset \coprod_{k\ge 0} M[k]. \]
Recall that $\emb(\uli k,M)$, also known as $M[k,\emptyset]$, is the open (top) stratum of $M[k]$.
The inclusion of nerves induced by $X^p \hookrightarrow \ms A$ is also a degreewise weak equivalence
of fiberwise complete Segal spaces over $N\fin$.

\medskip
We end this section with two observations which will be important later on.

\begin{prop} \label{prop-operad} For any open subset $U$ of $M$, the following square is homotopy cartesian:
\[
	\begin{tikzpicture}[descr/.style={fill=white}, baseline=(current bounding box.base)] ]
	\matrix(m)[matrix of math nodes, row sep=2.5em, column sep=2.5em,
	text height=1.5ex, text depth=0.25ex]
	{
	\config(U)_1 & \config(M)_1  \\
	\config(U)_0 & \config(M)_0 \\
	};
	\path[->,font=\scriptsize]
		(m-1-1) edge node [auto] {} (m-1-2);
	\path[->,font=\scriptsize]
		(m-2-1) edge node [auto] {\textup{incl.}} (m-2-2);
	\path[->,font=\scriptsize]
		(m-1-1) edge node [left] {$d_1=\textup{target}$} (m-2-1);
	\path[->,font=\scriptsize] 		
		(m-1-2) edge node [auto] {$d_1$} (m-2-2);
	\end{tikzpicture}
\]
\end{prop}

\begin{proof} Since we have established that all models are equivalent we are free to choose any of them.
If $M$ is smooth, we can use the multipatch model in which case the statement is a direct consequence
of (\ref{eqn-multipatchmor}). If $M$ is not smooth, we use the particle model and identify the map of vertical fibers over a configuration $z : \uli k \to U$ as
\[
\holink(z, \map(\uli k, U), \emb(\uli k, U)) \to \holink(z, \map(\uli k, M), \emb(\uli k, M)) \; .
\]
This map is a homotopy equivalence by the remarks after Proposition \ref{prop-miller}.
\end{proof}

\begin{cor} \label{cor-operad} For any open subset and any $r\ge 0$, the following is a homotopy pullback square:
\[
	\begin{tikzpicture}[descr/.style={fill=white}, baseline=(current bounding box.base)] ]
	\matrix(m)[matrix of math nodes, row sep=2.5em, column sep=2.5em,
	text height=1.5ex, text depth=0.25ex]
	{
	\config(U)_r & \config(M)_r  \\
	\config(U)_0 & \config(M)_0 \\
	};
	\path[->,font=\scriptsize]
		(m-1-1) edge node [auto] {} (m-1-2);
	\path[->,font=\scriptsize]
		(m-2-1) edge node [auto] {} (m-2-2);
	\path[->,font=\scriptsize]
		(m-1-1) edge node [left] {$d_1d_2\dots d_r$} (m-2-1);
	\path[->,font=\scriptsize] 		
		(m-1-2) edge node [auto] {$d_1d_2\dots d_r$} (m-2-2);
	\end{tikzpicture}
\]
where $d_1d_2\dots d_r$ picks the ultimate target.
\end{cor}

\begin{proof} Use the property ($\sigma$), section~\ref{sec-topcat}, to deduce this from proposition~\ref{prop-operad}. \end{proof}

\subsection{The configuration category as a functor} \label{sec-confunc}
Let $u:M\to N$ be an injective continuous map between topological
manifolds.
Then $u$ induces a simplicial map $u_*\co\config(M)\to \config(N)$ over the nerve of $\fin$.
This is immediately clear if we use
the particle model. \newline
Assuming now that $M$ and $N$ are smooth manifolds, we obtain a map
\[  \emb(M,N) \lra \rmap_{\fin}(\config(M),\config(N)) \]
where the source is the space of \emph{smooth} embeddings from $M$ to $N$.
(Slight abuse of notation: we write $\rmap_\fin(-,-)$ to mean $\rmap_{N\fin}(-,-)$, where $N\fin$ is the nerve of $\fin$.)

\medskip
For fixed $k\ge 1$, let $\config(M;k)$ be the simplicial subspace of $\config(M)$ determined by the configurations
of cardinality $\le k$. This is self-explanatory if we use the particle model. In the
multipatch model, we would allow as objects only those (codimension zero) embeddings $\uli\ell\times\RR^m\to M$
where $\ell\le k$. If we use the particle model, it is clear
that an embedding $u\co M\to N$ (in fact any injective continuous map
$u\co M\to N$) induces a simplicial map $u_*\co\config(M;k)\to \config(N;k)$ over the nerve of $\fin$.

Let $\man$ be the category of smooth manifolds without boundary \emph{of fixed dimension } $m$. The morphisms in $\man$ are the smooth codimension zero embeddings; morphism sets are viewed as spaces (with the compact-open $C^\infty$ topology),
as in \cite{BoavidaWeiss}, so that $\man$ is enriched over topological spaces. Then
\[ M\mapsto \config(M;k) \]
is a covariant (enriched) functor on $\man$, for any fixed $k$. Consequently, for fixed smooth $N$, the rule
\begin{equation}  \label{eqn-configfunctor}
M\mapsto \rmap_\fin(\config(M;k),\config(N))
\end{equation}
is a contravariant (enriched) functor on $\man$.

Recall from \cite{BoavidaWeiss} and \cite{WeissEmb} that $\man$ comes with a (basis for a) Grothendieck topology $\ms J_k$
where a nonempty family of morphisms $f_\alpha:V_\alpha\to M$ qualifies as a covering if and only if every subset of $M$ of cardinality $\le k$ is
contained in the image of one of the maps $f_\alpha$. (We do not allow the empty family as a covering. In the cases where $M=\emptyset$ or $k=0$,
the unique map $\emptyset\to M$ constitutes a covering of $M$.)

\begin{thm} \label{thm-configsheaf} For fixed $N$, the contravariant functor {\rm (\ref{eqn-configfunctor})}
is a homotopy sheaf for the Grothendieck topology $\ms J_k$.
\end{thm}

\begin{proof}
We show that the functor which to a manifold $M$ associates $\config(M;k)$ is a homotopy $\ms J_k$-cosheaf. Take a $\ms J_k$-cover $\{U_i \hookrightarrow M ~ |~ {i \in I}\}$. For a finite non-empty subset $S$ of $I$, let $U_S$ denote the intersection of all $U_i$ over $i \in S$. We need to show that the map
$$\hocolimsub{S \subset I} \config(U_S; k)_r \rightarrow \config(M; k)_r$$
is a weak equivalence for each $r \geq 0$. Here we have used the fact that homotopy colimits in the overcategory are calculated in the category of simplicial spaces, and those are calculated degreewise.

The case $r = 0$ is clear since, for each $j \leq k$, the collection of maps
$$\{\emb(\uli j, U_i) \to \emb(\uli j, M) ~ |~ i \in I \}$$
is a cover in the usual sense, and so
\[
\hocolimsub{S \subset I} \emb(\uli j, U_S) \to \emb(\uli j, M)
\]
is a weak equivalence (c.f. the proof of \cite[thm.7.2]{BoavidaWeiss}).

For $r = 1$ proceed as follows. Everything is $k$-truncated so we suppress $k$ from the notation. Because of Proposition \ref{prop-operad}, we know that $\config(U_S)_1$ is the homotopy pullback $$\config(U_S)_0 \times^h_{\config(M)_0} \config(M)_1$$

Then,
\[
\begin{array}{rcl}
\hocolimsub{S \subset I} \config(U_S)_1  & \simeq & (\hocolimsub{S}\config(U_S)_0) \times^h_{\config(M)_0} \config(M)_1  \\
				  & \simeq & \config(M)_0 \times^h_{\config(M)_0} \config(M)_1 \\
				  & \simeq & \config(M)_1\\
\end{array}
\]
The first equivalence is the derived version of the fact that in spaces colimits are stable under base change.
The second equivalence follows from the case $i = 0$. The Segal condition establishes the result for $i > 1$.
\end{proof}

\begin{cor} \label{cor-configsheaf} For fixed $N$ there is a commutative square (commutative up to preferred homotopy)
\[
	\begin{tikzpicture}[descr/.style={fill=white}, baseline=(current bounding box.base)] ]
	\matrix(m)[matrix of math nodes, row sep=2.5em, column sep=2.5em,
	text height=1.5ex, text depth=0.25ex]
	{
	\emb(M,N) & \rmap_\fin(\config(M),\config(N)) \\
	T_k\emb(M,N) & \rmap_\fin(\config(M;k),\config(N)) \\
	};
	\path[->,font=\scriptsize]
		(m-1-1) edge node [auto] {$u\mapsto u_*$} (m-1-2);
	\path[dotted,->,font=\scriptsize]
		(m-2-1) edge node [auto] {} (m-2-2);
	\path[->,font=\scriptsize]
		(m-1-1) edge node [left] {$$} (m-2-1);
	\path[->,font=\scriptsize] 		
		(m-1-2) edge node [auto] {$\textup{restriction}$} (m-2-2);
	\end{tikzpicture}
\]
\end{cor}

\begin{proof} This is a consequence of theorem~\ref{thm-configsheaf} and
the fact that $M\mapsto T_k\emb(M,N)$ can be characterized as the $\ms J_k$
homotopy sheafification of $M\mapsto \emb(M,N)$ (homotopy sheafification is a homotopy left adjoint). \end{proof}

\medskip
There is another, more descriptive, construction of the dotted arrow in the above corollary. We consider the case $k = \infty$;
the finite case is similar. Given a manifold $N$, let $\smallint E_M$ be the Grothendieck construction applied to the presheaf
$\emb(-,M)$ on $\disc$ as in section \ref{sec-conman}. It induces a map
\[
T_{\infty} \emb(M,N) \to \RR \map_{\disc}(\smallint E_M, \smallint E_N)
\]
(which, incidentally, is a weak equivalence) where we omitted the nerve symbol for a reason. Composing further with $\pi_0 : \disc \to \fin$, we get a map
$$
\RR \map_{\disc}(\smallint E_M, \smallint E_N) \to \rmap_\fin((\pi_0)_* \smallint E_M, (\pi_0)_* \smallint E_N)
$$
Here the derived mapping space in the target is taken with respect to the fiberwise complete Segal space model structure (see Appendix \ref{thm-fiberwisecpl}), and is none other than
\[
\rmap_\fin(\config(M),\config(N))
\]
by the discussion of the multipatch model in section \ref{sec-conman}.

\section{Immersions and local configuration categories} \label{sec-loc}
\subsection{Local configuration categories} In order to make a connection with the space of smooth immersions $\imm(M,N)$, we introduce
comma type constructions $\config^\loc(M)$ and $\config^\loc(N)$, the \emph{local configuration
categories} of $M$ and $N$. Loosely speaking, the objects of $\config^\loc(M)$ are the morphisms
in $\config(M)$ whose target is a singleton configuration. In more detail, writing
$X$ for the specific model of $\config(M)$ that we wish to use, we have
\[  X_0 =\coprod_{\ell\ge 0} X_0(\uli\ell) \]
where $X_0(\uli\ell)$ is the fiber of $X_0$ over $\uli\ell\in (N\fin)_0$.
Define $\config^\loc(M)$ as the simplicial space which in degree $r$ is the part of $X_{r+1}$ which under the
operator \emph{zero-th vertex} (=ultimate target) from $X_{r+1}$ to $X_0$ is mapped to $X_0(\uli 1)$. Note that there is projection map from $\config^\loc(M)$ to $M$ (viewed as constant simplicial space) given by the ultimate target map.

\medskip
\emph{Example}: in the particle model, $\config(M)=X$ is the
nerve of a certain category whose objects are the ordered configurations in $M$. Consequently $\config^\loc(M)$ is the nerve of the
comma type construction whose objects are morphisms $\gamma\co f\to g$ in that category of ordered configurations,
with injective $f\co \uli k\to M$, where $k\ge 0$ is arbitrary, and $g\co\uli 1\to M$. \newline
\emph{Example}: in the multipatch model, $\config(M)=X$ is the nerve of a category whose objects are
certain pairs $(S,\rho)$. So $\config^\loc(M)$ is the nerve of the comma type construction whose objects are
pairs $((T,\tau),(S,\rho))$ where $U_\rho(S)$ is contained in $U_\tau(T)$ and $T$ has cardinality 1.

\medskip
There is a map of simplicial spaces $\config^\loc(M) \lra \config(M)$
which can be regarded as a forgetful functor. In the particle model, this is the map
of nerves induced by an obvious forgetful functor which takes an object $\gamma\co f\to g$ in $\config^\loc(M)_0$,
where $f:\uli k\to M$ and $g\co\uli 1\to M$, to the object $(f\co\uli k\to M)\in \config(M)_0$.

\begin{lem}\label{lem-fiberconf}
For an open subset $U$ of $M$, the square of simplicial spaces
\[
	\begin{tikzpicture}[descr/.style={fill=white}, baseline=(current bounding box.base)] ]
	\matrix(m)[matrix of math nodes, row sep=2.5em, column sep=2.5em,
	text height=1.5ex, text depth=0.25ex]
	{
	 \config^\loc(U) & \config^\loc(M) \\
	U & M \\
	};
	\path[->,font=\scriptsize]
		(m-1-1) edge node [auto] {} (m-1-2);
	\path[->,font=\scriptsize]
		(m-2-1) edge node [auto] {} (m-2-2);
	\path[->,font=\scriptsize]
		(m-1-1) edge node [left] {} (m-2-1);
	\path[->,font=\scriptsize] 		
		(m-1-2) edge node [auto] {} (m-2-2);
	\end{tikzpicture}
\]
is (degreewise) homotopy cartesian.
\end{lem}
\begin{proof}
This follows from Corollary \ref{cor-operad}.
\end{proof}

\medskip
By restricting cardinalities, we have for every $k\ge 0$ a truncated local configuration category $\config^\loc(M;k)$,
which is again a fiberwise complete Segal space over $N\fin$. In the spirit and notation of section~\ref{sec-confunc},
we wish to fix a smooth manifold $N$ of arbitrary dimension and view
\begin{equation} \label{eqn-locoalconfigfunctor} M ~\mapsto~ \rmap_\fin(\config^\loc(M;k),\config^\loc(N)) \end{equation}
as a contravariant functor on $\man$~. The following makes
an interesting contrast with theorem~\ref{thm-configsheaf} and justifies our use of the word local (in \emph{local
configuration category}).

\begin{prop} \label{prop-localconfigsheaf} For fixed $N$, the contravariant functor {\rm (\ref{eqn-locoalconfigfunctor})}
is a homotopy sheaf for the Grothendieck topology $\ms J_1$~.
\end{prop}
\begin{proof}
Let $\{U_i \to M\}$ be a good $\ms J_1$-cover. We need to prove that the map
\begin{equation}\label{eqn-cosheafloc}
\hocolimsub{S \subset I} \conloc{U_S; k}_r \to \conloc{M; k}_r
\end{equation}
is a weak equivalence for each $r \ge 0$. The argument is independent of $k$, so we suppress $k$ from the notation. In the commutative diagram,
\begin{equation*}
	\begin{tikzpicture}[descr/.style={fill=white}, baseline=(current bounding box.base)] ]
	\matrix(m)[matrix of math nodes, row sep=2.5em, column sep=2.5em,
	text height=1.5ex, text depth=0.25ex]
	{
 	\conloc{U_S}_r & \conloc{M}_r \\
	\config(U_S)_{r+1} & \config(M)_{r+1} \\
	\config(U_S)_0 & \config(M)_0 \\
	};
	\path[->,font=\scriptsize]
		(m-1-1) edge node [auto] {$$} (m-1-2);
	\path[->,font=\scriptsize]
		(m-2-1) edge node [auto] {$$} (m-2-2);
	\path[right hook->,font=\scriptsize]
		(m-1-1) edge node [left] {$$} (m-2-1);
	\path[right hook->,font=\scriptsize]	
		(m-1-2) edge node [auto] {$$} (m-2-2);
	\path[->,font=\scriptsize]
		(m-2-1) edge node [left] {\textup{ultimate target}} (m-3-1);
	\path[->,font=\scriptsize]
		(m-3-1) edge node [auto] {$$} (m-3-2);
	\path[->,font=\scriptsize]
		(m-2-2) edge node [auto] {\textup{ultimate target}} (m-3-2);
\end{tikzpicture}
\end{equation*}
the bottom square is homotopy cartesian by Corollary \ref{cor-operad} and by inspection so is the top square. Since homotopy colimits are stable under homotopy base change it follows that the source of (\ref{eqn-cosheafloc}) maps by a weak equivalence to the homotopy pullback of
\[
\hocolimsub{S \subset I} \config(U_S)_0 \rightarrow \config(M)_0 \leftarrow \conloc{M}_r\,.
\]
By construction the right-hand arrow lands in the cardinality one part of $\config(M)_0$, which we identify with $M$ itself.
So that homotopy pullback is also the homotopy pullback of
\[
\hocolimsub{S \subset I} U_S \xrightarrow{\simeq} M \leftarrow \conloc{M}_r\,,
\]
which completes the proof.
\end{proof}

\begin{lem} \label{lem-locisloc} The map
\[ \rmap_\fin(\config^\loc(M;k),\config^\loc(N)) \lra \rmap_\fin(\config^\loc(M;k),\config(N)) \]
given by composition with the forgetful functor $\config^\loc(N)\to \config(N)$ is a weak equivalence.
\end{lem}

\begin{proof} We write $L(-)$ instead of $(-)^\loc$ and think of that as an endofunctor on a category
of certain fiberwise complete Segal
spaces $X$ over the nerve of $\fin$.
The condition on $X$ is that all elements of $X_1$ which cover $\id\co \uli 1\to\uli 1$
in $\fin$ are homotopy invertible, i.e., belong to $X_1^{\heq}$. Define $L(X)$
as the simplicial space which in degree $r$ is the part of $X_{r+1}$ which under the
operator \emph{zero-th vertex} (=ultimate target) from $X_{r+1}$ to $X_0$ is mapped to $X_0(\uli 1)$, the portion of
$X_0$ taken to $\uli 1$ by the reference functor. There is an obvious forgetful map $u_X\co L(X) \to X$,
which we could also describe as $d_0$\,.
The key properties are that
$u_{L(X)}\co L(L(X))\to L(X)$ is always a weak equivalence,
and furthermore, $L(u_X)\co L(L(X))\to L(X)$ is always a weak equivalence. ---
If $X$ and $Y$ both satisfy the condition, so that $L(X)$ and $L(Y)$ are defined, then there is a commutative diagram
\[
	\begin{tikzpicture}[descr/.style={fill=white}, baseline=(current bounding box.base)] ]
	\matrix(m)[matrix of math nodes, row sep=2.5em, column sep=2.5em,
	text height=1.5ex, text depth=0.25ex]
	{
	\rmap_\fin(L(X),L(Y)) & \rmap_\fin(L(X),Y) \\
	\rmap_\fin(L(L(X)),L(L(Y))) & \rmap_\fin(L(L(X)),L(Y)) \\
	};
	\path[->,font=\scriptsize]
		(m-1-1) edge node [auto] {$$} (m-1-2);
	\path[->,font=\scriptsize]
		(m-2-1) edge node [auto] {$\simeq$} (m-2-2);
	\path[->,font=\scriptsize]
		(m-1-1) edge node [left] {$\simeq$} (m-2-1);
	\path[->,font=\scriptsize] 		
		(m-1-2) edge node [auto] {$p$} (m-2-2);
	\end{tikzpicture}
\]
where the vertical arrows are given by applying $L$, while the horizontal ones are given
by composition with $u_Y\co L(Y)\to Y$ and $L(u_Y)\co L(L(Y))\to L(Y)$, respectively.
For the arrow labeled $p$ we have another commutative diagram
\[
	\begin{tikzpicture}[descr/.style={fill=white}, baseline=(current bounding box.base)] ]
	\matrix(m)[matrix of math nodes, row sep=2.5em, column sep=2.5em,
	text height=1.5ex, text depth=0.25ex]
	{
	\rmap_\fin(L(X),Y) & \rmap_\fin(L(X),Y) \\
	\rmap_\fin(L(L(X)),L(Y)) & \rmap_\fin(L(L(X)),Y) \\
	};
	\path[->,font=\scriptsize]
		(m-1-1) edge node [auto] {$=$} (m-1-2);
	\path[->,font=\scriptsize]
		(m-2-1) edge node [auto] {} (m-2-2);
	\path[->,font=\scriptsize]
		(m-1-1) edge node [left] {$p$} (m-2-1);
	\path[->,font=\scriptsize] 		
		(m-1-2) edge node [auto] {$\simeq$} (m-2-2);
	\end{tikzpicture}
\]
where the lower horizontal arrow is given by composition with $u_Y$ and the
right-hand vertical arrow by composition with $u_{L(X)}$~. This means that $p$ has both a homotopy
right inverse and a homotopy left inverse. Therefore $p$ is a weak equivalence and it follows that
\[ \rmap_\fin(L(X),L(Y)) \lra \rmap_\fin(L(X),Y) \]
given by composition with $u_Y$ is also a weak equivalence.
\end{proof}

\section{Calculus and configuration categories}\label{sec-mainthm}
\subsection{A homotopy pullback square} The following is one of our main results.

Consider the commutative square
\[
	\begin{tikzpicture}[descr/.style={fill=white}, baseline=(current bounding box.base)] ]
	\matrix(m)[matrix of math nodes, row sep=2.5em, column sep=2.5em,
	text height=1.5ex, text depth=0.25ex]
	{
	T_k\emb(M,N) & \rmap_\fin(\config(M;k),\config(N)) \\
	T_1\emb(M,N) & \rmap_\fin(\config^\loc(M;k),\config^\loc(N)) \\
	};
	\path[->,font=\scriptsize]
		(m-1-1) edge node [auto] {} (m-1-2);
	\path[->,font=\scriptsize]
		(m-2-1) edge node [auto] {} (m-2-2);
	\path[->,font=\scriptsize]
		(m-1-1) edge node [left] {} (m-2-1);
	\path[->,font=\scriptsize] 		
		(m-1-2) edge node [auto] {\textup{specialization}} (m-2-2);
	\end{tikzpicture}
\]
where the top arrow is from corollary~\ref{cor-configsheaf}; horizontal arrows due to theorem~\ref{thm-configsheaf} and proposition~\ref{prop-localconfigsheaf}, respectively; and the right-hand arrow is obtained by noting that a functor from $\config(M;k)$ to $\config(N)$ over $N\fin$ induces a functor from $\config^\loc(M;k)$ to $\config^\loc(N)$ over $N\fin$.

\begin{thm} \label{thm-main} This square is homotopy cartesian for any $k$ such that $1 \leq k \leq \infty$.
\end{thm}

\begin{proof}
We can assume that $k$ is finite (the case $k = \infty$ follows by applying $\holim_k$ to the square).
By \cite{BoavidaWeiss} and theorem~\ref{thm-configsheaf}, each term in the commutative square can be
described as a space of derived natural transformations between certain functors on the full subcategory of $\man$ spanned
by the objects $\uli \ell\times\RR^m$ for $\ell \le k$. To put it somewhat differently, each term in the square, viewed
as a contravariant enriched functor of the variable $M$ varying in $\man$~, is the homotopy right Kan extension (in an enriched sense)
of its restriction to the full subcategory of $\man$ spanned
by the objects $\uli \ell\times\RR^m$ for $\ell \le k$. Therefore it is enough to verify the claim when $M$ has the form
$\uli \ell\times\RR^m$ for some $\ell \le k$. We now assume this, fixing $\ell$ until further notice,
but continue to use the label $M$ where it is convenient.

Let $\imm'(M,N)\subset \imm(M,N)$ be the subspace consisting of those immersions which restrict to embeddings on each connected component
of $M$. The inclusion of $\imm'(M,N)$ in $\imm(M,N)\simeq T_1\emb(M,N)$ is a homotopy equivalence.
Therefore our commutative square simplifies to
\begin{equation}
	\begin{tikzpicture}[descr/.style={fill=white}, baseline=(current bounding box.base)] ]
	\matrix(m)[matrix of math nodes, row sep=2.5em, column sep=2.5em,
	text height=1.5ex, text depth=0.25ex]
	{
	\emb(M,N) & \rmap_\fin(\config(M;k),\config(N)) \\
	\imm'(M,N) & \rmap_\fin(\config^\loc(M;k),\config^\loc(N)) \\
	};
	\path[->,font=\scriptsize]
		(m-1-1) edge node [auto] {} (m-1-2);
	\path[->,font=\scriptsize]
		(m-2-1) edge node [auto] {} (m-2-2);
	\path[->,font=\scriptsize]
		(m-1-1) edge node [left] {} (m-2-1);
	\path[->,font=\scriptsize] 		
		(m-1-2) edge node [auto] {\textup{specialization}} (m-2-2);
	\end{tikzpicture}
\end{equation}
This fits into a larger commutative diagram
\begin{equation} \label{eqn-proofmain2}
	\begin{tikzpicture}[descr/.style={fill=white}, baseline=(current bounding box.base)] ]
	\matrix(m)[matrix of math nodes, row sep=2.5em, column sep=2.5em,
	text height=1.5ex, text depth=0.25ex]
	{
	\emb(M,N) & \rmap_\fin(\config(M;k),\config(N)) &  \emb(\uli\ell,N) \\
	\imm'(M,N) & \rmap_\fin(\config^\loc(M;k),\config^\loc(N)) & \map(\uli\ell,N) \\
	};
	\path[->,font=\scriptsize]
		(m-1-1) edge node [auto] {} (m-1-2);
	\path[->,font=\scriptsize]
		(m-2-1) edge node [auto] {} (m-2-2);
	\path[->,font=\scriptsize]
		(m-1-1) edge node [left] {} (m-2-1);
	\path[->,font=\scriptsize] 		
		(m-1-2) edge node [auto] {} (m-2-2);
	\path[->,font=\scriptsize] 		
		(m-1-2) edge node [auto] {} (m-1-3);
	\path[->,font=\scriptsize] 		
		(m-1-3) edge node [auto] {\textup{incl.}} (m-2-3);
	\path[->,font=\scriptsize] 		
		(m-2-2) edge node [auto] {} (m-2-3);
	\end{tikzpicture}
\end{equation}
The right-hand horizontal arrows are given by precomposition along the maps $\Delta[0] \to \config(M)$ and $\Delta[0] \times \uli \ell \to \config^\loc(M)$ which select the objects corresponding (in the multipatch model) to $M$ itself,
respectively, to connected components of $M$.

Since the outer square determined by diagram~(\ref{eqn-proofmain2})
is clearly homotopy cartesian, it is enough to show that the right-hand square in~(\ref{eqn-proofmain2})
is also homotopy cartesian. In that square, the lower left-hand term maps forgetfully and by a homotopy equivalence to
\[  \prod_{j=1}^\ell \rmap_\fin(\config(\RR^m;k),\config(N)). \]
Therefore it is also enough to show that
\begin{equation} \label{eqn-proofmain3}
	\begin{tikzpicture}[descr/.style={fill=white}, baseline=(current bounding box.base)] ]
	\matrix(m)[matrix of math nodes, row sep=2.5em, column sep=2.5em,
	text height=1.5ex, text depth=0.25ex]
	{
	\rmap_\fin(\config(M;k),\config(N)) & \emb(\uli\ell,N) \\
	\prod_{j=1}^\ell \rmap_\fin(\config(\RR^m;k),\config(N)) & \map(\uli\ell,N) \\
	};
	\path[->,font=\scriptsize]
		(m-1-1) edge node [auto] {} (m-1-2);
	\path[->,font=\scriptsize]
		(m-2-1) edge node [auto] {} (m-2-2);
	\path[->,font=\scriptsize]
		(m-1-1) edge node [left] {} (m-2-1);
	\path[->,font=\scriptsize] 		
		(m-1-2) edge node [auto] {} (m-2-2);
	\end{tikzpicture}
\end{equation}
is homotopy cartesian. Here the left-hand vertical arrow is induced by the inclusions
\[ \{j\}\times\RR^m\to \uli\ell\times\RR^m=M. \] 
In order to check this we need to choose a point $f$ in $\emb(\uli\ell,N)$ and compare the horizontal homotopy
fibers over $f$ in diagram~(\ref{eqn-proofmain3}), say $\Phi_1$ for the top row and $\Phi_2$ for the bottom row.

To investigate $\Phi_2$, choose a tubular neighborhood $U$ of the finite set $f(\ell)\subset N$;
more specifically let $U_j\subset U$ be the component of $U$ which contains $j\in\uli\ell$\,. Now
\[
	\begin{tikzpicture}[descr/.style={fill=white}, baseline=(current bounding box.base)] ]
	\matrix(m)[matrix of math nodes, row sep=2.5em, column sep=2.5em,
	text height=1.5ex, text depth=0.25ex]
	{
	\prod_{j=1}^\ell \rmap_\fin(\config(\RR^m;k),\config(U_j)) &  \\
	\prod_{j=1}^\ell \rmap_\fin(\config(\RR^m;k),\config(N)) & \map(\uli\ell,N) \\
	};
	\path[->,font=\scriptsize]
		(m-2-1) edge node [auto] {} (m-2-2);
	\path[->,font=\scriptsize]
		(m-1-1) edge node [left] {} (m-2-1);
	\end{tikzpicture}
\]
is a homotopy fiber sequence (where we take $f\in \map(\uli\ell,N)$ as the base point). For the proof of this observation, we can assume $\ell = 1$ as the general case follows easily. We then replace $\config(\RR^m;k)$ with $\config^\loc(\RR^m;k)$ since $\config(\RR^m; k)$ has a weakly terminal object and replace $\config$ by $\config^\loc$ throughout using Lemma \ref{lem-locisloc}. Then we apply Lemma \ref{lem-fiberconf}.

Therefore we may write
\[  \Phi_2=\prod_{j=1}^\ell \rmap_\fin(\config(\RR^m;k),\config(U_j))~. \]
Note that $\config(M; k)$ also has a weakly terminal object (which lives over $\uli \ell \in (N\fin)_0$). Using that we can get a description of $\Phi_1$ which makes it clear that the canonical map of horizontal homotopy fibers $v\co \Phi_1\to\Phi_2$ in (\ref{eqn-proofmain3}) is a weak equivalence.
\end{proof}

\bigskip
The handle dimension $\hodim(M)$ of a smooth manifold $M$ is the
minimum, taken over the handlebody decompositions of $M$, of the maximal index of handles in the decomposition.
See remark~\ref{rem-handle} for details.

\begin{cor} \label{cor-main} The following commutative square is homotopy cartesian if $\dim(N)$ exceeds $\hodim(M)$
by at least $3$:
\[
	\begin{tikzpicture}[descr/.style={fill=white}, baseline=(current bounding box.base)] ]
	\matrix(m)[matrix of math nodes, row sep=2.5em, column sep=2.5em,
	text height=1.5ex, text depth=0.25ex]
	{
	\emb(M,N) & \rmap_\fin(\config(M),\config(N)) \\
	\imm(M,N) & \rmap_\fin(\config^\loc(M),\config^\loc(N)) \\
	};
	\path[->,font=\scriptsize]
		(m-1-1) edge node [auto] {} (m-1-2);
	\path[->,font=\scriptsize]
		(m-2-1) edge node [auto] {} (m-2-2);
	\path[->,font=\scriptsize]
		(m-1-1) edge node [left] {} (m-2-1);
	\path[->,font=\scriptsize] 		
		(m-1-2) edge node [auto] {} (m-2-2);
	\end{tikzpicture}
\]
\end{cor}

\begin{proof} This is obtained by taking the homotopy inverse limit over $k$ in theorem~\ref{thm-main}
and combining with the main theorem of the smooth embeddings branch
of manifold calculus \cite{GoWeEmb}. \end{proof}

\begin{rem} \label{rem-handle} {\rm  For a smooth manifold $M$ without boundary and an integer $s$, we say that $\hodim(M)\le s$ if
there exists a sequence of compact codimension zero smooth submanifolds (boundary allowed)
\[ \emptyset =M_0 \subset M_1\subset M_2\subset M_3\subset M_4\subset \dots \]
of $M$ such that $M_{j-1}\subset M_j\smin\partial M_j$ for all $j>0$ and $M_j$ is obtained from $M_{j-1}$ by attaching
finitely many handles, all of index $\le s$. This last condition needs to be made more precise; one way to do so is
to say that there exist neatly embedded smooth disks $A_1,\dots,A_r$ in $M_j\smin M_{j-1}$, all of codimension $\le s$
in $M_j$~, such that $M_j$ minus a standard open tubular neighborhood of $A_1\cup A_2\cup\dots\cup A_r$
is the union of $M_{j-1}$ and a closed collar
attached to the boundary of $M_{j-1}$. \emph{Example}: $M=\RR$ and $M_j=[-j,j\,]$ for $j>0$ shows that $\hodim(\RR)\le 0$,
for in this case $M_j$ is obtained from $M_{j-1}$ by attaching a single handle of index $0$ if $j=1$, and no handle at all (just a collar)
if $j>1$.

There is a similar definition for smooth manifolds $M$ with boundary which will be needed in the next section.
In this case we say that $\hodim(M)\le s$ relative to $\partial M$ if
there exists a
sequence of compact codimension zero smooth submanifolds with boundary
\[ M_0 \subset M_1\subset M_2\subset M_3\subset M_4\subset \dots \]
where $M_0$ is a closed collar on $\partial M$,
such that $M_{j-1}\subset M_j\smin(\partial M_j\smin\partial M)$ for all $j>0$ and $M_j$ is obtained from $M_{j-1}$ by attaching
finitely many handles, all of index $\le s$, to $\partial M_{j-1}\smin\partial M$. \emph{Example}: $M=[0,\infty)$ and
$M_j=[0,j+1]$ shows that $\hodim(M)\le s$ relative to $\partial M$ in this case, for every $s\in\ZZ$.
}
\end{rem}

\section{Boundary conditions}\label{sec-bdry}
\subsection{Configuration category of a manifold with boundary.}
In this section we define the configuration category and the local configuration category of a manifold with boundary. This is mainly
for the purpose of studying spaces of \emph{neat} smooth embeddings $(M,\partial M)\to (N,\partial N)$ where the embedding
of boundaries $\partial M\to \partial N$ is prescribed and fixed.
Among the models proposed above for the boundariless case, the multipatch models
survive the generalization quite well. So we take care of these first.
The particle model also admits a beautiful generalization to the setting with boundary, but that is less
obvious. (Proposition~\ref{prop-miller} is again quite important in making the
comparison with the other models.) Unfortunately we are not aware of a variant of the Fulton-McPherson-et.al.~model for manifolds
with boundary.

\medskip
As in \cite{BoavidaWeiss}, we fix a smooth $(m-1)$-manifold $L$ without boundary and re-define $\man$ so that the
objects are smooth $m$-manifolds $M$ with boundary, together with a diffeomorphism $L\to \partial M$.
The morphisms in $\man$ are the codimension zero smooth embeddings which take boundary to boundary
and respect the identification of the boundaries with $L$. We may write $\emb_\partial(M_0,M_1)$ for the space
of morphisms $M_0\to M_1$ in $\man$. In the same spirit, $\disc$ is now the
full subcategory of $\man$ spanned by the objects
\[   L\times[0,1)~\sqcup~ \uli k\times\RR^m  \]
for $k\ge 0$. Let $\finplus$ be the category whose objects are the pointed sets
$[k]=\{0,1,\dots,k\}$ for $k\ge 0$, with base point $0$, so that a morphism from $[k]$
to $[\ell]$ is a based map. Note the difference between $\uli k=\{1,2,\dots,k\}$ and $[k]=\{0,1,\dots,k\}$.
Indeed there is a functor $\fin \to \finplus$ which takes $\uli k$ to $[k]$ and $f\co \uli k\to \uli \ell$
to the based map $[k]\to[\ell]$ which extends $f$. Morphisms in $\finplus$ in the image of that functor
will be called \emph{proper}. There is an obvious functor $\disc\to \finplus$ which we use for bookkeeping.

\bigskip
\emph{The multipatch model.} For fixed $M$ in $\man$ let
$E=E_M$ be the contravariant functor on $\disc$ defined by $E(U)=\emb_{\partial}(U,M)$,
again with the weak $C^\infty$ topology. Here $\emb_{\partial}$ refers to embeddings which take boundary
to boundary and respect the identification of boundaries with the reference manifold $L$. Let $\smallint E$ be the Grothendieck construction (alias transport category) on $E$.
The projection functor $\smallint E\to \disc$ induces a map of nerves
\[ N\smallint E \lra N\disc \]
which is a fiberwise complete Segal space over $N\disc$~. We compose
\[ N\smallint E \lra N\disc \lra  N\finplus~. \]
The composition $N\smallint E \lra N\finplus$
is typically not a fiberwise complete Segal space over $N\finplus$. But we can apply
fiberwise completion and this yields $\config(M)$.

\medskip
In the multipatch model, it is easy to give an explicit description
of fiberwise Rezk completion (over $N\finplus$) applied to the simplicial space $X:=N\smallint E$.
For an integer $k\ge 0$ we have $U(k):= L\times[0,1)~\sqcup~\uli k\times\RR^m$, an object in $\man$.
Clearly $X_r$ is a coproduct of spaces
\begin{equation} \label{eqn-stringybd}
X(k_0,k_1,\dots,k_r):=
\emb_\partial(U(k_r), U(k_{r-1}))\times \cdots \times \emb_\partial(U(k_{0}), M)
\end{equation}
indexed by $r$-tuples of non-negative integers. The Lie group $\Or(m)^{\sum k_i}$
acts on the space in~(\ref{eqn-stringybd})
via the actions of $\Or(m)^{k_i}$ on $U(k_i)$; compare~(\ref{eqn-stringy}). We form the homotopy orbit space
of this action. Let $Y_r$ be the coproduct of these homotopy orbit spaces. The coproduct is still
indexed by $r$-tuples of non-negative integers. Then $Y$ is a
simplicial space over $N\finplus$ which satisfies ($\sigma$) and ($\kappa_{f}$). The inclusion
$X\to Y$ is a Dwyer-Kan equivalence.
Therefore $Y$ so defined can be identified with the fiberwise Rezk completion of $X$ over
the nerve of $\finplus$~.

\medskip
Based on the explicit description $Y$ of $\config(M)$, we can make some observations on objects and morphisms,
that is, the spaces $Y_0$ and $Y_1$. It is clear that
\begin{equation}  Y_0~\simeq\coprod_{k\ge 0} \emb(\uli k\,,M_-), \end{equation}
where $M_-=M\smin\partial M$.
That is to say, $Y_0$ has the homotopy type of the disjoint union of the
ordered configuration spaces of $M_-$ for the various finite cardinalities. Fix a point in $Y_0$
represented by some embedding $f\co L\times[0,1)\,\sqcup\,\uli k\times\RR^m\to M$.
The (homotopy type of the) homotopy fiber $\Phi_f$ of the map \emph{target} (alias $d_1$) from $Y_1$ to $Y_0$
over that point \emph{depends only on the integers $k$ and $m$}, and on $L\cong\partial M$, but not on $M$ itself.
Indeed we have
\begin{equation} \label{eqn-formulamorbd}
\Phi_f~~\simeq~~\coprod_{r\ge 0}~~\coprod_{g\co [r]\to [k]}~~\prod_{j\in [k]} \emb\big(\uli r\cap g^{-1}(j),V_j\big)
\end{equation}
where $V_j=\RR^m$ if $j>0$ and $V_j=L\times\RR$ if $j=0$\,.
In this expression, $g\co[r]\to [k]$ stands for a morphism in $\finplus$\,.
There is a more coordinate-free way to write this. Instead of expressing $\Phi_f$ in terms of the source
of $f$, we could express it in terms of $\im(f)$ which is a standard tubular neighborhood of $\partial M\cup S$,
where $S$ is an ordered configuration of $k$ points in $M_-$\,.

\bigskip
\emph{The Riemannian model.} In this description we assume that $M$ is smooth and equipped with a
Riemannian metric which is a product metric near the boundary. More precisely we assume that for some fixed $\epsilon>0$,
the set of all $x\in M$ whose geodesic distance from $\partial M$ is less than $\epsilon$ is an open collar
(of width $\epsilon$) on $\partial M$ where the Riemannian metric is a product of the restricted Riemannian
metric on the boundary and the standard metric on $[0,\epsilon)$. Let $X_0$ be the space of pairs
$(S,\rho)$ where $S$ is a finite subset of $M_-$ with a total ordering,
and $\rho$ is a function from $S\cup\{\infty\}$
to the set of positive real numbers, subject to three conditions.
\begin{itemize}
\item $\rho(\infty)\le \epsilon$.
\item For each $s\in S$, the exponential map $\exp_s$ at $s$ is defined and regular on the open ball
$B_{\rho(s)}(0_s)$ of radius $\rho(s)$ about the origin $0_s$ in $T_sM$;
\item The images $\exp_s(B_{\rho(s)}(0_s))$ for $s\in S$ and the collar of width $\rho(\infty)$
on $\partial M$ are pairwise disjoint. We write $U_\rho(S\cup\partial M)$ for their union, an
open subset of $M$ which is diffeomorphic to $L\times[0,1)~\sqcup~S\times\RR^m$.
\end{itemize}
We promote $X_0$ to the space of objects of a topological category as follows.
A morphism from $(S,\rho)$ to $(R,\omega)$ exists if and only if
\[ U_\rho(S\cup\partial M)\subset U_\omega(R\cup\partial M) \]
and in that case it is unique. Let $X$ be the topological nerve of that category, a simplicial space.
It is easy to verify that $X$ is a fiberwise complete Segal space over $N\finplus$.
As such it is another incarnation of $\config(M)$.

\bigskip
\emph{The particle model.}
Here we allow $M$ to be a topological manifold with compact boundary. (Unfortunately we do not have a convincing particle model
if the boundary is noncompact.) Let $k\in\{0,1,2,\dots\}$. The space of maps from $\uli k$ to $M/\partial M$ comes with a stratification.
There is one stratum for each pair $(S,\eta)$ where $S\subset\uli k$ and $\eta$ is an equivalence relation
on $\uli k$ such that $S$ is either empty or an equivalence class of $\eta$. The
points of that stratum are the maps $\uli k\to M/\partial M$
which can be factored as projection from $\uli k$ to $\uli k/\eta$
followed by an injection of $\uli k/\eta$ into $M/\partial M$ which
\begin{itemize}
\item[-] avoids the base point of $M/\partial M$ if $S$ is empty,
\item[-] takes the class $S$ to the base point if $S$ is nonempty.
\end{itemize}
Now we construct a topological category whose object space is
\begin{equation}  X_0:=\coprod_{k\ge 0} \emb(\uli k\,,M_-). \end{equation}
A morphism from $f\in \emb(\uli k\,,M_-)$ to $g\in\emb(\uli\ell\,,M_-)$ is a
pair consisting of a morphism $v\co[k]\to[\ell]$ in $\finplus$
and a Moore path $\gamma=(\gamma_t)_{t\in[0,a]}$ in $\map(\uli k\,,M/\partial M)$ which is an exit path in reverse. It is required to
satisfy $\gamma_0=f$ and $\gamma_a(x)=g(v(x))$ if $v(x)\in\uli\ell$\,, but $\gamma_a(x)=$ base point of $M/\partial M$
if $v(x)=0$. Composition of morphisms is obvious.
The nerve of this category is a fiberwise complete Segal space
over $N\finplus$ which we denote by $X$ for now, and which we can also regard as a definition of $\config(M)$.

Fix a point in $X_0$
given by some embedding $f\co \uli k\to M_-$.
The homotopy type of the homotopy fiber $\Phi_f$ of the map \emph{target} (alias $d_1$) from $X_1$ to $X_0$
over that point depends only on the integers $k$ and $m=\dim(M)$, and on $L\cong\partial M$.
We recover formula~(\ref{eqn-formulamorbd})
as in the multipatch model. This follows easily from proposition~\ref{prop-miller} applied to
the stratified spaces $\map(\uli \ell\,,M/\partial M)$ for $\ell\in \NN$.

\bigskip
We wish to generalize theorems~\ref{thm-configsheaf} and~\ref{thm-main} to the setting with boundary. In the boundariless
setting, we are looking at functors between configuration categories. In the setting with boundary, we will be dealing
with functors between configuration categories satisfying some boundary conditions.
For fixed $N$ of dimension $n$, with a preferred smooth embedding $L\to \partial N$, and for $M$ in $\man$, let
$\emb_\partial(M,N)$
be the space of smooth neat embeddings $M\to N$ which on $\partial M\cong L$ agree with the preferred embedding.
Let $W=L\times[0,1)$ and $Z_M=\emb_\partial(W,M)$, and similarly $Z_N=\emb_\partial(W,N)$. There is a diagram
\begin{equation} \label{eqn-nastybdry}
	\begin{tikzpicture}[descr/.style={fill=white}, baseline=(current bounding box.base)] ]
	\matrix(m)[matrix of math nodes, row sep=2.5em, column sep=2.5em,
	text height=1.5ex, text depth=0.25ex]
	{
	 & \rmap_{\finplus}(\config(M;k),\config(N)) \\
	\map(Z_M,Z_N) & \map(Z_M,\rmap_{\finplus}(\config(W;k),\config(N))) \\
	};
	\path[->,font=\scriptsize]
		(m-2-1) edge node [auto] {} (m-2-2);
	\path[->,font=\scriptsize] 		
		(m-1-2) edge node [auto] {} (m-2-2);
	\end{tikzpicture}
\end{equation}
where the vertical arrow is adjoint to a composition map
\[ Z_M\times\rmap_{\finplus}(\config(M;k),\config(N)) \lra
\rmap_{\finplus}(\config(W;k),\config(N)) \]
and the horizontal one is determined by an evaluation
map
\[ Z_N\lra \rmap_{\finplus}(\config(W;k),\config(N))). \]
We define $\rmap_{\finplus}^{\partial}(\config(M;k),\config(N))$
as the homotopy pullback of diagram~(\ref{eqn-nastybdry}). (In spirit it is a homotopy
fiber since $Z_M$, $Z_N$ and $\map(Z_M,Z_N)$ are contractible.)
Then there is a canonical map
\[ \emb_\partial(M,N) \lra \rmap_{\finplus}^{\partial}(\config(M;k),\config(N)). \]

As in the non-boundary case, let $\ms J_k$ (for a given non-negative integer $k$) be the Grothendieck topology on $\man$ with
coverings given by collections $\{U_i \hookrightarrow M\}$ subject
to the condition that every finite subset of cardinality $k$ in the interior of $M$ is contained in (the image of) $U_j$ for some $j$.

\begin{thm} \label{thm-configsheafbdry} For fixed $N$ with $L\hookrightarrow \partial N$, the contravariant functor
\[ M\mapsto \rmap_{\finplus}^{\partial}(\config(M;k),\config(N)) \]
is a homotopy sheaf for the Grothendieck topology $\ms J_k$ on $\man$~.
\end{thm}

There are analogous definitions in the local setting. Let $\config^\loc(M;k)$ be the
comma type construction whose objects are morphisms in $\config(M;k)$ which under the reference functor to
$\finplus$ are taken to \emph{proper} morphisms with target $[1]$. Clearly $\config^\loc(M;k)$ comes with a
reference map to $N\fin\subset N\finplus$~.

\begin{expl} {\rm We can use the particle model of $\config(M)$.
Then $\config^\loc(M;k)$ is exactly the same as $\config^\loc(M_-;k)$, which we defined previously in the
boundariless setting.
}
\end{expl}

\medskip
By analogy with the definition of $\rmap_{\finplus}^{\partial}(\config(M;k),\config(N))$,
we define
\[ \rmap_{\fin}^{\partial}(\config^\loc(M;k),\config^\loc(N)) \]
as the homotopy pullback of
\[
	\begin{tikzpicture}[descr/.style={fill=white}, baseline=(current bounding box.base)] ]
	\matrix(m)[matrix of math nodes, row sep=2.5em, column sep=2.5em,
	text height=1.5ex, text depth=0.25ex]
	{
	 & \rmap_\fin(\config^\loc(M;k),\config^\loc(N)) \\
	\map_\partial(Z_M,Z_N) & \map(Z_M,\rmap_\fin(\config^\loc(W;k),\config^\loc(N))) \\
	};
	\path[->,font=\scriptsize]
		(m-2-1) edge node [auto] {} (m-2-2);
	\path[->,font=\scriptsize] 		
		(m-1-2) edge node [auto] {} (m-2-2);
	\end{tikzpicture}
\]

\begin{prop} \label{prop-localconfigsheafbdry} For fixed $N$ with $L\hookrightarrow\partial N$, the contravariant functor
\[ M\mapsto \rmap_{\fin}^{\partial}(\config^\loc(M;k),\config^\loc(N)) \]
is a homotopy sheaf for the Grothendieck topology $\ms J_1$ on $\man$~.
\end{prop}

The following is then a slight generalization of our main theorem~\ref{thm-main}.

\begin{thm} \label{thm-mainbdry} For any integer $k\ge 1$, the commutative square
\[
	\begin{tikzpicture}[descr/.style={fill=white}, baseline=(current bounding box.base)] ]
	\matrix(m)[matrix of math nodes, row sep=2.5em, column sep=2.5em,
	text height=1.5ex, text depth=0.25ex]
	{
	T_k\emb_{\partial}(M,N) & \rmap^\partial_\finplus(\config(M;k),\config(N)) \\
	T_1\emb_{\partial}(M,N) & \rmap^\partial_\fin(\config^\loc(M;k),\config^\loc(N)) \\
	};
	\path[->,font=\scriptsize]
		(m-1-1) edge node [auto] {} (m-1-2);
	\path[->,font=\scriptsize]
		(m-2-1) edge node [auto] {} (m-2-2);
	\path[->,font=\scriptsize]
		(m-1-1) edge node [left] {} (m-2-1);
	\path[->,font=\scriptsize] 		
		(m-1-2) edge node [auto] {\textup{specialization}} (m-2-2);
	\end{tikzpicture}
\]
is homotopy cartesian. \emph{(The horizontal arrows are due to theorem~\ref{thm-configsheafbdry} and
proposition~\ref{prop-localconfigsheafbdry},
respectively.)}
\end{thm}

\begin{cor} \label{cor-mainbdry} If $\hodim(M)\le n-3$ relative to $\partial M$, then the commutative square
\[
\begin{tikzpicture}[descr/.style={fill=white}, baseline=(current bounding box.base)] ]
	\matrix(m)[matrix of math nodes, row sep=2.5em, column sep=2.5em,
	text height=1.5ex, text depth=0.25ex]
	{
	\emb_{\partial}(M,N) & \rmap^\partial_\finplus(\config(M),\config(N)) \\
	\imm_{\partial}(M,N) & \rmap^\partial_\fin(\config^\loc(M),\config^\loc(N)) \\
	};
	\path[->,font=\scriptsize]
		(m-1-1) edge node [auto] {} (m-1-2);
	\path[->,font=\scriptsize]
		(m-2-1) edge node [auto] {} (m-2-2);
	\path[->,font=\scriptsize]
		(m-1-1) edge node [left] {} (m-2-1);
	\path[->,font=\scriptsize] 		
		(m-1-2) edge node [auto] {\textup{specialization}} (m-2-2);
	\end{tikzpicture}
\]
is homotopy cartesian.
\end{cor}

\begin{rem} \label{rem-imap} {\rm The map
\[
\emb_\partial(M,N) \lra \rmap_{\finplus}^{\partial}(\config(M),\config(N))
\]
of corollary~\ref{cor-mainbdry}
factors through the space of injective continuous maps from $M$ to $N$ satisfying the usual boundary condition. The space of injective
continuous maps from $M$ to $N$ should not be confused with the space of locally flat topological embeddings $M\to N$. For example, the
space of continuous injective maps from $D^1$ to $D^3$ satisfying the standard boundary conditions is contractible. By contrast, the
space of locally flat topological embeddings $D^1\to D^3$ satisfying the standard boundary conditions
has many connected components, the study of which is called knot theory! However, in the high codimension situation the two agree up
to weak equivalence (see \cite{Lashof}, \cite{Sakai}).

In the case where $L=\emptyset$, using the particle models
for $\config(M)$ and $\config(N)$ the claimed factorization is obvious.
In the case of nonempty $L$, with our current definition of $\rmap_{\finplus}^{\partial}(\config(M),\config(N))$ some
tweaking is needed, which we leave to the reader.
}
\end{rem}

\section{Operads and fiberwise complete Segal spaces}  \label{sec-oper}
\subsection{Preview}
Let $V$ and $W$ be finite dimensional real vector spaces.
One of the main outcomes of this section is an identification (weak homotopy equivalence)
of
\[ \rmap_\fin(\config(V),\config(W)) \]
with the space of derived morphisms of operads from
the operad of little disks in $V$ to the operad of little disks in $W$. An immediate corollary is that
\[ \rmap_\fin(\config^\loc(M),\config^\loc(N)) \]
can be identified with the space
of sections of a fibration $E\to M$ defined as follows. The fiber over $x\in M$ is the space of pairs $(y,h)$ where $y\in N$
and $h$ is a derived morphism of operads, from the operad of little disks in $T_xM$ to the operad of little disks in $T_yN$.
This identification is intended to make theorem~\ref{thm-main} and corollary~\ref{cor-main}
more applicable. There is a variant for
the case with boundary.

\subsection{Operads and categories}
Let $P$ be an operad in the symmetric monoidal category of spaces. In this chapter, by the category of spaces we mean
the category of simplicial sets. Any other of the usual convenient categories of topological spaces would be equally valid, and we leave that routine translation to the interested reader.

We think of this in the following terms: $P$ is a functor from the category of finite sets and bijections to
spaces, and for every map $f:T\to S$ of finite sets there is an operation
\[  \lambda_f\co P(S)\times\prod_{i\in S} P(T_i) \lra P(T) \]
where $T_i=f^{-1}(i)\subset T$. Also $P(S)$ contains a distinguished unit element when $S$ is a singleton.
Sensible naturality, associativity and unital properties are satisfied. Note in particular that any
permutation $f\co S\to S$ induces a map $P(S)\to P(S)$ in two ways: firstly because
$P$ is a functor from the category of finite sets and bijections to spaces, and secondly by
\[ P(S)\ni x\mapsto \lambda_f(x,1,1,\dots,1)\in P(S)~. \]\
for $x\in P(S)$. These two maps agree as per definition. \newline
What we have described is also called a \emph{plain} operad in the category of spaces, in contradistinction
to \emph{colored operads}.

\begin{expl} \label{expl-opertomon} {\rm Let $P$ be an operad in spaces, as above, and let
$S$ be a singleton. The space $P(S)$ is a topological monoid with unit.
Indeed, let $f\co S\to S$ be the identity. This determines an operation
$\lambda_f\co P(S) \times P(S) \to P(S)$
which amounts to an associative multiplication on $P(S)$. }
\end{expl}

\medskip
To the operad $P$ we associate a topological category $\ms C_P$ (category object in the category of spaces), and a
functor $\nu\co \ms C_P\to \fin$, as follows. The space of objects is
\[  \coprod_{k\ge 0} P(\uli k) \]
which we map to the set of objects of $\fin$ by taking all of $P(\uli k)$ to $\uli k$\,.
The space of morphisms in $\ms C_P$ lifting a morphism $f:\uli k\to \uli\ell$ in $\fin$ is
\[ P(\uli\ell)\times\prod_{i\in \uli\ell} P(f^{-1}(i)). \]
Source and target of an element in that space are determined by applying to it $\lambda_f$
and the projection to $P(\ell)$, respectively. Composition and identity morphisms are obvious.

\begin{expl} \label{expl-conoper}
{\rm Let $P$ be the operad of little $m$-disks. The simplicial spaces
$\config(\RR^m)$ and $N\ms C_P$ are related by a chain of degreewise weak equivalences over $N\fin$.
This is fairly clear if we use the multipatch model for $\config(\RR^m)$, with the
standard metric on $\RR^m$. More precisely, in the multipatch model $\config(\RR^m)$ is $N\ms B$
for a certain topological category $\ms B$. There is a forgetful functor
$\ms B\to \ms C_P$ which induces a degreewise weak equivalence $N\ms B\to N\ms C_P$. \newline
Note in passing that the forgetful functor $\config^\loc(\RR^m)\to \config(\RR^m)$ is also a
degreewise homotopy equivalence.
}
\end{expl}

\begin{rem} {\rm The standard way to associate a category to an operad $P$ is to make a PROP from it.
See \cite{Voronov} and \cite{MSS} for details. The PROP is a category with a monoidal structure.
The operad can be recovered from the associated PROP (with the monoidal
structure). We emphasize that $\ms C_P$ is not the PROP associated to $P$, but rather
the comma category \emph{over the object $\uli 1$} in the PROP associated to $P$. The point that we are
trying to make in this section is that passing from $P$ to $\ms C_P$ is nevertheless in many cases a rather faithful process.
But there are some cases where it loses essential information, and here is one such case. \newline
Let $K$ be a topological monoid with unit. There is an operad $P=P_K$ in spaces such that $P(S)=K$ when $|S|=1$ and
$P(S)=\emptyset$ in all other cases. Set it up
in such a way that the multiplication on $K$ arising from the operad structure on $P$ agrees with
the prescribed multiplication on $K$. (The functor $K\mapsto P_K$ is left adjoint to the functor
from operads to monoids that we saw in example~\ref{expl-opertomon}.) For this $P$, the space of objects of $\ms C_P$
is identified with $K$. The map
\[  (\textrm{source,target})\co \mor(\ms C_P) \lra \ob(\ms C_P)\times\ob(\ms C_P) \]
turns out to be the map $K\times K\lra K\times K$ given by $(a,b)\mapsto (ab,a)$. If $K$ happens to be a
topological group, then it is a homeomorphism. To say it informally, for arbitrary $x,y\in \ob(\ms C_P)$ there is a unique
morphism $x\to y$. Consequently, $\ms C_P$ is equivalent to a category with one object and one morphism, and it has no information about the multiplication in $K$.}
\end{rem}

\begin{lem} \label{lem-opfib} Let $P$ be an operad in spaces such that $P(\uli 1)$ is weakly contractible. Then
$N\ms C_P$ is a fiberwise complete Segal space over $N\fin$.
\end{lem}

\begin{proof} Since $P(\uli 1)$ is weakly contractible, a morphism $f\co x\to y$ in $\ms C_P$ is homotopy invertible if and only
if it determines an invertible morphism in $\fin$, in which case $\nu(x)=\nu(y)$. Hence the
space of homotopy invertible morphisms in $\ms C_P$ is
\[ \coprod_{\ell\ge 0}\, \prod_{\sigma\in\Sigma_\ell} P(\uli\ell)\times\prod_{i\in \ell} P(\uli 1) \]
which is weakly homotopy equivalent to
\[ \coprod_{\ell\ge 0}\, \prod_{\sigma\in\Sigma_\ell} P(\uli\ell). \]
Therefore condition ($\kappa_{\rel}$) is satisfied. \end{proof}

\bigskip
The mapping space $\map(P,Q)$ between two operads $P$ and $Q$ has $k$-simplices given by operad maps
\[  P\lra Q^{\Delta[k]} \]
where $Q^{\Delta[k]}$ is the operad defined by $Q^{\Delta[k]}(S)=\map(\Delta[k],Q(S))$ for a finite set $S$,
and $\map(\Delta[k],Q(S))$ is the space whose set of $\ell$-simplices is the set of maps of spaces
$\Delta[\ell]\times\Delta[k]\to Q(S)$.

By Berger and Moerdijk \cite[Expl~3.3.1]{BergerMoerdijk2003} there is a model structure on the category of operads in spaces in which a morphism $f\co P\to Q$ is a weak equivalence, or a fibration, if it is degreewise a weak equivalence, respectively fibration. This model structure is simplicial with mapping spaces as above, and so the derived mapping space $\rmap(P,Q)$ is the derived variant of $\map(P,Q)$ (see Appendix \ref{section-Rmap}).

\medskip
Let $P$ and $Q$ be operads in the category of spaces. We assume that $P(\uli 1)$ and $Q(\uli 1)$ are weakly contractible. For
simplicity we also assume that $Q$ is fibrant (which means that each of the spaces $Q(S)$ is fibrant).
The functor which to an operad $P$ associates the simplicial space $N\sC_P$ over $\fin$ preserves weak equivalences and so induces a map
\begin{equation} \label{eqn-passage}
\rmap(P,Q) \to \rmap_\fin(N\ms C_P,N\ms C_Q)
\end{equation}
Our main goal in the remainder of this section is to show:

\begin{thm} \label{thm-passage} The map~{\rm(\ref{eqn-passage})} is a weak homotopy
equivalence if $P(\uli 0)$, $P(\uli 1)$, $Q(\uli 0)$ and $Q(\uli 1)$ are weakly contractible.
\end{thm}

\subsection{Operads and dendroidal spaces}
Recall the concept of a dendroidal space from Moerdijk-Weiss\footnote{This is Ittay Weiss.} and Cisinski-Moerdijk \cite{MoerdijkWeiss,CisinskiMoerdijk1,CisinskiMoerdijk2,CisinskiMoerdijk3}. A dendroidal
space is a contravariant functor from a certain category $\dend$ to spaces.
(Cisinski-Moerdijk write $\Omega$ where we have $\dend$.)
An object $T$ of $\dend$ is a finite nonempty set $\epsilon(T)$ with a partial order $\le$
and a distinguished subset $\lambda(T)$ of the set of maximal elements of $\epsilon(T)$, the set of \emph{leaves}, so that the following conditions are
satisfied:
\begin{itemize}
\item $\epsilon(T)$ has a minimal element (called the \emph{root});
\item for each element $e$ of $\epsilon(T)$, the set $\{y\in \epsilon(T)~|~y\le e\}$ with the restricted ordering is linearly ordered.
\end{itemize}
The elements of $\epsilon(T)$ are also called \emph{edges} of the tree $T$. The elements of $\epsilon(T)\smin\lambda(T)$
can be called vertices.
But we tend to depict the vertices $v_e$ as little bullets attached at the upper end of each non-leaf edge $e$. The
\emph{outgoing edge} of a vertex $v_e$ is $e$. In drawing such trees, we try to ensure that the outgoing edge $e$ of a
vertex $v_e$ is situated immediately below $v_e$,
and that the immediate successors of that edge are attached to $v_e$ but above it. These are called the
\emph{incoming edges} of $v_e$\,. The following is an example where $\epsilon(T)$ has 13 elements and $\lambda(T)$ has 5 elements.
\[
\xymatrix@C=20pt@R=18pt{
&& \ar@{-}[dr] & \bullet \ar@{-}[d] & \ar@{-}[dl] & \ar@{-}[d] \\
\ar@{-}[dr]_-a & \bullet \ar@{-}[d]^-b &  \bullet \ar@{-}[dl]^-c & \bullet  \ar@{-}[dr]  & \ar@{-}[d]
& \ar@{-}[dl] \bullet \\
& \bullet \ar@{-}[dr]_d &&& \bullet \ar@{-}[dll] & \\
&& \bullet \ar@{-}[d] && \\
&  &&
}
\]
Such a tree $T$ freely generates a finite (colored) operad
in the category of sets. The set of colors is $\epsilon(T)$, while each vertex $v_e$
contributes a generating operation $\omega_e$ from the set of incoming edges of $v_e$ to the outgoing edge $e$.
Typically there are other operations obtained by composing some of the $\omega_e$\,. More precisely,
given edges $s_1,\dots,s_m$ and $t$ in $T$, an $m$-ary operation $\omega$ with source
$(s_1,\dots,s_m)$ and target $t$ exists if and only if $s_i\ge t$ for $i=1,\dots,m$ and the following conditions are satisfied:
\begin{itemize}
\item[-] (independence) whenever $s_i\le s_j$ for some $i,j\in\{1,2,\dots,m\}$, then $i=j$;
\item[-] (cut property) every edge path from some leaf of $T$ to $t$ contains one of the edges $s_i$
(by independence, not more than one).
\end{itemize}
In that case the operation $\omega$ is unique. This description makes it clear how operations are to be composed.
If there is an operation $\omega_1$ with source $(s_1,\dots,s_m)$ and target $t$, and if there
is an operation $\omega_0$ with source $(r_1,\dots,r_k)$ and target $s_i$ for some $i\in\{1,2,\dots,m\}$, then
$\omega_1\circ_i\omega_0$ is the unique operation with source $(s_1,\dots,s_{i-1},r_1,\dots,r_k,s_{i+1},\dots,s_m)$
and target $t$. \newline
In our example, there are 8 vertices and the associated generating operations are $0$-ary (three instances), $1$-ary
(one instance), $2$-ary (one instance) and $3$-ary (three instances). The possible sources for operations
with target $d$
are $(a)$, $(a,b)$, $(a,c)$ and $(a,b,c)$, up to (irrelevant) permutations.
The source $(a)$ for edge $d$ corresponds to the $1$-ary operation from edge $a$ to edge $d$ obtained by substituting the $0$-ary
generating operations corresponding to $v_b$ and $v_c$ in the $3$-ary generating operation corresponding to $v_d$\,. \newline
The set of morphisms in $\dend$ from a tree $T$ to another tree $T'$ is the set of
morphisms between the corresponding operads. In particular a morphism $T\to T'$
comes with an underlying map from the set of colors $\epsilon(T)$ to the
set of colors $\epsilon(T')$. By the observations that we have just made, a morphism in $\dend$ from $T$ to $T'$ is determined
by the underlying map $\epsilon(T)\to \epsilon(T')$. See also \cite{MoerdijkWeiss}.

\begin{expl} \label{expl-corolla} {\rm For each $n\ge 0$ there is an object in $\dend$, unique up to isomorphism,
which has just one vertex and $n$ leaves. These objects are called \emph{corollas}. The following pictures show a
corolla with zero leaves and a corolla with 3 leaves:
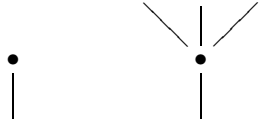
\[
\xymatrix@C=16pt@R=16pt{        && \ar@{-}[dr] & \ar@{-}[d] & \ar@{-}[dl] \\
\bullet \ar@{-}[d]&&& \bullet \ar@{-}[d] && \\
 &&&& &&
}
\]
For an object $T$ and vertex $v_e$ with $n$ incoming edges, there exists a morphism $f_e\co S\to T$
where $S$ is a corolla with $n$ leaves, such that $f_e$ takes the root of $S$ to $e$ and the leaves of $S$
to the incoming edges of $v_e$. This morphism is unique up to composition with an automorphism of $S$.
}
\end{expl}

\begin{expl} \label{expl-treedecomp} {\rm
A plain operad $Q$ in the category of spaces
determines a dendroidal space $N_dQ$, the \emph{operadic nerve}, such that $(N_dQ)_T$~, for an object $T$ in $\dend$,
is the space of operad maps from the (multi-colored, discrete) operad freely generated by $T$ to the
plain operad $Q$. 
We note that $(N_dQ)_T$ is a point when $T$ is the tree $\eta$ with one edge and no vertices; this reflects our assumption that $Q$ is plain.
Furthermore $(N_dQ)_S\cong Q(\uli n)$ when $S$ is a corolla with $n$ leaves. For general $T$ in $\dend$
and a selection of morphisms $f_e\co S_e\to T$ from corollas as in example~\ref{expl-corolla}, one for each $e\in \epsilon(T)\smin
\lambda(T)$, the induced maps $f_e^*$ produce a homeomorphism
\[  (N_dQ)_T\lra \prod (N_dQ)_{S_e}~. \]
}
\end{expl}

\medskip
Being a functor category, the category of dendroidal spaces has (at least) two canonical model structures having degreewise
weak equivalences (either by using degreewise fibrations or degreewise cofibrations). We refer to any of these model structures as the degreewise model structure.

There is another model structure on the category of dendroidal spaces, obtained from the degreewise one by left Bousfield localization, whose fibrant objects are the dendroidal spaces $X$ which are fibrant in the degreewise model structure, satisfy $X_\eta \simeq *$ and satisfy the Segal condition: the map
\[
X_T \to \prod_{v \in T} X_{S_v} \; ,
\]
where $v$ runs over the vertices of $T$, is a weak homotopy equivalence. A morphism (natural transformation) between fibrant dendroidal spaces is a weak equivalence if it is a degreewise weak equivalence; a morphism is a cofibration if it is cofibration in the degreewise model structure.

The nerve functor $N_d$ preserves weak equivalences, so it induces a canonical map
\begin{equation}\label{eqn-nerveOp}
\rmap(P,Q) \to \rmap(N_dP,N_dQ)
\end{equation}
for plain operads $P$ and $Q$. For the target derived mapping space, it does not matter which of the above model
structures is used.

\begin{thm}\emph{\cite[Thm 8.15]{CisinskiMoerdijk3}, \cite{CisinskiMoerdijk2}, \cite[Thm 1.1]{BergnerHackney}}\label{thm-CisMoe}
The map \emph{(\ref{eqn-nerveOp})} is a weak homotopy equivalence.
\end{thm}

\subsection{Dendroidal spaces and simplicial spaces over $\fin$} \label{subsec-dendroidal}
There is a close connection between dendroidal spaces and simplicial spaces with a
simplicial map to $N\fin$. We begin with the observation
that a simplicial space $X$ with a
simplicial map to a simplicial set $Z$ is the same thing as a contravariant functor
from $\simp(Z)$ to spaces. Here $\simp(Z)$ is a Grothendieck construction; it
has objects $(m,y)$ with $y\in Z_m$\,, and a morphism from $(m,y)$ to $(n,z)$ is a monotone map
$f\co [m]\to [n]$ such that $f^*z=y$.
We apply the observation with
$Z=N\fin$, writing $\simp(\fin)$ instead of $\simp(N\fin)$.
There is a functor
\[ \varphi\co\simp(\fin) \to \dend \] defined as follows. To an object $(p,S_*)$ of $\simp(\fin)$ where
\[ S_*=(S_0 \leftarrow S_1 \leftarrow S_2 \leftarrow \cdots \leftarrow S_p) \]
we associate the tree $T$ where $\epsilon(T)$ is the disjoint union of the $S_i$ and an
additional element $r$, with $\lambda(T)$ corresponding to $S_p$\,. The partial order on $\epsilon(T)$ is the
obvious one, where $r$ is the minimal element and $x\in S_i$ is $\le y\in S_j$ if
$i\le j$ and $y$ is in the preimage of $x$ under the composite map $S_j\to S_i$ in the string above. \newline
Composition with $\varphi\co\simp(\fin)\to \dend$
is again a functor $\varphi^*$ which takes us from dendroidal spaces, alias functors from $\dend^\op$ to spaces,
to simplicial spaces over $N\fin$, alias functors from $\simp(\fin)^\op$ to spaces.

\begin{lem} \label{lem-recog} For a plain operad $P$ in the category of spaces,
$\varphi^*(N_d P)$ is exactly $N\ms C_P$ as a simplicial space over $N\fin$. 
\end{lem}

By dint of this lemma and theorem~\ref{thm-CisMoe}, we have reduced theorem~\ref{thm-passage} to the
following proposition.

\begin{prop} \label{prop-othertree} If $P$ and $Q$ are operads in spaces such that $P(\uli 0)$, $P(\uli 1)$, $Q(\uli 0)$
and $Q(\uli 1)$ are weakly contractible, then the map
\[ \rmap(N_d P,N_d Q) \lra \rmap(\varphi^*(N_d P),\varphi^*(N_d Q)) \]
given by composition with $\varphi$ is a weak equivalence.
\end{prop}

The proof of proposition~\ref{prop-othertree} begins with another reduction step
showing that $\dend$ can be replaced by a more user-friendly subcategory $\dendroca$.
That is to say, no essential information is lost if the dendroidal nerves $N_dP$, $N_dP$ are replaced
by their restrictions to $\dendroca$.

\begin{defn}
{\rm The category $\dend$ has a full subcategory $\dendca$ spanned by the objects $T$ where the set of
leaves $\lambda(T)$ is empty.
The category $\dendca$\,, as a category in its own right, is much easier to understand than $\dend$.
Its objects are finite posets $T$ (we no longer distinguish between $T$ and $\epsilon(T)$, the set of edges)
with a unique minimal element such that, for every $e\in T$, the set $\{y\in T~|~y\le e\}$ is
linearly ordered. A morphism $S\to T$ in $\dendca$ is a map $f\co S\to T$ of sets
such that $e_0\le e_1$ in $S$ implies $f(e_0)\le f(e_1)$ in $T$\,, and $e_0,e_1$ unrelated in $S$
implies that $f(e_0),f(e_1)$ are unrelated in $T$. Beware that such an $f$ need not be injective.
(For example, if $S$ and $T$ happen to be totally ordered, then any map $f\co S\to T$ which
respects $\le$ is a morphism.) We still find
it convenient to draw objects $T$ of $\dendca$ as 1-dimensional CW-spaces with edges and vertices; every edge has two vertices except
the root, which has just one. But the vertices in those drawings are admittedly redundant.
\newline
The category $\dendca$ has a subcategory $\dendroca$ which comprises all objects, but only those morphisms
which take root to root. \newline
The inclusion functor $\dendca\to\dend$ has a left adjoint which forgets the leaf set $\lambda(T)$, or replaces it by
the empty set. The inclusion functor $\dendroca\to\dendca$ also has a left adjoint given by $T\mapsto T_+$ where $T_+$~, as a poset,
is the same as $T$ with a new element added (not previously in $T$) which is decreed to be minimal and so
serves as the root of $T_+$. The unit morphisms of the adjunction are the maps $T\to T_+$ defined by the obvious inclusion on
edge sets. The counit morphisms $T_+\to T$ of the adjunction are defined by $e\mapsto e$ for $e\in T$, and root to root. We leave
to the reader the verification that these prescriptions do define morphisms.
}
\end{defn}

Let $N_d^{rc}P$ and $N_d^{rc}Q$ be the restrictions of $N_dP$ and $N_dQ$ to $\dendroca$, respectively.

\begin{lem} \label{lem-rocared} The map $\rmap(N_d P,N_d Q) \lra \rmap(N_d^{rc} P,N_d^{rc} Q)$ given by restriction is a
weak equivalence.
\end{lem}

\begin{proof}[Proof of lemma~\ref{lem-rocared}]
Write $\iota \co \dendroca\to\dend$ for the inclusion and $X$ for $N_d P$.
By lemma \ref{lem:LK1}, it is enough to show that the derived counit
\[
\hLK{\iota} \RR {\iota}^* X \to X
\]
is a weak equivalence. Because $\iota$ has a left adjoint, which we denote here by $\kappa$, it is easy to see that $\hLK{\iota}$ agrees
with $\kappa^*$. Thus it suffices to show that
\[
\LL \kappa^* \RR {\iota}^* X \to X
\]
is a weak equivalence. Moreover, $\kappa^*$ and $\iota^*$ preserve weak equivalences, so it is enough that the non-derived counit
$\kappa^* \iota^* X \to X$ is a weak equivalence. This is true because of the assumption that $P(\uli 0)$ and $P(\uli 1)$ are
weakly contractible.
\end{proof}

\medskip
The category $\simp(\fin)$ has an endofunctor $\beta$ defined on objects by
\[  (S_0\leftarrow S_1\leftarrow\cdots\leftarrow S_p)~~\mapsto~~
(S_0\leftarrow S_1\leftarrow\cdots\leftarrow S_p\leftarrow\emptyset)~. \]
The composition $\varphi\beta$ lands in the subcategory $\dendroca$~, so that we have the following
commutative diagram:
\begin{equation} \label{eqn-betapsi}
\begin{tikzpicture}[descr/.style={fill=white}, baseline=(current bounding box.base)],
	\matrix(m)[matrix of math nodes, row sep=2.5em, column sep=2.5em,
	text height=1.5ex, text depth=0.25ex]
	{
	\simp(\fin) & \dend \\
	\simp(\fin) & \dendroca \\
	};
	\path[->,font=\scriptsize]
		(m-1-1) edge node [auto] {$\varphi$} (m-1-2);
	\path[->,font=\scriptsize]
		(m-2-1) edge node [auto] {$\psi$} (m-2-2);
	\path[->,font=\scriptsize]
		(m-2-1) edge node [left] {$\beta$} (m-1-1);
	\path[->,font=\scriptsize] 		
		(m-2-2) edge node [right] {$\iota=\textup{incl.}$} (m-1-2);
	\end{tikzpicture}
\end{equation}
This defines $\psi\co\simp(\fin)\to\dendroca$. Clearly composition with $\beta$ defines a weak
equivalence
\[  \rmap(\varphi^*(N_dP),\varphi^*(N_dQ))\lra \rmap(\beta^*\varphi^*(N_dP),\beta^*\varphi^*(N_dQ))). \]
(Here we are using the weak contractibility of $P(0)$ and $Q(0)$.)
Therefore the commutative diagram~(\ref{eqn-betapsi}) together with lemma~\ref{lem-rocared} and lemma \ref{lem:LK1}
(and the fact that $\psi^*$ preserves weak equivalences)
amount to a reduction of proposition~\ref{prop-othertree} to the following.

\begin{prop} \label{prop-baad} In the circumstances of proposition~\ref{prop-othertree}, the unit map
\[ N_d^{rc} Q \lra (\hRK{\psi} \psi^*)(N_d^{rc} Q) \]
is a weak equivalence, where $\hRK{\psi}$ is the homotopy right Kan extension along $\psi$.
\end{prop}

\subsection{Sets of independent edges in a tree}
Technicalities apart, our proof of proposition~\ref{prop-baad} has three main ingredients.
One of them is lemma~\ref{lem:LK1} which we have already seen in action.
Another main ingredient, lemma~\ref{lem-diago}, will allow us to exploit the special property of $N_dQ$ expressed in
example~\ref{expl-treedecomp}. The third main ingredient is definition~\ref{defn-blockers} which extracts a poset from
any object of $\dendroca$.

\begin{lem} \label{lem-diago}
Let $Z(1),\dots,Z(m)$ be simplicial sets
and let $F$ be a contravariant functor from the product $\prod_{j=1}^m\simp(Z(j))$ to spaces. Suppose that for every
morphism
\[ (g_1,\dots,g_m)\co ((k_1,\sigma_1),\dots,(k_m,\sigma_m))
\lra ((\ell_1,\tau_1),\dots,(\ell_m,\tau_m)) \]
in $\prod_{j=1}^m\simp(Z(j))$ where each of the
monotone maps $g_j\co[k_j]\to [\ell_j]$
is surjective, the induced map
$F((\ell_1,\tau_1),\dots,(\ell_m,\tau_m))\to
F((k_1,\sigma_1),\dots,(k_m,\sigma_m))$
is a weak equivalence. Let
\[ \delta\co \simp(\prod_j Z(j))\to\prod_j\simp(Z(j)) \]
be the diagonal inclusion. Then the restriction map
$\holim~F \lra \holim~F\delta$ is a weak equivalence.
\end{lem}

The proof is given in appendix~\ref{sec-diago}.
It is closely related to the fact that geometric realization, viewed as a
functor from simplicial sets to compactly generated Hausdorff spaces, commutes with products.
We will use the lemma in a case where
\[ F((k_1,\sigma_1),\dots,(k_m,\sigma_m))= \prod_{j=1}^m F_j(k_j,\sigma_j) \]
for functors $F_j$ from $\simp(Z(j))$ to spaces, $j=1,\dots,m$. We
need to understand $\holim~F\delta$ and we hope therefore that we can understand $\holim~F$. There may be a strong temptation here to
assume right away that $\holim~F$ is the product of factors $\holim~F_j$ for $j=1,\dots,m$, but a moment's consideration shows that instead
\[  \holim~F \cong \prod_{j=1}^m \holim~F_jp_j \]
where $p_j$ is the projection from $\prod_i\simp(Z(i))$ to the factor $\simp(Z(j))$. Another moment's consideration shows that
\[ \holim~F_jp_j \simeq \map(Z^\perp_j\,,\,\holim~F_j) \]
where $Z^\perp_j$ is the product of all classifying spaces $B\simp(Z(i))$ for
\[ i\in \{0,1,\dots,m\}\smin\{j\}. \]
It is well known that $B\simp(Z(i))\simeq|Z(i)|$ \cite[15.1.20]{Hirschhorn}.
So we are allowed to write
\[  \holim~F\simeq \holim~F\delta \simeq \prod_j \holim~F_j \]
after all, if we can show that $|Z(j)|$ is contractible, for $j=1,\dots,m$.
For us, $Z(j)$ is the nerve of a poset and there is typically no great difficulty in showing that it is contractible,
e.g. by exhibiting a maximal element in the poset.

\begin{defn} \label{defn-blockers} {\rm With an object $T$ of $\dendroca$ we associate a poset $\ms B^T$. The
elements of $\ms B^T$ are subsets $S$ of $T$ (formerly $\epsilon(T)$, the set of edges) which satisfy
\emph{independence}: for $s_0,s_1\in S$, if $s_0\le s_1$ then $s_0=s_1$.
(These subsets $S$ of $\epsilon(T)$ are exactly the sources of operations with target \emph{root}
in the operad associated to $T$.)
For elements $S,S'$ of $\ms B^T$ we write $S\le S'$ if there exists a map $w\co S\to S'$ such that
$w(s)\le s$ (no misprint) for all $s\in S$; the map $w$ is unique if it exists. The poset $\ms B^T$ has a
maximal element $\{\textup{root}\}$ and a minimal element $\emptyset$. \newline
The rule $T\mapsto \ms B^T$ is a covariant functor from $\dendroca$ to the category of posets.
}
\end{defn}

\begin{expl} \label{expl-blockers} {\rm Let $S_*=(S_0\leftarrow S_1\leftarrow\cdots\leftarrow S_{n-1}\leftarrow S_n)$ be an $n$-simplex
in the nerve of $\fin$, alias object of $\simp(\fin)$. Then $\psi(S_*)=:T$ is an object of $\dendroca$. The
sets $S_0,\dots,S_n$ can be viewed as subsets
of $T$ and as such they turn out to be elements of the poset $\ms B^T$. Moreover we have
\[  S_0\ge S_1\ge \dots \ge S_n \]
using the order in $\ms B^T$; this is an $n$-simplex in the nerve of $\ms B^T$. \newline
Conversely, for an object $T$ in $\dend$, any $n$-simplex
\[   S_0\ge S_1\ge \dots \ge S_n \]
in $\ms B^T$ determines an $n$-simplex in the nerve of $\fin$ once we take the trouble to choose a
total ordering on each of the sets $S_i$~.
}
\end{expl}

In preparation for the proof of proposition~\ref{prop-baad}, we unravel the meaning of $(\hRK{\psi} \psi^*)(N^{rc}_dQ)$.
Write $N$ for $N^{rc}_dQ$.
Write $N(T)$ for the value of $N$ at $T$,
where $T$ is an object of $\dendroca$.
The standard formula for $(\hRK{\psi} \psi^*)(T)$ is
\begin{equation} \label{eqn-unrav1}  \holimsub{(S_*,f)\textup{ in }(\psi\downarrow T)} N(\psi(S_*))~. \end{equation}
Here the homotopy inverse limit is taken over the comma category $(\psi\!\downarrow\! T)$ with objects which are pairs $(S_*,f)$,
where $S_*$ is an object of $\simp(\fin)$ and $f$ is a morphism in $\dendroca$ from $\psi(S_*)$ to $T$.
We want to simplify formula~(\ref{eqn-unrav1}) to
\begin{equation} \label{eqn-unrav2}  \holimsub{S_*\textup{ in }\simp(\ms B^T)} N(\psi^T(S_*)) \end{equation}
where $\psi^T\co \simp(\ms B^T)\to \dendroca$ is defined as follows.
An object of $\simp(\ms B^T)$ is a string
$S_0\ge S_1\ge \cdots \ge S_n$
in $\ms B^T$. This determines an object $\psi^T(S_*)$
in $\dendroca$ whose set of edges is the disjoint union
of the $S_i$ and a singleton for the root. (For example, the elements of $S_0$ become the edges just above
the root.) Reason for the simplification~(\ref{eqn-unrav1}) $\leadsto$ (\ref{eqn-unrav2}): there is a diagram of functors,
\begin{equation} \label{eqn-unrav3}
	\begin{tikzpicture}[descr/.style={fill=white}, baseline=(current bounding box.base)] ]
	\matrix(m)[matrix of math nodes, row sep=2.5em, column sep=2.5em,
	text height=1.5ex, text depth=0.25ex]
	{
	(\psi\downarrow T) & \simp(\ms B^T)   \\
	\simp(\fin)  & \dendroca \\
	};
	\path[dotted,->,font=\scriptsize]
		(m-1-1) edge node [auto] {$w$} (m-1-2);
	\path[->,font=\scriptsize]
		(m-2-1) edge node [auto] {$\psi$} (m-2-2)
		(m-1-1) edge node [left] {\textup{forget}} (m-2-1)
		(m-1-2) edge node [auto] {$\psi^T$} (m-2-2);
	\end{tikzpicture}
\end{equation}
commutative up to natural isomorphism, where $w$ is explained by Example \ref{expl-blockers} and the fact that $\ms B^S$ is functorial in $S \in \dendroca$. This determines a map
$w^*$ from~(\ref{eqn-unrav2}) to~(\ref{eqn-unrav1}).
It is shown in lemma~\ref{lem-orderinfl} that $w^*$ is a weak equivalence.

Occasionally we need to simplify beyond~(\ref{eqn-unrav2}). Let $e$ be an edge of $\,T$ with the property that every edge
$e'$ of $\,T$ satisfies $e'\ge e$ or $e'\le e$. Let
\begin{equation} \label{eqn-cutposetform}
 \ms B^{T,e}=\{S\in \ms B^T~|~S<\{e\}~\}.
\end{equation}
(For clarification we point out that an element $S\in \ms B^T$ satisfies $S<\{e\}$ if and only if every edge in $S$ is $>e$.
Therefore every element of $\ms B^T\smin \ms B^{T,e}$ is a singleton, i.e., as a subset of $T$ it has exactly one element.)
According to lemma~\ref{lem-stem}, the projection map from~(\ref{eqn-unrav2}) to
\begin{equation} \label{eqn-unrav4}  \holimsub{S_*\textup{ in }\simp(\ms B^{T,e})} N(\psi^T(S_*)) \end{equation}
is a weak equivalence.

\medskip
\begin{proof}[Proof of proposition~\ref{prop-baad}] Formula~(\ref{eqn-unrav2}) is the starting point.
The aim is therefore to show that for every $T$ in $\dendroca$, the comparison map
$N(T) \to \holim~N\psi^T$
is a weak homotopy equivalence. We proceed by induction on the number of nodes in $T$, where \emph{node}
means a vertex with more than one incoming edge.
\newline
For the induction beginning, suppose that $T$ has no nodes at all. Then $T$ is totally ordered.
Let $e$ be the maximal edge of $T$. Apply lemma~\ref{lem-stem} with that $e$. Then $\ms B^{T,e}$
has only one element, $\emptyset$. Therefore, by the lemma, $\holim~N\psi^T$ is weakly contractible.
The same is true for $N(T)$. \newline
For the induction step, suppose that $T$ has at least one node.
Let $U$ be the set of incoming edges of the node closest to the root of $T$. For each $u\in U$ let
$T(u)$ be the tree consisting of the edges of $T$ which are $\ge u$, with $u$ as the root of $T(u)$, and no leaves.
The inclusion $T(u)\to T$ is a morphism in $\dend$, but not in $\dendroca$. We can factorize $T(u)\to \kappa(T(u))\to T$
where $\kappa$ is the left adjoint of the inclusion $\dendroca\to \dend$ and $T(u)\to \kappa(T(u))$ is a unit morphism
associated with the adjunction. Now $\kappa(T(u))\to T$ is a morphism in $\dendroca$. \newline
There is a functor
\[  A\co \simp(\prod_{u\in U}\ms B^{T(u)}) \lra \simp(\ms B^T) \]
induced by a map of posets given by $(R^u)_{u\in U} \mapsto \coprod_u R^u$.
The functor $A$ is a full embedding. The composition
\[ N(T) \lra \holim~N\psi^T \to \holim~(N\psi^T A) \]
is a weak equivalence for the following reasons. The product formula of example~\ref{expl-treedecomp} and our assumption
that $Q(\uli 0)$ and $Q(\uli 1)$ are weakly contractible allow us to write
\begin{equation} \label{eqn-newproduct}
N(T) \xrightarrow{\; \simeq \;}  N(\kappa(c_{|U|}))\times \prod_u N(\kappa(T(u)))
\end{equation}
where $c_{|U|}$ is a corolla with $|U|$
leaves. (To see the connection between~(\ref{eqn-newproduct}) and the product formula of example~\ref{expl-treedecomp},
observe that the morphisms in $\dend$ used in ~\ref{expl-treedecomp} to set up the old product formula factor through the morphisms
in $\dendroca$ required here to set up the new product formula.) There is a similar product formula to simplify each
value of the functor $\psi^T A$. Namely,
\[ (N\psi^T A)((R^u_*)_{u\in S})
\simeq N\psi^T(R^U_{-1}) \times \prod_u N\psi^T(R^u_*) \]
where $R^u_*= (R^u_0\ge R^u_1\ge\cdots\ge R^u_k)$
is a $k$-simplex in $\simp(\ms B^{T(u)})$, with the same $k$ for all $u\in U$, and $R^U_{-1}\subset U$ denotes
the set of those edges in $U$ which are on the edge path from some
element of $\coprod_u R^u_0$ to the root of $T$. (We can view $R^U_{-1}$ as a $0$-simplex in $\ms B^T$.) Consequently, using
lemma~\ref{lem-diago} and observations following it, we can also write
\[  \holim~(N\psi^T A) \simeq N(\kappa(c_{|U|}))\times \prod_u \holim~N\psi^{T(u)}~. \]
Now the inductive assumption comes to our help because $\kappa(T(u))$ has fewer nodes than $T$:
\[
N(\kappa(T(u))) \xrightarrow{\; \simeq \;} \holim~N\psi^{\kappa(T(u))} \xrightarrow{\; \simeq \;} \holim~N\psi^{T(u)}~.
\]
(The second map is induced by the inclusion $\ms B^{T(u)}\to \ms B^{\kappa(T(u))}$, and it is a weak equivalence by lemma~\ref{lem-stem}.)
Therefore it remains only to show, for the induction step, that the prolongation or restriction map
\[  \holim~N\psi^T \lra \holim~(N\psi^TA) \]
is a weak equivalence. This is again a consequence of lemma~\ref{lem-stem}: apply the lemma with $e$ equal to
the edge just below the node nearest to the root. \end{proof}

\section{Conservatization} \label{sec-face}
\subsection{Preview}
In this section and the next we develop something which could be called a substitute for the Alexander trick
applicable to (some) configuration categories. It is a complicated substitute. The Alexander trick as we
understand it here concerns injective continuous maps from $D^m$ to $D^n$, where $m\le n$, which extend the standard embedding of $\partial D^m$
into $\partial D^n$. We say \emph{injective continuous map} in order to emphasize that we do not impose local
flatness conditions. Let $\imap_\partial(D^m,D^n)$ be the space of these injective continuous maps, with the compact-open topology.
The Alexander trick is a two-line shrinking argument showing that $\imap_\partial(D^m,D^n)$ is contractible.
An illuminating special case is the case $m=1$, $n=3$. (Knot theorists will remember that the knots and
especially the isotopies allowed in knot theory satisfy some local flatness conditions; therefore the existence of many
distinct knots does not imply that $\imap_\partial(D^1,D^3)$ has more than one path component.)

More specifically, in this section we break up the object space of $\config(D^m)$ and impose new conditions on morphisms which
express something resembling general position with respect to $0\in D^m$. The resulting Segal space is
$\varphi^*\config(D^m)$ which appears in proposition~\ref{prop-genpos}. (For the meaning of $\varphi$ see
definition~\ref{defn-fino}.) The main point is to show
that $\config(D^m)$ can be recovered from $\varphi^*\config(D^m)$ by a form of localization, a homotopy
invariant procedure. The name of the localization procedure
is \emph{conservatization}. Taken by itself this is not very hard to understand and it is interesting in its own right.
Therefore we begin with that.

\subsection{Conservative maps} In category theory, a functor is called \emph{conservative} if it takes non-isomorphisms
to non-isomorphisms. There is a similar concept for maps between simplicial spaces. (We re-discovered this for ourselves
and used the name \emph{face-local} for it in an earlier version of this article. We are grateful to Claudia Scheimbauer and Victoriya Ozornova for pointing out to us that the concept goes back at least as far as \cite{GalToKo} under the name \emph{conservative}.)

\begin{defn} {\rm Let $w\co X\to Z$ be a map of simplicial spaces. We say that $w$ is \emph{conservative},
or that \emph{$X$ is conservative over $Z$} if, for every monotone surjection $u\co [k]\to [\ell]$ in $\Delta$, the diagram
\[
	\begin{tikzpicture}[descr/.style={fill=white}, baseline=(current bounding box.base)] ]
	\matrix(m)[matrix of math nodes, row sep=2.5em, column sep=2.5em,
	text height=1.5ex, text depth=0.25ex]
	{
	X_\ell & X_k  \\
	Z_\ell  & Z_k \\
	};
	\path[->,font=\scriptsize]
		(m-1-1) edge node [auto] {$u^*$} (m-1-2)
		(m-2-1) edge node [auto] {$u^*$} (m-2-2)
		(m-1-1) edge node [left] {$w$} (m-2-1)
		(m-1-2) edge node [auto] {$w$} (m-2-2);
	\end{tikzpicture}
\]
is homotopy cartesian.
}
\end{defn}

\smallskip
\emph{Example.} In the case where $Z=\Delta[0]$, a simplicial space
$X$ over $Z$ is conservative if and only if all simplicial operators $X_s\to X_t$ are weak
equivalences.

\medskip
Let $\sC$ be a small (discrete) category and let $X$ be a fiberwise complete Segal space over the nerve $N\sC$,
with reference map $w\co X\to N\sC$.

\begin{lem} \label{lem-hensel} The following conditions on $w\co X\to N\sC$ are equivalent.
\begin{itemize}
\item[a)] An element $\gamma\in X_1$ is weakly invertible if and only if $w(\gamma)\in (N\sC)_1$ is an isomorphism.
\item[b)] $X$ is conservative over $N\sC$.
\end{itemize}
\end{lem}

\begin{proof} For an object $c$ of $\sC$, let $X(c)$ be the part of $X_0$ projecting to $c\in (N\sC)_0$. For a morphism $g$ in $\sC$,
let $X(g)$ be the part of $X_1$ projecting to $g\in (N\sC)_1$. Write $X(g)^\heq:=X(g)\cap X_1^\heq$~, where $X_1^\heq$
is the weakly invertible part of $X_1$. \newline
Assume a). Then for any invertible morphism $f\co c\to d$ in $\sC$ we have $X(f)=X(f)^\heq$. The fiberwise
completeness condition means that $d_0\co X(f)^\heq\to X(c)$ is a weak equivalence.
Therefore $d_0\co X(f)\to X(c)$ is a weak equivalence. It follows that $d_1\co X(f)\to X(c)$
is a weak equivalence. In particular this holds when $f=\id_c$\,.
Now take a $k$-simplex
\[
\sigma \quad = \quad \big( c_0 \xleftarrow{ \; f_1 \;} c_1 \xleftarrow{ \; f_2 \;} c_2 \xleftarrow{ \; \; \; \;} \cdots \xleftarrow{ \; f_k \;} c_k \big)
\]
in $N\sC$ and let $X(\sigma)$ be the part of $X_k$ projecting to $\sigma$. By the Segal condition
we can identify $X(\sigma)$ with the homotopy limit of
\[
X(c_0) \xleftarrow{ \; d_0 \;} X(f_1) \xrightarrow{ \; d_1 \;} X(c_1) \xleftarrow{ \; d_0 \;} X(f_2) \xrightarrow{ \; d_1 \;} X(c_2)
 \xleftarrow{ \; \; \; \;} \cdots \xleftarrow{ \; d_0 \;} X(c_k) \; .
\]
If one of the $f_i$ happens to be an identity morphism, then the corresponding arrows
$X(c_{i-1}) \leftarrow X(f_i)\to X(c_i)$ are weak equivalences as we have just shown. To rephrase this in a
more abstract manner, suppose that $\sigma=u^*\tau$ for some $\ell$-simplex $\tau$ in $N\sC$ and monotone
surjective $u\co [k]\to [\ell]$. Select a monotone right inverse $v\co [\ell]\to [k]$ for $u$, so that $uv=\id$.
Then the map $X(\sigma)\to X(\tau)$ obtained by restricting $v^*$ is a weak equivalence.
Consequently the map $X(\tau)\to X(\sigma)$ obtained by restricting $u^*$ is also a weak
equivalence. This is exactly property b). \newline
Now assume property b). Then for any object $c$ of $\sC$ the degeneracy map
$s_0$ from $X(c)$ to $X(\id_c)$ is a weak equivalence. It follows that $X(\id_c)$ is contained in $X_1^\heq$.
Next, let $f\co b\to c$ be an isomorphism in $\sC$ with inverse $g\co c\to b$.
Let $\bar f\in X(f)$. We need to show that $\bar f$ is weakly invertible. By fiberwise completeness,
there exists a weakly invertible $\bar g \in X(g)$ such that
$d_0\bar g$ and $d_1\bar f$ are in the same path component of $X(c)$. By the Segal property,
it follows that there exists $z\in X_2$ such that $d_0z$ is in the
path component of $\bar f$ and $d_2z$ is in the path component of $\bar g$. But $d_1z$\,,
which should be regarded as the composition of $d_0z$ and $d_2z$~, belongs to $X(\id_b)$
and therefore to $X_1^\heq$  by what we saw earlier.
It follows that $d_0z$ and with it $\bar f$ belong to $X_1^\heq$, too, by a two-out-of-three principle for weak
invertibility.
\end{proof}

\medskip
\emph{Example.} For any manifold $M$, the standard reference map from $\config(M)$ to $N\fin$ is conservative; and for any
manifold $M$ with boundary, the standard reference map from
$\config(M)$ to $N\finplus$ is conservative.

\subsection{A conservatization procedure}
We now describe a conservatization procedure which is universal in a derived sense. The input is a map $Y\to Z$ of simplicial
spaces, where $Z$ happens to be a simplicial \emph{set} viewed as a simplicial discrete space. The output of the
procedure is a commutative diagram of simplicial spaces
\begin{equation} \label{eqn-faceloc}
\begin{tikzpicture}[descr/.style={fill=white}, baseline=(current bounding box.base)],
	\matrix(m)[matrix of math nodes, row sep=2.5em, column sep=2.5em,
	text height=1.5ex, text depth=0.25ex]
	{
	 Y & Y^! & \Lambda Y \\
	  & Z & \\
	};
	\path[->,font=\scriptsize]
		(m-1-2) edge node [above] {$\simeq$} (m-1-1)
		(m-1-2) edge node [right] {} (m-1-3)
		(m-1-1) edge node [right] {} (m-2-2)
		(m-1-2) edge node [above] {} (m-2-2)
		(m-1-3) edge node [auto] {} (m-2-2);
\end{tikzpicture}
\end{equation}
where the map from $\Lambda Y$ to $Z$ is conservative. The map from
$Y^!$ to $Y$ is a degreewise weak equivalence. More precisely,
$Y\mapsto \Lambda Y$ and $Y\mapsto Y^!$ are endofunctors of the category of simplicial spaces over $B$. The
horizontal arrows in diagram~(\ref{eqn-faceloc}) are natural transformations.

\begin{defn} \label{defn-faceloc} {\rm The formula for $\Lambda Y$ is
\[ (\Lambda Y)_r:= \hocolimsub{[r]\to [k]\leftarrow [\ell]}  Y_\ell\times_{Z_\ell} Z_k \]
where the homotopy colimit is taken over the following category $\sE(r)$. An object is a diagram $[r]\to [k]\leftarrow [\ell]$
in $\Delta$ where the second arrow, from $[\ell]$ to $[k]$, is onto. A morphism is a commutative diagram
\[
\begin{tikzpicture}[descr/.style={fill=white}, baseline=(current bounding box.base)],
	\matrix(m)[matrix of math nodes, row sep=2.5em, column sep=2.5em,
	text height=1.5ex, text depth=0.25ex]
	{
	 {[r]} & {[k]} & {[\ell]} \\
	 {[r]} & {[k']} & {[\ell']} \\
	};
	\path[->,font=\scriptsize]
		(m-1-1) edge node [right] {} (m-1-2)
		(m-1-3) edge node [right] {} (m-2-3)
		(m-1-3) edge node [right] {} (m-1-2)
		(m-2-3) edge node [above] {} (m-2-2)
		(m-2-1) edge node [auto] {} (m-2-2)
		(m-1-2) edge node [auto] {} (m-2-2);
	\draw [double equal sign distance] (m-1-1) to (m-2-1) node [anchor=mid west] {$$};
\end{tikzpicture}
\]
in $\Delta$ (top row = source, bottom row = target). The reference map from $\Lambda Y$ to $Z$ is fairly obvious from
the definition of $\Lambda Y$~; it is the composition
\[ \hocolimsub{[r]\to [k]\leftarrow [\ell]}  Y_\ell\times_{Z_\ell} Z_k
\quad\lra \colimsub{[r]\to [k]\leftarrow [\ell]} Z_k \quad \lra \quad Z_r\,.
\]
The formula for $Y^!$ is
\[ (Y^!)_r= \hocolimsub{[r]\to [k]} Y_k = \hocolimsub{\rule{0mm}{4mm}[r]\to [k]\xleftarrow{=} [\ell]}  Y_\ell\times_{Z_\ell} Z_k\]
so that $Y^! \subset \Lambda Y$. For $r\ge 0$, the preferred map from $(Y^!)_r$ to $Y_r$ is the standard comparison map
from $\hocolim_{[r]\to [k]} Y_k$ to $Y_r\cong\colim_{[r]\to [k]} Y_k$.
}
\end{defn}

\begin{rem} \label{rem-locdetails} {\rm
For a simplicial set $Z$ let $\simp(Z)$
be the category of simplices of $Z$, as defined early on in section~\ref{subsec-dendroidal}.
A simplicial space $Y$ with a map to a simplicial set $Z$ is the same thing as a contravariant functor from $\simp(Z)$ to spaces.
We write $Y(z)$ for the value of that functor on the object $z\in Z_r$~. So $Y(z)$ is the part of $Y_r$ which is taken to
$z\in Z_r$ by the reference map $Y\to Z$.

For $x\in \simp(Z)$ let
$\sE(x)$ be the following category. An object of $\sE(x)$ is a diagram
$x\to y\leftarrow z$
in $\simp(Z)$ where the second arrow (from $z$ to $y$) is \emph{dominant}, i.e., the underlying morphism
in $\Delta$ is a (monotone) surjection (so that $z$ is a degeneracy of $y$). A morphism in $\sE(x)$ is a commutative
diagram in $\simp(Z) $ of the form
\[
\begin{tikzpicture}[descr/.style={fill=white}, baseline=(current bounding box.base)],
	\matrix(m)[matrix of math nodes, row sep=2.5em, column sep=2.5em,
	text height=1.5ex, text depth=0.25ex]
	{
	 x & y & z \\
	 x & y' & z' \\
	};
	\path[->,font=\scriptsize]
		(m-1-1) edge node [right] {} (m-1-2)
		(m-1-3) edge node [right] {} (m-2-3)
		(m-1-3) edge node [right] {} (m-1-2)
		(m-2-3) edge node [above] {} (m-2-2)
		(m-2-1) edge node [auto] {} (m-2-2)
		(m-1-2) edge node [auto] {} (m-2-2);
	\draw [double equal sign distance] (m-1-1) to (m-2-1) node [anchor=mid west] {$$};
\end{tikzpicture}
\]
where the rows qualify as objects of $\sE(x)$ (top row = source, bottom row = target).

Using this notation, we arrive at the following reformulation of definition~\ref{defn-faceloc}.
\begin{equation} \label{eqn-facelocsimp}  (\Lambda Y)(x) = \hocolimsub{x\to y\leftarrow z}  Y(z)
\end{equation}
for $x$ in $\simp(Z)$, where the homotopy colimit is taken over $\sE(x)$; in particular, $x\to y\leftarrow z$
describes an object of $\sE(x)$. In the same spirit
\[  Y^!(x) = \hocolimsub{ x\to y = z}  Y(z) \]
for $x$ in $\simp(Z)$.

For an object $x$ in $\simp(Z)$ let $\sE_0(x)$ be the full subcategory of $\sE(x)$ spanned by the objects
$x\to y\leftarrow z$ of $\sE(x)$ where both arrows are dominant. The inclusion $\sE_0(x)\to \sE(x)$ has a right adjoint.
It follows that the inclusion
\begin{equation} \label{eqn-facelocsimp2}  (\Lambda^\flat Y)(x)~:=~\hocolimsub{\twosub{x\to y\leftarrow z}{\textup{in }\sE_0(x)}}  Y(z)
\lra \hocolimsub{\twosub{x\to y\leftarrow z}{\textup{in }\sE(x)}}  Y(z)~~~=~~~(\Lambda Y)(x)
\end{equation}
is a weak equivalence. Indeed this inclusion has a preferred left inverse
\begin{equation} \label{eqn-facelocsimp3}
(\Lambda Y)(x)\lra (\Lambda^\flat Y)(x)
\end{equation}
(which is also a homotopy inverse) determined by the functor $\sE(x)\to \sE_0(x)$ right adjoint to the inclusion $\sE_0(x)\to \sE(x)$
and the counit morphisms of the adjunction. It follows that $\Lambda^\flat Y$ is again a contravariant functor on $\simp(Z)$: namely,
a morphism $f\co x_0\to x_1$ in $\simp(Z)$ determines a map $(\Lambda^\flat Y)(x_1)\to (\Lambda^\flat Y)(x_0)$ by pre-composing
and post-composing $f^*\co (\Lambda Y)(x_1)\to (\Lambda Y)(x_0)$ with the maps~(\ref{eqn-facelocsimp2}) for $x=x_1$
and~(\ref{eqn-facelocsimp3}) for $x=x_0$~, respectively. The maps~(\ref{eqn-facelocsimp3}) for arbitrary $x$ then amount
to a natural transformation $\Lambda Y\to \Lambda^\flat Y$ of contravariant functors on $\simp(Z)$ which is a degreewise
weak equivalence. So $\Lambda^\flat Y$ is an acceptable alternative to $\Lambda Y$. \newline
If $x$ happens to be a nondegenerate simplex of $Z$, then the only dominant
morphism in $\simp(Z)$ with source $x$ is $\id_x$\,. Then we get
\begin{equation} \label{eqn-facelocsimp4} (\Lambda^\flat Y)(x)~=~\hocolimsub{x\leftarrow z}  Y(z)
\end{equation}
where the homotopy colimit is taken over all dominant morphisms $x\leftarrow z$ (fixed target $x$) in $\simp(Z)$.
}
\end{rem}

\begin{lem} \label{lem-locadj1} $\Lambda Y$ is conservative over $Z$.
\end{lem}

\begin{proof} By remark~\ref{rem-locdetails}, it is enough to show that $\Lambda^\flat Y$ is conservative over $Z$.
For $x\in Z_r$ the category $\ms E_0(x)$ has a full subcategory $\ms E_1(x)$ consisting of the objects
$x\to y\leftarrow z$ where $y$ is a nondegenerate simplex. This condition determines
the arrow $x\to y$ because $x$ can be uniquely written in the form $g^*y$ for some nondegenerate simplex $y$ and
monotone surjection $g$. The argument also shows that the inclusion of $\ms E_1(x)$ in $\ms E_0(x)$ has a left adjoint,
say $\lambda$. This is given by passing from upper row to lower row in the following commutative diagram in $\simp(Z)$,
\[
\begin{tikzpicture}[descr/.style={fill=white}, baseline=(current bounding box.base)],
	\matrix(m)[matrix of math nodes, row sep=2.5em, column sep=2.5em,
	text height=1.5ex, text depth=0.25ex]
	{
	 x & y & z \\
	 x & y' & z \\
	};
	\path [-] ([xshift=0.9pt]m-1-1.south) edge node {} ([xshift=0.9pt]m-2-1.north);
	\path [-] ([xshift=-0.9pt]m-1-1.south) edge node {} ([xshift=-0.9pt]m-2-1.north);
	\path [-] ([xshift=0.9pt]m-1-3.south) edge node {} ([xshift=0.9pt]m-2-3.north);
	\path [-] ([xshift=-0.9pt]m-1-3.south) edge node {} ([xshift=-0.9pt]m-2-3.north);
	\path[->,font=\scriptsize]
		(m-1-1) edge node [right] {} (m-1-2);
	\path[->,font=\scriptsize] 		
		(m-1-3) edge node [right] {} (m-1-2);
	\path[->,font=\scriptsize] 		
		(m-2-3) edge node [right] {} (m-2-2);
	\path[->,font=\scriptsize]
		(m-2-1) edge node [auto] {} (m-2-2);
	\path[->,font=\scriptsize] 		
		(m-1-2) edge node [right] {} (m-2-2);
\end{tikzpicture}
\]
where the upper row represents a random object of $\ms E_0(x)$ and the rest of the diagram
is determined by the condition that $y'$ be nondegenerate and that $x\to y'$ be dominant. As
pre-composition with $\lambda\co \ms E_0(x)\to \ms E_1(x)$ does not affect the forgetful functor
\[  \big(x\to y\leftarrow z\big) \mapsto z \]
it follows that the inclusion
\[ \hocolimsub{\twosub{x\to y\leftarrow z}{\textup{in }\ms E_1(x)}}~Y(z) \quad
\vlra{2.0em}
\hocolimsub{\twosub{x\to y\leftarrow z}{\textup{in }\ms E_0(x)}}~Y(z)
\]
is also a weak equivalence \cite[Proposition A.4]{Dugger}.

The assignment $x\mapsto \ms E_1(x)$ is not contravariantly functorial on all of $\simp(Z)$,
but it is well behaved for any dominant morphism $w\to x$ in $\simp(Z)$.
In such a case it is clear that the induced functor $\ms E_1(x)\to \ms E_1(w)$ is an
isomorphism of categories, and consequently the resulting map
\[ \hocolimsub{\twosub{x\to y\leftarrow z}{\textup{in }\ms E_1(x)}}~Y(z) \quad
\vlra{2.0em}
\hocolimsub{\twosub{w\to y\leftarrow z}{\textup{in }\ms E_1(w)}}~Y(z)
\]
is a homeomorphism. By the commutativity of the diagram
\[
\begin{tikzpicture}[descr/.style={fill=white}, baseline=(current bounding box.base)],
	\matrix(m)[matrix of math nodes, row sep=2.5em, column sep=2.5em,
	text height=1.5ex, text depth=2.0ex]
	{
	{\hocolimsub{\ms E_0(x)} ...} &  {\hocolimsub{\ms E_0(w)}...} \\
	 {\hocolimsub{\ms E_1(x)}...} & {\hocolimsub{\ms E_1(w)}...} \\
	};
	\path[->,font=\scriptsize] 		
		(m-1-1) edge node [right] {} (m-1-2);
	\path[->,font=\scriptsize] 		
		(m-2-1.north) edge node [left] {$\simeq$} (m-1-1.south);
	\path[->,font=\scriptsize]
		(m-2-1) edge node [auto] {} (m-2-2);
	\path[->,font=\scriptsize] 		
		(m-2-2) edge node [right] {$\simeq$} (m-1-2);
\end{tikzpicture}
\]
this implies that the map $(\Lambda^\flat Y)(x)\to (\Lambda^\flat Y)(w)$ induced by the
dominant morphism  $w\to x$ in $\simp(Z)$ is a weak equivalence. \end{proof}

\begin{lem} \label{lem-locadj1plus} The preferred map $Y^!\to Y$ is a degreewise weak equivalence.
\end{lem}

\begin{proof} For $x\in Z_r$ let $\ms E_2(x)$ be the full subcategory of $\ms E(x)$
whose objects are the diagrams $x\to y\leftarrow z$
where the arrow from $z$ to $y$ is an identity in $\simp(Z)$, so that $y=z$. This is contravariantly functorial in $x$.
We defined $Y^!$ as a simplicial space over $Z$ by
\[ Y^!(x) :=\hocolimsub{\twosub{x\to y\leftarrow z}{\textup{in }\ms E_2(x)}}~Y(z). \]
Since $\ms E_2(x)$ has an initial object, the preferred map
from $Y^!(x)$ to $Y(x)$ is a weak equivalence. \end{proof}

\begin{lem} \label{lem-locadj2} If $Y$ is already conservative over $Z$, then the preferred map $Y^!\to \Lambda Y$ over $Z$ is a
degreewise weak equivalence.
\end{lem}

\begin{proof} For fixed $x$ in $\simp(Z)$, the inclusion of $\ms E_2(x)$ in $\ms E(x)$ has a left adjoint.
A morphism in $\ms E(x)$ given by a commutative diagram
\[
\begin{tikzpicture}[descr/.style={fill=white}, baseline=(current bounding box.base)],
	\matrix(m)[matrix of math nodes, row sep=2.5em, column sep=2.5em,
	text height=1.5ex, text depth=0.25ex]
	{
	 x & y & z \\
	 x & u & v \\
	};
	\path[->,font=\scriptsize]
		(m-1-1) edge node [right] {} (m-1-2);
	\path[->,font=\scriptsize]
		(m-1-1) edge node [right] {$id$} (m-2-1);
	\path[->,font=\scriptsize]
		(m-1-3) edge node [right] {} (m-2-3);
	\path[->,font=\scriptsize] 		
		(m-1-3) edge node [right] {} (m-1-2);
	\path[->,font=\scriptsize] 		
		(m-2-3) edge node [above] {$b$} (m-2-2);
	\path[->,font=\scriptsize]
		(m-2-1) edge node [auto] {} (m-2-2);
	\path[->,font=\scriptsize] 		
		(m-1-2) edge node [auto] {$a$} (m-2-2);
\end{tikzpicture}
\]
in $\simp(Z)$ is a unit morphism for the adjunction if and only if the arrows labeled $a$ and $b$ are isomorphisms
(alias identities). Then the arrow from $z$ to $v$ must be dominant, and so the induced map
$Y(v)\to Y(z)$ is a weak equivalence by assumption on $Y$. It follows
that the inclusion
\[
\hocolimsub{\twosub{x\to y\leftarrow z}{\textup{in }\ms E_2(x)}}~Y(z) \quad
 \vlra{2.0em}
\hocolimsub{\twosub{x\to y\leftarrow z}{\textup{in }\ms E(x)}}~Y(z)
\]
is a degreewise weak equivalence.
\end{proof}

\begin{lem} \label{lem-locadj3} For any simplicial space $Y$ over $Z$,  the inclusion $Y^!\to \Lambda Y$ induces a
degreewise weak equivalence from $\Lambda(Y^!)$ to $\Lambda(\Lambda Y)$.
\end{lem}

\begin{proof} Let $x$ be an object in $\simp(Z)$.
We want to show that the inclusion of $(\Lambda(Y^!))(x)$ in  $(\Lambda(\Lambda Y))(x)$ is a weak
equivalence. Since $\Lambda(Y^!)$ and $\Lambda(\Lambda Y)$ are both conservative over $Z$, there is no loss of generality
in assuming that $x$ is nondegenerate. In forming $\Lambda(Y^!)(x)$ from  $Y^!$ and in the outside application of $\Lambda$
to form $(\Lambda(\Lambda(Y)))(x)$ from $\Lambda(Y)$, we can replace the indexing category
$\ms E(x)$ by the less functorial $\ms E_1(x)$; see the
proof of lemma~\ref{lem-locadj1}. Then we are looking at the inclusion of
\[ \hocolimsub{x\leftarrow z} Y^!(z)
=\hocolimsub{x\leftarrow z} \hocolimsub{z\to z'} Y(z')~, \]
an acceptable substitute for $(\Lambda(Y^!))(x)$, in
\[  \hocolimsub{x\leftarrow z} (\Lambda Y)(z)
=\hocolimsub{x\leftarrow z} \hocolimsub{z\to z'\leftarrow z''} Y(z'')~,
\]
an acceptable substitute for $(\Lambda(\Lambda Y))(x)$. All arrows
pointing from right to left are dominant.
These double homotopy colimits can be replaced by the weakly equivalent
\[ \hocolimsub{x\leftarrow z\to z'} Y(z') \quad \textup{and} \;
\hocolimsub{x\leftarrow z\to z'\leftarrow z''} Y(z''),
\]
respectively.
The weak homotopy type of the left-hand expression is not affected if
we add the condition (under the hocolim sign) that $z\to z'$ be an identity morphism, because the inclusion
of the full subcategory (of the appropriate indexing category)
defined by that condition has a right adjoint. By the same reasoning, the weak
homotopy type of the right-hand expression is not affected if
we add the condition (under the hocolim sign) that $z\to z'$ be dominant.
Now we are looking at the inclusion
\[  \hocolimsub{x\leftarrow z}  Y(z)
\quad  \vlra{2.0em} \hocolimsub{x\leftarrow z\to z'\leftarrow z''} Y(z'')
\]
(where all arrows $\leftarrow$ and $\rightarrow$ are dominant)
obtained by viewing diagrams $x\leftarrow z$ as diagrams of the form $x\leftarrow z\to z'\leftarrow z''$
where $z''=z'=z$. Note that in the right-hand side under the hocolim sign, the simplices $z$, $z'$ and $z''$
are degeneracies of the nondegenerate $x$ (on the understanding that a degeneracy of $x$ need not be distinct from $x$).
Therefore another simplification can be
made here that does not affect the weak homotopy type. Namely, we can assume that the arrow $z\to z'$ is an identity.
(The inclusion of the subcategory defined by that condition has a left adjoint which replaces
a diagram $x\leftarrow z\to z'\leftarrow z''$ by
\[
x \vlla{1.0em} z' \xvlra{1.0em}{=} z' \vlla{1.0em} z^{''}
\]
To make sense
of the prescription, observe that the arrow from $z$ to $x$ in the diagram
$x\leftarrow z\to z'\leftarrow z''$ must factor uniquely through the arrow from $z$ to $z'$ in the same diagram.)
Now we are looking at the inclusion
\[ \hocolimsub{x\leftarrow z}  Y(z)
\;
\vlra{2.0em}
\hocolimsub{x\leftarrow z\leftarrow z''} Y(z'')
\]
where all arrows under the hocolim signs are dominant.
We make yet another simplification under the hocolim sign on the right
that does not affect the weak homotopy type. Namely, we can assume that the arrow from $z$ to $x$ is an identity.
(The inclusion of the full subcategory defined by that condition has a left adjoint which takes a diagram
$x\leftarrow z\leftarrow z''$ to $x\leftarrow x\leftarrow z''$.) After that last simplification, our comparison
map is an identity map. As such it is a weak equivalence. \end{proof}

\medskip
Lemmas~\ref{lem-locadj1},~\ref{lem-locadj1plus},~\ref{lem-locadj2} and~\ref{lem-locadj3} justify the statement that $\Lambda$
is a derived left adjoint for the inclusion of the full subcategory of the conservative simplicial spaces over $Z$
in the category of all simplicial spaces over $Z$.

\begin{rem} \label{rem-pinought} {\rm In diagram~(\ref{eqn-faceloc}), the map from $\pi_0 Y^!_0\cong \pi_0 Y_0$ to $\pi_0((\Lambda Y)_0)$
determined by the horizontal arrows is surjective. This can be seen from the description
\[  (\Lambda Y)_0 \simeq \hocolimsub{[0]\to[0]\leftarrow[\ell]} Y_\ell\times_{Z_\ell} Z_0~. \]
The indexing category is identified with $\Delta$ and the hocolim can be interpreted as geometric realization.
}
\end{rem}

\begin{defn} {\rm Let $X$ and $Y$ be simplicial spaces over $Z$, a simplicial set. A map $g\co X\to Y$ over $Z$
is a \emph{conservatization map} if $Y$ is conservative over $Z$ and
the map $\Lambda g\co \Lambda(X)\to \Lambda(Y)$ induced by $g$ is a degreewise weak equivalence.
}
\end{defn}

\subsection{Property beta} \label{subsec-trans}
Let $\sC$ be a small (discrete) category and put $Z=N\sC$.
In this section we give some conditions which ensure that the Segal property and fiberwise completeness
are preserved under conservatization $\Lambda$ relative to $Z$.

\begin{lem} \label{lem-faceloccplt} If $Y$ is a fiberwise complete Segal space over $Z=N\sC$ and if $\Lambda Y$ is a Segal space, then
$\Lambda Y$ is also fiberwise complete over $Z$.
\end{lem}

\begin{proof} Let $f\co b\to c$ be an isomorphism in $\sC$. If $f$ is an identity morphism, then
$d_1, d_0 : (\Lambda Y)(f) \to (\Lambda Y)(b)$ are weak equivalences since $\Lambda Y$ is conservative. It suffices to show that
\begin{itemize}
\item[($*$)] $d_0, d_1 : (\Lambda Y)(f) \to (\Lambda Y)(b)$ are weak equivalences, even if $f$ is not the identity, just an isomorphism.
\end{itemize}
Indeed, $(*)$ implies that for any $\bar{f}$ in $(\Lambda Y)(f)$ we can find $\bar{g}$ in $\Lambda Y(g)$ where $g = f^{-1}$ such that $\bar{f}$ and $\bar{g}$ are weakly composable (in $\Lambda Y$). Then $\bar{g} \bar{f}$ is in $\Lambda Y(\id_b)$ and is therefore weakly invertible by the special case of $(*)$ when $f = \id$. This implies that $\bar{f}$ has a weak left inverse. Similarly, $\bar{f}$ has a weak right inverse. It follows that $(\Lambda Y)^{\heq}(f) = (\Lambda Y)(f)$, and so $(*)$ implies that $\Lambda Y$ is fiberwise complete over $Z$.

Now we prove ($*$). Take $\bar{c}$ in $\Lambda Y(c)$ and consider the homotopy fiber of $$d_1 \co (\Lambda Y)(f) \to (\Lambda Y)(c)$$ over $\bar{c}$. We claim there exists $\bar{f}$ in that homotopy fiber which is weakly invertible as an element of $\Lambda Y (f)$. By remark \ref{rem-pinought}, we can assume that $\bar{c}$ is in $Y(c)$ and, consequently, it is sufficient to search for a weakly invertible $\bar{f}$ in the homotopy fiber over $\bar{c}$ of $d_1 \co Y(f) \to Y(c)$. This is easy because $Y$ is fiberwise complete and $f$ is an isomorphism.

Let $\bar{b}$ be the source of $\bar{f}$ in $\Lambda Y(b)$. Composition with $\bar{f}$ is a weak equivalence from the homotopy fiber over $\bar{b}$ of $d_1 \co (\Lambda Y)(id_b) \to (\Lambda Y)(b)$ to the homotopy fiber over $\bar{c}$ of $d_1 \co (\Lambda Y)(f) \to (\Lambda Y)(c)$. Since the first homotopy fiber is weakly contractible, so is the second.

This argument shows that $d_1$ in $(*)$ is a weak equivalence. A similar argument works for $d_0$.
\end{proof}

\medskip

Again let $\sC$ be a small (discrete) category and put $Z=N\sC$. Let $Y$ be a fiberwise
complete Segal space over $Z$, with reference map $w\co Y\to Z$. For $y \in Y_0$, we denote by $(Y\downarrow y)$ the Segal space whose space of $n$-simplices is the homotopy fiber of the ultimate target map $Y_{n+1} \to Y_0$ over $y$. It is a Segal space over $Z$ by the composite $Y_{n+1} \to Y_{n} \to Z_n$ where the first map is $d_0$.

\begin{defn} \label{defn-transp} {\rm We say that $w\co Y\to Z$ has
\emph{property beta} if, for $f\in Y_1$ with source $x=d_0f\in Y_0$ and target $y=d_1f\in Y_0$
such that $w(f)\in Z_1=(N\sC)_1$ is an identity morphism, the map
$\Lambda(Y\downarrow x) \to \Lambda(Y\downarrow y)$
induced by $f$ is a weak equivalence in degree 0.
}
\end{defn}

\begin{thm} \label{thm-transp} Under conditions as in definition~\emph{\ref{defn-transp}}, if $Y\to Z$ has property beta,
then $\Lambda Y$ is a fiberwise complete Segal space over $Z$.
\end{thm}

\begin{proof} The following general principle (gpr) will be used several times. Let $\sA$ be a small category, let $F_1$ and $F_2$ be
functors from $\sA$ to spaces and let $\nu\co F_1\to F_2$ be a natural transformation. If for
every morphism $g\co a\to b$ in $\sA$ the square
\[
	\begin{tikzpicture}[descr/.style={fill=white}, baseline=(current bounding box.base)] ]
	\matrix(m)[matrix of math nodes, row sep=2.5em, column sep=2.5em,
	text height=1.5ex, text depth=0.25ex]
	{
	F_1(a) & F_2(a)  \\
	F_1(b)  & F_2(b) \\
	};
	\path[->,font=\scriptsize]
		(m-1-1) edge node [auto] {$\nu$} (m-1-2)
		(m-2-1) edge node [auto] {$\nu$} (m-2-2)
		(m-1-1) edge node [left] {$g_*$} (m-2-1)
		(m-1-2) edge node [auto] {$g_*$} (m-2-2);
	\end{tikzpicture}
\]
is homotopy cartesian, then for every object $a$ in $\sA$ the square
\[
	\begin{tikzpicture}[descr/.style={fill=white}, baseline=(current bounding box.base)] ]
	\matrix(m)[matrix of math nodes, row sep=2.5em, column sep=2.5em,
	text height=1.5ex, text depth=0.25ex]
	{
	F_1(a) & F_2(a)  \\
	\hocolim~F_1  & \hocolim~F_2 \\
	};
	\path[->,font=\scriptsize]
		(m-1-1) edge node [auto] {$\nu$} (m-1-2)
		(m-2-1) edge node [auto] {$\nu$} (m-2-2)
		(m-1-1) edge node [left] {} (m-2-1)
		(m-1-2) edge node [auto] {} (m-2-2);
	\end{tikzpicture}
\]
is homotopy cartesian. This statement appears in \cite{Puppe}, and the proof there uses quasi-fibration theory.

Now select $\sigma\in Z_r$ corresponding to a diagram
\[
c_0 \xleftarrow{ \; f_1 \;} c_1 \xleftarrow{ \; f_2 \;} c_2 \xleftarrow{ \; \; \; \;} \cdots \xleftarrow{ \; \; \; \;} c_{r-1} \xleftarrow{ \; f_r \;} c_r
\]
in $\sC$; recall that $Z=N\sC$. We want to investigate the homotopy fibers of the map
\begin{equation} \label{eqn-beta1} (\Lambda Y)(\sigma) \lra (\Lambda Y)(d_r\sigma) \end{equation}
induced by $d_r$~, using (gpr).

Assume to begin with that $\sigma$ is nondegenerate. (Then
$d_r\sigma\in Z_{r-1}$ and $f_r\in Z_1$ are also nondegenerate.) We can use the alternative definition
\[  (\Lambda Y)(\sigma) = \hocolimsub{q\co [k]\to [r]} Y(q^*\sigma) \]
of~(\ref{eqn-facelocsimp4}), where $q$ runs through the surjective monotone maps with fixed target
$[r]$ and variable source $[k]$. These maps $q$ are the objects of a small category $\sD_r$ where the
morphisms are commutative triangles. We can write
\[ \sD_r\cong \sD_{r-1}\times\sD_0~, \]
roughly by restricting
a monotone surjection $[k]\to [r]$ to the preimages of the subsets $[r-1]$ and $\{r\}$ of $[r]$.
This leads to a projection functor $\sD_r\to \sD_{r-1}$.
Therefore~(\ref{eqn-beta1}) can be recast in the form
\begin{equation} \label{eqn-beta2}
\hocolimsub{(g,h)\textup{ in }\sD_{r-1}\times\sD_0} Y((g\sqcup h)^*\sigma) \lra
\hocolimsub{g\textup{ in }\sD_{r-1}} Y(g^*d_r\sigma)
\end{equation}
where $g\sqcup h$ denotes the juxtaposition (an object in $\sD_r$) corresponding to $q$.
Next, we can write the hocolim on the left as a hocolim taken over $\sD_{r-1}$
using the homotopy left Kan extension along the projection $\sD_{r-1}\times \sD_0\to \sD_{r-1}$.
After a routine simplification we obtain the description
\begin{equation} \label{eqn-beta3}
\hocolimsub{g\textup{ in }\sD_{r-1}} \hocolimsub{h\in \sD_0} Y((g\sqcup h)^*\sigma) \lra
\hocolimsub{g\textup{ in }\sD_{r-1}} Y(g^*d_r\sigma).
\end{equation}
Fix $g:[k]\to [r-1]$ in $\sD_{r-1}$ for the time being. For $h\co [\ell]\to [0]$ in $\sD_0$ write $\id_0\sqcup h\co [\ell+1]\to [1]$
for the map taking $0$ to $0$ and all other elements to $1$. There is a commutative square
\begin{equation} \label{eqn-prebeta4}
\begin{tikzpicture}[descr/.style={fill=white}, baseline=(current bounding box.base)] ]
	\matrix(m)[matrix of math nodes, row sep=2.5em, column sep=2.5em,
	text height=1.5ex, text depth=0.25ex]
	{
	{\hocolimsub{h\in \sD_0} \rule{0mm}{6mm} Y((g\sqcup h)^*\sigma)} & {\hocolimsub{h\in \sD_0} \rule{0mm}{6mm} Y(g^*d_r\sigma)}  \\
	{\hocolimsub{h\in \sD_0} \rule{0mm}{6mm} Y((\id_0\sqcup h)^* f_r)} & {\hocolimsub{h\in \sD_0} \rule{0mm}{6mm} Y(c_{r-1})}  \\
	};
	\path[->,font=\scriptsize]
		(m-1-1) edge node [auto] {} (m-1-2)
		(m-2-1) edge node [auto] {} (m-2-2)
		(m-1-1) edge node [left] {} (m-2-1)
		(m-1-2) edge node [auto] {} (m-2-2);
	\end{tikzpicture}
\end{equation}
whose arrows are induced by the monotone maps in
\[
\begin{tikzpicture}[descr/.style={fill=white}, baseline=(current bounding box.base)] ]
	\matrix(m)[matrix of math nodes, row sep=2.5em, column sep=3.5em,
	text height=1.5ex, text depth=0.25ex]
	{
	{[k+\ell+1]} & {[k]}  \\
	{[\ell+1] } & {[0]}  \\
	};
	\path[<-,font=\scriptsize]
		(m-1-1) edge node [auto] {$x\mapsto x$} (m-1-2)
		(m-2-1) edge node [auto] {$x\mapsto x$} (m-2-2)
		(m-1-1) edge node [left] {$x\mapsto x+k$} (m-2-1)
		(m-1-2) edge node [auto] {$x\mapsto x+k$} (m-2-2);
	\end{tikzpicture}
\]
The square (\ref{eqn-prebeta4}) can be seen to be homotopy cartesian by applying (gpr) to each column, using the fact that $Y$ is a Segal space. (For the left-hand column, set $F_1(h) = Y((g \sqcup h)^*\sigma)$ and $F_2(h) = Y((\id_0 \sqcup h)^*f_r$ in (gpr), for $h$ in $\sD_0^\op$. For the right-hand column, we are looking at constant functors on $\sD_0^\op$.) The classifying space of $\sD_0$ is contractible, so that the right-hand column simplifies and we get a homotopy cartesian square
\begin{equation} \label{eqn-beta4}
	\begin{tikzpicture}[descr/.style={fill=white}, baseline=(current bounding box.base)] ]
	\matrix(m)[matrix of math nodes, row sep=2.5em, column sep=2.5em,
	text height=1.5ex, text depth=0.25ex]
	{
	{\hocolimsub{h\in \sD_0} \rule{0mm}{6mm} Y((g\sqcup h)^*\sigma)}& Y(g^*d_r\sigma)  \\
	{\hocolimsub{h\in \sD_0} \rule{0mm}{6mm} Y((\id_0\sqcup h)^* f_r)}  & Y(c_{r-1}) \;.  \\
	};
	\path[->,font=\scriptsize]
		(m-1-1) edge node [auto] {} (m-1-2)
		(m-2-1) edge node [auto] {} (m-2-2)
		(m-1-1) edge node [left] {} (m-2-1)
		(m-1-2) edge node [auto] {} (m-2-2);
	\end{tikzpicture}
\end{equation}

It follows that the top horizontal homotopy fiber over $\bar x \in Y(g^*d_r\sigma)$ is identified with the lower horizontal homotopy fiber over the image $x\in Y(c_{r-1})$ , which is further identified with
\[ \hocolimsub{h\in \sD_0} \hofiber_x[\,Y((\id_0\sqcup h)^*f_r) \to Y(c_{r-1})\,]. \]
since homotopy colimits are stable under homotopy base change. In other words, it is the part of $\Lambda(Y\downarrow x)_0$ projecting to the 0-simplex $f_r$ in $N(\finplus\downarrow c_{r-1})$.

Now we return to ~(\ref{eqn-beta3}). Letting $g$ vary in square (\ref{eqn-beta4}), the resulting map between horizontal homotopy fibers (corresponding to a morphism from $g$ to $g^\prime$ in $\sD_{r-1}$, say) takes the form $\Lambda(Y\downarrow x)_0 \to \Lambda(Y\downarrow x^\prime)_0$ (restricted to the parts over $f_r$) induced by a morphism $x \to x^\prime$ which covers the identity of $c_{r-1}$ in $\finplus$. Property beta means that
this map is a weak equivalence. This means that we can apply (gpr) once more, now to the functors
\[
F_1(g) = {\hocolimsub{h\in \sD_0} \rule{0mm}{6mm} Y((g\sqcup h)^*\sigma)} \quad \mbox{ and } \quad F_2(g) = Y(g^*d_r\sigma)
\]
for $g \in \sD_{r-1}^\op$. Therefore the homotopy fiber of~(\ref{eqn-beta3}) or
equivalently of~(\ref{eqn-beta1}) over $\bar x$ can again be identified with the part of $\Lambda(Y\downarrow x)$
projecting to the 0-simplex $f_r$ in $N(\sC \downarrow c_{r-1})$. Since this argument works for $\sigma\in (N\sC)_r$
as well as for $f_r\in (N\sC)_1$, it follows that the square
\begin{equation} \label{eqn-segaldia}
	\begin{tikzpicture}[descr/.style={fill=white}, baseline=(current bounding box.base)] ]
	\matrix(m)[matrix of math nodes, row sep=2.5em, column sep=2.5em,
	text height=1.5ex, text depth=0.25ex]
	{
	(\Lambda Y)(\sigma) & (\Lambda Y)(d_r\sigma)  \\
	(\Lambda Y)(f_r)   & (\Lambda Y)(c_{r-1})  \\
	};
	\path[->,font=\scriptsize]
		(m-1-1) edge node [auto] {$d_r$} (m-1-2)
		(m-2-1) edge node [auto] {$d_1$} (m-2-2)
		(m-1-1) edge node [left] {$(d_0)^{r-1}$} (m-2-1)
		(m-1-2) edge node [auto] {$(d_0)^{r-1}$} (m-2-2);
	\end{tikzpicture}
\end{equation}
is homotopy cartesian.

\emph{The case where $\sigma$ is degenerate.} If it so happens that $f_r$ is an identity morphism,
then the two horizontal arrows in diagram~(\ref{eqn-segaldia}) are weak equivalences since $\Lambda Y$
is conservative over $Z$. Then diagram~(\ref{eqn-segaldia}) is certainly homotopy cartesian.
Next, assume that $f_r$ is not an identity morphism. There exist a monotone surjective map $v\co [r]\to [q]$
and a \emph{nondegenerate} $q$-simplex $\tau$ corresponding to a diagram
\[
b_0 \xleftarrow{ \; g_1 \;} b_1 \xleftarrow{ \; g_2 \;} b_2 \xleftarrow{ \; \; \; \;} \cdots \xleftarrow{ \; \; \; \;} b_{q-1} \xleftarrow{ \; g_q \;} b_q
\]
in $\sC$ such that $\sigma=v^*\tau$. The commutative square~(\ref{eqn-segaldia}) is the target of a map (natural transformation)
induced by $v$; the source is a similar square
\begin{equation} \label{eqn-segaldia2}
	\begin{tikzpicture}[descr/.style={fill=white}, baseline=(current bounding box.base)] ]
	\matrix(m)[matrix of math nodes, row sep=2.5em, column sep=2.5em,
	text height=1.5ex, text depth=0.25ex]
	{
	(\Lambda Y)(\tau) & (\Lambda Y)(d_q\tau)  \\
	(\Lambda Y)(g_q)  & (\Lambda Y)(b_{q-1})  \\
	};
	\path[->,font=\scriptsize]
		(m-1-1) edge node [auto] {$d_q$} (m-1-2)
		(m-2-1) edge node [auto] {$d_1$} (m-2-2)
		(m-1-1) edge node [left] {$(d_0)^{q-1}$} (m-2-1)
		(m-1-2) edge node [auto] {$(d_0)^{q-1}$} (m-2-2);
	\end{tikzpicture}
\end{equation}
which we already know to be homotopy cartesian.
That map from square~(\ref{eqn-segaldia2}) to square~(\ref{eqn-segaldia}) is a termwise weak equivalence because
$\Lambda Y$ is conservative over $Z$.
It follows that square~(\ref{eqn-segaldia}) is homotopy cartesian, too. ---
Therefore square~(\ref{eqn-segaldia}) is homotopy cartesian for all $\sigma\in (NZ)_r$. This means that
$\Lambda Y$ has the Segal property.
\end{proof}

There is a canonical map from
$\Lambda(Y\downarrow x)$ to $(\Lambda Y\downarrow x)$ which is in reality a diagram of the form
\[
\Lambda(Y\downarrow x) \xleftarrow{\; \simeq \;} \Lambda(Y^!\downarrow x^!)  \xrightarrow{\; \simeq \;} \Lambda(\Lambda Y\downarrow x^!)  \xleftarrow{\; \simeq \;} (\Lambda Y\downarrow x^!)^! \xrightarrow{\; \simeq \;} (\Lambda Y\downarrow x^!).
\]
The $(-)^!$ notation is explained near diagram~(\ref{eqn-faceloc}). We are assuming, without loss of generality, that $x\in Y_0$ is the
image of $x^!\in Y^!_0$\,. The second arrow is induced by the inclusion
$Y^!\to \Lambda Y$. The third arrow is a weak equivalence because $(\Lambda Y\downarrow x^!)$
is already conservative over $Z$. The proof above shows that the decisive second arrow is a weak equivalence in degree $0$:
the homotopy fiber of (\ref{eqn-beta3}) is identified with $(\Lambda Y \downarrow x)_0$ when $r = 1$. Hence
the sentence \emph{Therefore the homotopy fiber of~\emph{(\ref{eqn-beta3})}\dots} just before (\ref{eqn-segaldia}) implies the following.


\begin{prop}
For any $x\in Y_0$ the canonical map from $\Lambda(Y\downarrow x)$ to $(\Lambda Y\downarrow x)$ is a weak
equivalence in degree $0$.
\end{prop}

\subsection{Examples and applications} The examples in this (sub)section
culminate in the proof of proposition~\ref{prop-genpos} below.

\begin{defn} \label{defn-fino} {\rm Let $\fino$ be the following category. An object is an object $[k]$ of $\finplus$
together with a distinguished element $a\in\uli k\subset[k]$ and a choice of color $c$, black or white. (It follows that $k>0$.
Sometimes we think of the color $c$ as a property of the distinguished element $a$, which we may accordingly call a white hole or a black hole.)
The color choice \emph{white} is only allowed if $a=k \in \uli k$.
A morphism from $([k],a_0,c_0)$ to $([\ell],a_1,c_1)$ is a morphism $[k]\to[\ell]$ in $\finplus$ subject to the following conditions:
\begin{itemize}
\item[{}] $f(a_0)=a_1$~;
\item[{}] if $c_1=$ white, then $c_0=$ white and $f(x)\ne a_1$ whenever $x\ne a_0$.
\end{itemize}
Composition of morphisms is as in $\finplus$\,, so that
there is a forgetful functor
\[ \varphi\co \fino\to \finplus\,. \]
There is another functor $\delta\co \fino \to \finplus$
which takes an object $([k],a,c)$ to $[k]$ if $c=$ black, and to
$[k-1]$ if $c=$ white. A morphism
from $([k],a,c)$ to $([\ell],a',c')$
in $\fino$ determines a morphism in $\finplus$ from $\delta([k],a,c)$ to $\delta([\ell],a',c')$
forgetfully and by restriction where applicable. The functors $\delta$ and $\varphi$ are related
by a natural transformation $\delta\to \varphi$ which can be regarded as an inclusion.
(In other words there is a standard morphism $\delta([k],a,c)\to \varphi([k],a,c)$ which
is the identity of $[k]$ if $c=$ black, and the inclusion of $[k-1]$ in $[k]$
if $c=$ white.) \newline
The functor $\delta$ has a left adjoint $\alpha\co \finplus\to \fino$ which identifies $\finplus$ with the full
subcategory of $\fino$ consisting of the white objects. It is given on objects
by $\alpha([k]) = ([k+1],k+1,\textup{white})$. We have $\delta\alpha=\id$. \newline
\emph{Mnemonic}: $\varphi$ forgets, $\alpha$ adds, $\delta$ deletes.
}
\end{defn}

\begin{rem} {\rm The functor $\delta$ makes $N\fino$ into a fiberwise complete Segal space over $N\finplus$.
(But the functor $\varphi$ does not make $N\fino$ into a fiberwise complete Segal space
over $N\finplus$.)}
\end{rem}

\begin{rem} \label{rem-suck} {\rm The functor $\delta\co N\fino\to N\finplus$ does not make
$N\fino$ conservative over $N\finplus$\,. This is related to lemma~\ref{lem-hensel}. There are a few morphisms in $\fino$
which are not isomorphisms, but turn into isomorphisms on applying $\delta$. These are the morphisms of the form
$([k],k,\textup{white}) \to ([k-1],a,\textup{black})$
where the underlying map $f$ is surjective and $f^{-1}(a)$ has exactly two elements.
}
\end{rem}

\medskip
Let $M$ be a manifold with boundary. We write $\varphi^*\config(M)$ for the pullback
of $\config(M)\to N\finplus$ along the map $\varphi\co N\fino\to N\finplus$.

It may help to think of objects (alias $0$-simplices) of $\varphi^*\config(M)$ in the following way.
An object is a pair consisting essentially of an ordered configuration in $M_-=M\smin\partial M$
and a point $p$ in $M_-$\,. We want a clear decision on whether $p$ is in the configuration or not.
Therefore the two possible answers correspond to a splitting of the object space into two disjoint
topological summands, the \emph{black} summand and the \emph{white} summand. Specifically,
if an object of $\varphi^*\config(M)$ is taken to an object in $\fino$ of the form $([k],a,\textup{black})$,
then we are dealing with an ordered configuration of $k$ points in $M_-$ and a point
$p\in M_-$ which is \emph{in} the configuration and as such occupies the number $a\in \uli k$\,. If the
underlying object in $\fino$ has the form $([k],k,\textup{white})$
then we are dealing with an ordered configuration of $k-1$ points in $M_-$ (corresponding to labels $1,2,\dots,k-1$)
and a point $p\in M_-$ which is not in the configuration. 

\begin{cor} The composition $\varphi^*\config(M)\to N\fino\xrightarrow{\delta} N\finplus$
makes $\varphi^*\config(M)$ into a fiberwise complete Segal space over $N\finplus$. \emph{(We tend to write
$\delta_*\varphi^*\config(M)$ for the simplicial space $\varphi^*\config(M)$ if we have this
specific reference map to $N\finplus$ in mind.)}
\end{cor}

\begin{proof} Since $\config(M)$
is a fiberwise complete Segal space over $N\finplus$\,, it follows that $\varphi^*\config(M)$
is a fiberwise complete Segal space over $N\fino$. But $N\fino$ is a fiberwise complete Segal space
over $N\finplus$ by means of $\delta$. \end{proof}

\begin{rem} \label{rem-rw} {\rm Proposition~\ref{prop-genpos} below is an analogue for configuration categories of an easy
observation concerning \emph{ordered configuration spaces} which we saw in \cite{Wahl}, attributed to Randal-Williams.
Let $M$ be an $m$-manifold without boundary for simplicity, $p\in M$ and $k\ge 1$.
Then $\emb(\,\uli k\,,M)$ is a manifold $A$. It has a submanifold $C$, closed as a subset of $A$,
consisting of the embeddings $\uli k\to M$ which have $p$ in their image.
The submanifold $C$ has a trivializable normal disk bundle with fiber $D_p$~, where $D_p\subset M$ is a compact
disk about $p$. This leads to a homotopy pushout square
\[
	\begin{tikzpicture}[descr/.style={fill=white}, baseline=(current bounding box.base)] ]
	\matrix(m)[matrix of math nodes, row sep=2.5em, column sep=2.5em,
	text height=1.5ex, text depth=0.25ex]
	{
	\partial D_p\times C & C  \\
	A\smin C  & A  \\
	};
	\path[->,font=\scriptsize]
		(m-1-1) edge node [auto] {\textup{proj.}} (m-1-2)
		(m-2-1) edge node [auto] {} (m-2-2)
		(m-1-1) edge node [left] {} (m-2-1)
		(m-1-2) edge node [auto] {} (m-2-2);
	\end{tikzpicture}
\]
which upon decoding turns into
\begin{equation} \label{eqn-rwpix}
	\begin{tikzpicture}[descr/.style={fill=white}, baseline=(current bounding box.base)] ]
	\matrix(m)[matrix of math nodes, row sep=2.5em, column sep=2.5em,
	text height=1.5ex, text depth=0.25ex]
	{
	 \partial D_p\times \coprod_{j=1}^k \emb(\,\uli k\smin j\,,M\smin p) & \coprod_{j=1}^k  \emb(\,\uli k\smin j\,,M\smin p)  \\
	\emb(\,\uli k\,,M\smin p)  & \emb(\,\uli k\,,M) \;.  \\
	};
	\path[->,font=\scriptsize]
		(m-1-1) edge node [auto] {\textup{proj.}} (m-1-2)
		(m-2-1) edge node [auto] {\textup{incl.}} (m-2-2)
		(m-1-1) edge node [left] {} (m-2-1)
		(m-1-2) edge node [auto] {$\textup{extend by }j\mapsto p$} (m-2-2);
	\end{tikzpicture}
\end{equation}
(The square is commutative up to a preferred homotopy, and as such it is a homotopy pushout square.
We have written $M\smin p$ for $M\smin\{p\}$ and the like.)
This means roughly that the ordered configuration spaces of $M$ can be obtained, up to weak equivalence, from the ordered
configuration spaces of $M\smin p$ by a homotopy pushout construction (which also involves $\partial D_p$).
}
\end{rem}

We return to the more general situation where $M$ is an $m$-manifold with boundary.
There is a canonical map
\[ u\co \delta_*\varphi^*\config(M) \lra \config(M) \]
over $N\finplus$. (Equivalently, $u$ can be described as a map $\varphi^*\config(M)\to \config(M)$ which covers
$\delta\co N\fino\to N\finplus$.) Here is a description using the particle model. (This means that a differentiable
structure on $M$ is not required, but compactness of $\partial M$ is required.) On \emph{white} configurations
the map acts by deleting the last point (label $k$) in any ordered configuration of $k$ elements. On \emph{black}
configurations it acts forgetfully, forgetting the distinguished element.

We also need a map $v$ from $\delta_*\varphi^*\config(M)$ to the part of $\config(M)$ lying above the
simplicial subset of $N\finplus$ generated by the $0$-simplex $[1]$, where $[1]$ is viewed as an object of $\finplus$.
The map is given by associating to each ordered configuration in sight (with additional data)
the sub-configuration of cardinality 1 determined by the distinguished label. Note that the part of $\config(M)$ lying above
the $0$-simplex $[1]$ and its degeneracies is, up to degreewise weak equivalence, a constant
simplicial space with constant term $M_-$ or $M$. Therefore we think of $(u,v)$ as a simplicial map over $N\finplus$
from $\delta_*\varphi^*\config(M)$ to $\config(M)\times M$. Here $\config(M)\times M$ is a simplicial space over $N\finplus$
because it has a projection map to $\config(M)$, and $\config(M)$ is a simplicial space over $N\finplus$ in the usual way.

\begin{prop} \label{prop-genpos} The map $(u,v)\co \delta_*\varphi^*\config(M) \to \config(M)\times M$ so defined is a
conservatization map over $N\finplus$.
\end{prop}

\begin{proof} For $p\in M_-$ let $\delta_*\varphi^{*|p}\config(M)$ be the degreewise homotopy fiber of the map $v$ over $p$.
This can also be defined directly like $\delta_*\varphi^*\config(M)$, except that all configurations (with additional data)
have their distinguished element, whether
black or white, at $p$, and morphisms of configurations respect that.
It is easy to show that the statement about $(u,v)$ being a conservatization map over $N\finplus$ is equivalent
to the following: \emph{for every $p$, the map
\[   u_p\co \delta_*\varphi^{*|p}\config(M) \to \config(M) \]
obtained by restricting $u$ is a conservatization map over $N\finplus$.} We now turn to the proof of that,
writing $Q$ for $\delta_*\varphi^{*|p}\config(M)$ when it is convenient and keeping $p$ fixed.
Since $\config(M)$ is conservative over $N\finplus$, we obtain a map $\Lambda(u_p)$ from
$\Lambda Q$ to $\Lambda\config(M)$ alias $\config(M)$.
We want to show that it is a degreewise weak equivalence. We proceed in four steps.
\begin{itemize}
\item[(i)] $\Lambda(u_p)$ is a weak equivalence in degree 0.
\item[(ii)] It follows easily from (i) that $Q$ has property $\beta$.
\item[(iii)] It follows easily from (i), (ii) or proof of (ii), theorem~\ref{thm-transp} and remark~\ref{rem-pinought}
that $\Lambda(u_p)$ in (i) determines a weak equivalence in degree 1.
\item[(iv)] It follows easily from (i), (iii), (ii) and theorem~\ref{thm-transp} that $\Lambda(u_p)$ in (i) is a weak equivalence in all
degrees.
\end{itemize}
Of these statements, (i) is the only one which requires a longer argument. We postpone that and turn to
the proofs of (ii), (iii) and (iv) modulo (i).

For statement (ii), pick an element $x\in Q_0$\,. If $x$ projects to an
object of the form $([k],k,\textup{white})$ in $\fino$\,, then $Q\downarrow x$
is identified with $\config(U)$ for some open subset $U$ of $M\smin p$\,. This follows from
formula~(\ref{eqn-formulamorbd}). The open subset $U$ is a tubular neighborhood of $\partial M\cup T$ where $T$ is a
subset of $M_-\smin p$ with $k-1$ elements. Therefore $Q\downarrow x$ is
conservative over $N\finplus$ and $\Lambda(Q\downarrow x)$ in degree 0 is identified with $\config(U)$ in degree 0.
If $x$ projects to an object of the form $([k],a,\textup{black})$ in $\fino$\,, then $Q\downarrow x$ is identified with
$\delta_*\varphi^{*|p}\config(U)$ for some open subset $U$ of $M$
which contains $p$. The open
subset $U$ is a tubular neighborhood of $\partial M\cup T$ where $T$ is a subset of $M_-$ with $k$ elements,
and $p\in T$. Therefore, by (i), the degree 0 part of $\Lambda(Q\downarrow x)$ is again identified with
$\config(U)$ in degree $0$. With that, it is easy to verify that property $\beta$ holds. (The
nontrivial instances of property $\beta$ are those where we look at
transport maps $\Lambda(Q\downarrow x)\to \Lambda(Q\downarrow y)$ induced by elements $f\in Q_1$
with $d_0f=x$ and $d_1f=y$~, covering elements $g\in (N\fino)_1$ such that $\delta(g)$ is an identity morphism in $\finplus$ but
$g$ itself is not invertible in $\fino$\,. These special morphisms $g$ in $\fino$ are among those which were
highlighted in remark~\ref{rem-suck}.)

For statement (iii), we observe that in proving (ii) we saw already that for every $x\in Q_0$ the
map from $\Lambda(Q\downarrow x)$ to $\config(M)\downarrow u_p(x)$ induced by $u_p$ is a weak
equivalence in degree 0. By (ii) and theorem~\ref{thm-transp}, it is allowed to write that map in the form
\[  \Lambda Q\downarrow x \to  \config(M) \downarrow u_p(x)\,. \]
By remark~\ref{rem-pinought}, it does not matter here whether we say $x\in Q_0$ or $x\in (\Lambda Q)_0$.
Therefore we can say that $\Lambda(u_p)\co \Lambda Q\to \config(M)$
induces maps of the over categories, $\Lambda Q\downarrow x\to \config(M)\downarrow \Lambda(u_p)(x)$,
which are always weak equivalences in degree $0$. Combining this with (i), we obtain (iii).

By (ii) and theorem~\ref{thm-transp}, we know that $\Lambda Q$ is a Segal space. Therefore
the map $\Lambda(u_p)$ is a weak equivalence in all degrees, since it is a weak equivalence in degrees 0 and 1
by (i) and (iii). This proves (iv).

It remains to establish (i). Unraveling the definition of $\Lambda Q$ in degree 0, and the part of that
projecting to $[k]\in (N\finplus)_0$, we see a homotopy colimit
\begin{equation} \label{eqn-rwhope1}
\hocolimsub{\tau} Q(\tau).
\end{equation}
Here $\tau$ runs through $\simp(N(\sC_k))=\simp(N\sC_k)$ where $\sC_k$ is the subcategory of $\fino$ consisting
of the morphisms which are taken to the identity of $[k]$ by $\delta$, and the objects which are
taken to $[k]$. That means:
\begin{itemize}
\item[-] the objects $([k+1],k+1,\textup{white})$ and $([k],a,\textup{black})$
for arbitrary $a\in \uli k$\,;
\item[-] apart from identity morphisms, only the morphisms
\[ f_a:([k+1],k+1,\textup{white}) \lra  ([k],a,\textup{black}) \]
where $f_a\co [k+1]\to [k]$ is the unique map which is based, surjective, takes $k+1$ to $a$ and is
order-preserving on $[k]$.
\end{itemize}
Let $\simp^1(N\sC_k)$ be the full subcategory of $\simp(N\sC_k)$
whose objects are the strings of length $\le 1$; that is, elements of $(N\sC_k)_0$ and $(N\sC_k)_1$\,.
The inclusion
\[ \simp^1(N\sC_k)\to \simp(N\sC_k) \]
admits a canonical retraction functor (left inverse)
\[ R\co \simp(N\sC_k) \lra \simp^1(N\sC_k)\,. \]
Namely, an object $\tau$ of $\simp(N\sC_k)$ is a string of morphisms in $\sC_k$~; if it consists
of identity morphisms only, it obviously determines an element $R(\tau)$ of $(N\sC_k)_0$~; otherwise it will
have exactly one non-identity morphism, which we declare to be $R(\tau)\in (N\sC_k)_1$\,. This is
useful because we can write
\[ Q(\tau)~\simeq~Q(R(\tau)) \]
(by a chain of natural weak equivalences),
so that the expression~(\ref{eqn-rwhope1}) turns into
\begin{equation} \label{eqn-rwhope2}
\hocolimsub{\tau\textup{ in }\simp(N\sC_k)} Q(R(\tau)).
\end{equation}
Moreover the functor $R$ is homotopy terminal so that the map
\begin{equation} \label{eqn-rwhope3}
\hocolimsub{\tau\textup{ in }\simp(N\sC_k)} Q(R(\tau))
\lra \hocolimsub{\lambda\textup{ in }\simp^1(N\sC_k)} Q(\lambda)
\end{equation}
induced by $R$ (think $\lambda=R(\tau)$) is a weak equivalence. After
these fairly abstract preparations, we must come to terms with
the target in~(\ref{eqn-rwhope3}). The objects of $\simp^1(N\sC_k)$
come in three types:
\begin{itemize}
\item[(1)] $([k+1],k+1,\textup{white})\in (N\sC_k)_0$~;
\item[(2)] $([k],a,\textup{black})\in (N\sC_k)_0$ for some $a\in \uli k$~;
\item[(3)] $\big(f_a:([k+1],k+1,\textup{white})\to([k],a,\textup{black})\big)\in (N\sC_k)_1$ for some $a$
in $\uli k$, where $f_a$ was defined earlier.
\end{itemize}
The values of $Q$ on these types are as follows.
\begin{itemize}
\item[$(1)_v$] $\emb(\uli k\,,\,M_-\smin p)$ in case (1). More precisely, we should think of embeddings
from $\{1,2,\dots,k,k+1\}$ to $M_-$ which take $k+1$ to $p$\,.
\item[$(2)_v$] $\emb(\uli k\smin a,\,M_-\smin p)$ in case (2). More precisely, we should think of embeddings
from $\{1,2,\dots,k\}$ to $M_-$ which take $a$ to $p$\,.
\item[$(3)_v$] $\partial D_p\times \emb(\uli k\smin a,\,M_-\smin p)$ in case (3).
More precisely, we should think of morphisms in $\config(M_-)$ whose source is
an embedding $e_1$ as in $(1)_v$ and whose target
is an embedding $e_2$ as in $(2)_v$\,. The underlying map
of these morphisms is assumed to be $f_a$ and the underlying homotopy from $e_1$ to $e_2f_a$ is
assumed to be stationary on $k+1$.
\end{itemize}
Therefore the target in~(\ref{eqn-rwhope3}) is precisely the homotopy pushout model
of the ordered configuration space $\emb(\uli k\,,\,M_-)$ that we saw in remark~\ref{rem-rw}, specifically in diagram~(\ref{eqn-rwpix}).
(Just replace $M$ by $M_-$ there.)
So our map from~(\ref{eqn-rwhope1}) to the part of $\config(M_-)$ projecting to $[k]\in (N\finplus)_0$
is a weak equivalence because diagram~(\ref{eqn-rwpix}) is a homotopy pushout square. \end{proof}

\medskip
\emph{Remark.} After writing up a first version of proposition~\ref{prop-genpos} for the special case $M=D^m$~, we learned
that Ricardo Andrade was already aware of something similar for arbitrary $M$. The added value in our formulation, compared to his, is mainly in the emphasis on conservatization (a form of localization).

\section{Shadowing} \label{sec-shadow}
\subsection{Black box formulation} The main result of this section complements the special case $M=D^m$ of proposition~\ref{prop-genpos}.
Together, the two give us a homotopical description of $\config(D^m)$ in terms of $\config(D^m_\circ)$,
where $D^m_\circ$ is the punctured disk $D^m\smin\{0\}$. To express this in a more formal manner: proposition~\ref{prop-genpos},
specialized to $M=D^m$, gives a homotopical description of $\config(D^m)$ in terms of $\varphi^*\config(D^m)$. The
main result of this section is a homotopical description of $\varphi^*\config(D^m)$ in terms of $\alpha^*\varphi^*\config(D^m)$,
where $\alpha\co \finplus\to \fino$ comes from definition~\ref{defn-fino}.
Since $\varphi\alpha$ is the functor from $\finplus$ to $\finplus$ given by $[k]\mapsto [k+1]$ on objects, etc., it is clear
that $\alpha^*\varphi^*\config(D^m)$ can be identified with $\config(D^m_\circ)$, for example by a simplicial map from $\config(D^m_\circ)$ to
$\alpha^*\varphi^*\config(D^m)$ which is a degreewise weak equivalence.

We begin with a black box formulation, proposition~\ref{prop-shadowbox} below. The technicalities inside the
box are summarized in lemma~\ref{lem-shadow} further down. After the proof of lemma~\ref{lem-shadow} we will deduce
proposition~\ref{prop-shadowbox} from lemma~\ref{lem-shadow}.

Our advice to readers who need an incentive for reading this long section is this. Read the statement of
proposition~\ref{prop-shadowbox}. Read definition~\ref{defn-exitindex} and remark~\ref{rem-exitindex}. Then turn to
section~\ref{sec-knottynew} for applications of proposition~\ref{prop-shadowbox}.
Then come back to this section~\ref{sec-shadow}. 

\begin{prop} \label{prop-shadowbox} There are endofunctors $F_0$, $F_1$, $F_2$ on the category of
simplicial spaces over $N\fino$ and natural transformations
\[  F_0 \Leftarrow F_1 \Rightarrow F_2, \]
such that
\begin{itemize}
\item[(i)] each $F_j$ is a homotopy functor, i.e., preserves degreewise weak equivalences between simplicial spaces over $N\fino$;
\item[(ii)] $F_0=\id$ and $F_2=E\alpha^*$ where
$E$ is a homotopy functor from simplicial spaces over $N\finplus$ to simplicial spaces over $N\fino$;
\item[(iii)] for $E$ in \emph{(ii)}, the composition $\alpha^*E$ is weakly equivalent to the identity functor on the category of simplicial spaces over $\finplus$\,;
\item[(iv)] the natural transformations specialize to (degreewise) weak equivalences
\[
F_0(Y) \xleftarrow{\; \simeq \;} F_1(Y) \xrightarrow{\; \simeq \;} F_2(Y)
\]
in the cases $Y=\varphi^*\config(D^m)$ and $Y=\delta^*\config(D^m)$, for arbitrary $m\ge 0$.
\end{itemize}
\end{prop}

\emph{Extract for now:} The proposition gives a zigzag of weak equivalences
relating $Y=\varphi^*\config(D^m)$ to $E(\alpha^*\varphi^*\config(D^m))\simeq E(\config(D^m_\circ))$.

\subsection{Shadowing: the setup}
Let $\sigma$ be an $r$-simplex in $N\fino$\,, which we may also write in the form of a diagram
\begin{equation} \label{eqn-sigmadiagram}
([k_0],a_0,c_0) \leftarrow ([k_1],a_1,c_1) \leftarrow ([k_2],a_2,c_2) \leftarrow \cdots \leftarrow ([k_r],a_r,c_r).
\end{equation}
If we wish to emphasize $\sigma$ more than $r$, as we often do,
then we use $I(\sigma)$ as alternative notation for the poset $[r]$
and write $\sigma\co I(\sigma)^\op \to \fino$\,.

\begin{defn} \label{defn-exitindex} {\rm Let $s$ be the maximum of the $i\in \{0,1,\dots,r\}$ such that $c_i=\textup{black}$,
if such an $i$ exists; if no such $i$ exists, set $s=-1$.
For each non-distinguished and non-zero element $z$ in one of the sets $[k_j]$, we can speak of
the \emph{exit index} of $z$, an element of $\{-1,0,2,\dots,s\}$.
This is the maximum of the integers $i$ such that $i<j$ and $c_i=\textup{black}$ and $z$ is mapped to
the distinguished element $a_i\in [k_i]$ under the appropriate composition of arrows in diagram~(\ref{eqn-sigmadiagram}), if such an
$i$ exists. If no such $i$ exists, then the exit index of $z$ is $-1$.
}
\end{defn}

\begin{rem} \label{rem-exitindex} {\rm
Let $Y$ be a simplicial space over $N\fino$. Write $Y(\sigma)$ for the part of $Y_r$ projecting to $\sigma\in (N\fino)_r$.
If $Y=\varphi^*\config(D^m)$ or $Y=\delta^*\config(D^m)$, then $Y(\sigma)$
is weakly equivalent to a product of terms specified by the exit index. In more detail:
\begin{itemize}
\item[-] Let $s$ be as in definition~\ref{defn-exitindex}. For each $i\in\{-1,\dots,s\}$ there is one factor $Y(\sigma^i)$.
The simplex $\sigma^i\in (N\fino)_r$ is obtained from $\sigma$ by selecting from each set $[k_j]$ in the diagram description of $\sigma$
only those elements which have exit index $i$, as well as the element $0$ and the distinguished element $a_j$. (The selected elements
need to be renumbered in accordance with the total ordering already in place.) There is an obvious map
\[ Y(\sigma) \to \prod_{i=-1}^s Y(\sigma^i) \]
by restriction of configurations; it is claimed to be a weak equivalence.
\end{itemize}
Here is a sketch proof of that. If $s=-1$,
there is nothing to do; therefore assume $s\ge 0$.
We start by making a product decomposition with two factors:
\begin{equation} \label{eqn-prodtwofac}
Y(\sigma) \xrightarrow{ \; \simeq \;} Y(\sigma^{0,1,2,\dots,s})\times Y(\sigma^{-1}).
\end{equation}
As already explained, $\sigma^{-1}$ is the $r$-simplex in $N\fino$
obtained from $\sigma$ by selecting from each set $[k_j]$ in the diagram
description~(\ref{eqn-sigmadiagram}) of $\sigma$ only those elements which have exit index $-1$, as well as $0$ and $a_j$.
Equivalently: we select from each set $[k_j]$ in the diagram
description of $\sigma$ only those elements which are \emph{not} mapped to the distinguished element $a_0\in [k_0]$,
as well as the distinguished element $a_j$. The $r$-simplex $\sigma^{0,1,2,\dots,s}$ in $N\fino$ is obtained
from $\sigma$ by selecting from each set $[k_j]$ in the diagram description of $\sigma$ only those elements which \emph{are}
mapped to $a_0\in [k_0]$, as well as the element $0$. The map~(\ref{eqn-prodtwofac}) is given by restriction to
subconfigurations.

To show that it is a weak equivalence, we view it as a map over the ordered configuration space $Y([k_0])$. Indeed,
$Y(\sigma)$ is a space over $Y([k_0])$ via the ultimate target map. The target of~(\ref{eqn-prodtwofac}) is a space
over $Y([k_0])$ because we can first project to $Y(\sigma^{-1})$ and then apply the ultimate target map from there
to $Y([k_0])$. In this way,~(\ref{eqn-prodtwofac}) is a map over $Y([k_0])$. By formula~(\ref{eqn-formulamorbd}),
the map induced by~(\ref{eqn-prodtwofac}) on homotopy fibers, over any chosen point in  $Y([k_0])$, is a weak equivalence.
Therefore~(\ref{eqn-prodtwofac}) itself is a weak equivalence.

We now repeat the argument with $Y(\sigma^{0,1,2,\dots,s})$ in place of $Y(\sigma)$, assuming $s>0$.
Let us write $[\ell_0],\dots,[\ell_r]$ for the sets in the diagram description of $\sigma^{0,1,2,\dots,s}$; note that $[\ell_0]=[1]$ and that all nonzero elements in any of $[\ell_j]$ are mapped to $1\in [\ell_0]$.
We get a product decomposition
\begin{equation}  \label{eqn-prodtwofacagain}
Y(\sigma^{0,1,2,\dots,s})  \xrightarrow{ \; \simeq \;} Y(\sigma^{1,2,\dots,s})\times Y(\sigma^0).
\end{equation}
To define the $r$-simplex $\sigma^{1,2,\dots,s}$ in $N\fino$,
we select from each set $[\ell_j]$ in the diagram description of $\sigma^{0,1,2,\dots,s}$ (where $j>1$)
only those elements which are mapped to the distinguished element of $[\ell_1]$, as well as the element $0$.
The map~(\ref{eqn-prodtwofacagain}) as such is given by restriction to subconfigurations. We can view it as a map
over $Y([\ell_1])$. Formula~(\ref{eqn-formulamorbd}) tells us that the map induced by~(\ref{eqn-prodtwofacagain}) on homotopy fibers, over any chosen point in  $Y([\ell_1])$, is a weak equivalence.
Therefore~(\ref{eqn-prodtwofacagain}) itself is a weak equivalence. And so on.

Furthermore, if $Y=\varphi^*\config(D^m)$, then $Y(\sigma^i)$ has another description, up to weak equivalence,
as the preimage in $\config(D^m_\circ)$, configuration category of the \emph{punctured} disk, of
some simplex of degree $r-i-1$ in $N\finplus$. This gives an idea of how $Y=\varphi^*\config(D^m)$
can be recovered from $\config(D^m_\circ)$. But the question remains
how these product decompositions of the various $Y(\sigma)$ might be compatible under the simplicial operators.
We have not tried to formulate a conceptual answer to that question, but we suggest that lemma~\ref{lem-shadow} below
is either a technical answer or a way to work around the question. In any case our constructions in the remainder of the section are guided by
the observations above concerning the exit index. We use the exit index to expand $\sigma$ into more complicated diagrams in $\fino$
of a typically 2-dimensional and 3-dimensional appearance. The typically 2-dimensional expansion makes the filtration by exit index
explicit and the typically 3-dimensional
expansion also makes the subquotients explicit.
}
\end{rem}  

\begin{defn} \label{defn-Vr} {\rm With the $r$-simplex $\sigma$ in $N\fino$ we associate a poset
$V(\sigma)$ containing $I(\sigma)^\op$ as a full sub-poset. As a set, $V(\sigma)$ is the set of all pairs $(x,y)$ where $y\in \{0,1,\dots,r\}$ and
\begin{itemize}
\item[] either $x\in \{0,1,\dots,r\}$ and $\sigma(x)$ has color black, or $x=-1$.
\end{itemize}
The order relation is given by $(x_0,y_0)\le (x_1,y_1)$ if and only if $x_0 \le x_1$ and $y_0 \ge y_1$~, nota bene. The dependence of $V(\sigma)$ on $\sigma$
is weak; we only need to know how many of the $x\in \{0,1,\dots,r\}$ are taken to black objects by $\sigma$.
We identify the full sub-poset of $V(\sigma)$ consisting of the pairs $(-1,y)$ with $I(\sigma)^\op$.
}
\end{defn}

\begin{expl} {\rm
If $I(\sigma)=\{0,1,2,3\}$
and $\sigma(0)$, $\sigma(1)$ are black while $\sigma(2)$ and $\sigma(3)$ are white,
then $V(\sigma)$ is identified with a subset of the plane $\RR^2$, with the ordering where $(a,b)\le (c,d)$ if and only if $a\le c$ and $b\ge d$, as in the following picture:
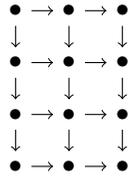
\[
\begin{tikzpicture}[descr/.style={fill=white}, baseline=(current bounding box.base)]
\matrix[column sep=3.0mm, row sep=3.0mm]{
        \node (l4) {$\bullet$}; & \node (c4) {$\bullet$};  & \node (r4) {$\bullet$}; \\
        \node (l3) {$\bullet$}; & \node (c3) {$\bullet$};  & \node (r3) {$\bullet$};  \\
        \node (l2) {$\bullet$}; & \node (c2) {$\bullet$};  & \node (r2) {$\bullet$};  \\
        \node (l1) {$\bullet$}; & \node (c1) {$\bullet$};  & \node (r1) {$\bullet$}; \\
    };
    \foreach \y [remember=\y as \lasty (initially 1)] in {2,...,4}{
    \path[draw, ->] (l\y)--(l\lasty);
    \path[draw, ->] (c\y)--(c\lasty);
    \path[draw, ->] (r\y)--(r\lasty);
    }
    \foreach \x in {1,...,4}{
    \path[draw, ->] (l\x)--(c\x);
    \path[draw, ->] (c\x)--(r\x);
    }
\end{tikzpicture}
\]
The left-hand column in the picture corresponds to the subset $I(\sigma)^\op$ of $V(\sigma)$.
}
\end{expl}

\begin{defn} \label{defn-sigmaV} {\rm Continuing in the notation of definition~\ref{defn-Vr}, we construct a preferred
extension $\sigma^V\co V(\sigma) \to \fino$
of $\sigma\co I(\sigma)^\op \to \fino$.

The idea is that $\sigma^V(x,y)$ for $(x,y)\in V(\sigma)\smin I(\sigma)^\op$
is the quotient of $\sigma(y)$ obtained by identifying all elements of $\sigma(y)$
whose exit index is less than $x$ with the base point 0. See remark~\ref{rem-exitindex}.
More precisely, the relationship $(x,y)\ge (-1,y)$ which holds in $V(\sigma)$ must induce a morphism in
$\fino$ from $\sigma(y)=\sigma^V(-1,y)$ to $\sigma^V(x,y)$. That morphism will be color-preserving: $\sigma^V(x,y)$ has the
same color as $\sigma(y)$. As a map of sets it will be surjective,
and it will be order-preserving and injective on the complement of the preimage of the base point $0$.
It remains only to say in each case what the preimage of 0 is.
\begin{itemize}
\item[-] In the case $x=-1$, the preimage of $0$ consists of $0$ only. (The morphism is an identity morphism.)
\item[-] In the cases where $0\le x\le y$, the preimage of 0 consists of all elements of $\sigma(y)$ which are not mapped to the distinguished element of $\sigma(x)$
under the map $\sigma(y)\to \sigma(x)$ induced by $x\le y$ in $I(\sigma)$.
\item[-] In the cases where $x\ge y$, the preimage of 0 consists of all elements of $\sigma(y)$ except for the
distinguished element.
\end{itemize}
This description of $\sigma^V$ so far specifies what $\sigma^V$ does on objects, and on some of the
morphisms in $V(\sigma)$. (In the drawing of $V(\sigma)$ just above, these are the horizontal morphisms with
source in the left-hand column, and the vertical morphisms in the left-hand
column.) There is a unique way to extend the
partial definition of $\sigma^V$ so that we have a functor from all of $V(\sigma)$ to $\fino$.
}
\end{defn}

\begin{expl} \label{expl-Vfeat} {\rm Suppose that $\sigma=([k],a,c)$ is a $0$-simplex of $N\fino$, alias object of $\fino$,
where $c=\textup{black}$. Then $I(\sigma)=[0]$
and $V(\sigma)$ is a totally ordered poset with 2 elements. The extension $\sigma^V$ of
$\sigma$ is the unique morphism from $([k],a,c)$ to $([1],1,c)$ such that
the preimage of $\{1\}$ is exactly $\{a\}$.

Suppose that $\sigma$ is the $1$-simplex determined by the morphism
\[
([7],3,\textup{black}) \xlongleftarrow{f}  ([5],5,\textup{white})
\]
where $f$ is defined by
\[  \begin{pmatrix} 0\mapsto 0 \\ 1,4\mapsto 4 \\ 2,5\mapsto 3 \\ 3\mapsto 6  \end{pmatrix}.  \]
(Note that $f(0)=0$ and $f(5)=3$ are compulsory.) Then $\sigma^V$ is the following diagram in $\fino$\,.
\[
	\begin{tikzpicture}[descr/.style={fill=white}, baseline=(current bounding box.base)] ]
	\matrix(m)[matrix of math nodes, row sep=4.5em, column sep=7em,
	text height=1.5ex, text depth=0.25ex]
	{
	([5],5,\textup{white}) & ([2],2,\textup{white})\\
	([7],3,\textup{black}) &  ([1],1,\textup{black}) \\
	};
	\path[->,font=\scriptsize]
		(m-1-2) edge node [auto] {${\footnotesize\begin{pmatrix} 0\mapsto 0 \\ 1,2\mapsto1 \end{pmatrix}}$} (m-2-2)
		(m-1-1) edge node [below] {${\footnotesize\begin{pmatrix} 0,1,3,4\mapsto 0 \\ 2\mapsto 1 \\ ~5\mapsto 2\end{pmatrix}}$} (m-1-2)
		(m-1-1) edge node [left] {$f$} (m-2-1)
		(m-2-1) edge node [below] {$\footnotesize\begin{pmatrix}  \ne 3\mapsto 0 \\ 3\mapsto 1 \end{pmatrix}$} (m-2-2);
	\end{tikzpicture}
\]
(The appearance of $[2]$ in the top right-hand term reflects the fact that the distinguished element $3$ in the target of $f$ has preimage of cardinality $2$.)

Suppose that $\sigma$ is the $2$-simplex determined by the string of morphisms
\[
([5],4,\textup{black})  \xlongleftarrow{g}  ([7],3,\textup{black}) \xlongleftarrow{f} ([5],5,\textup{white})
\]
in $\fino$~, where $f$ is defined as above and $g$ is defined by
\[ \begin{pmatrix} 0,2\mapsto 0 \\ 1\mapsto 1 \\ 4,7\mapsto 2 \\ 5\mapsto 3 \\ 3,6\mapsto 4 \end{pmatrix} \]
so that $gf$ is given by
\[  \begin{pmatrix} 0\mapsto 0 \\ 1,4\mapsto 2 \\ 2,3,5 \mapsto 4 \end{pmatrix}. \]
Then $\sigma^V$ is the following commutative diagram in $\fino$\,.
\[
	\begin{tikzpicture}[descr/.style={fill=white}, baseline=(current bounding box.base)] ]
	\matrix(m)[matrix of math nodes, row sep=4.5em, column sep=7em,
	text height=1.5ex, text depth=0.25ex]
	{
	([5],5,\textup{white}) & ([3],3,\textup{white}) & ([2],2,\textup{white}) \\
	([7],3,\textup{black}) & ([2],1,\textup{black}) & ([1],1,\textup{black}) \\
	([5],4,\textup{black}) & ([1],1,\textup{black}) & ([1],1,\textup{black}) \\
	};
	\path[->,font=\scriptsize]
		(m-1-1) edge node [below] {${\footnotesize \begin{pmatrix} 0,1,4\mapsto 0 \\ 2\mapsto1 \\ 3\mapsto 2
			\\ 5\mapsto 3 \end{pmatrix}}$} (m-1-2)
		(m-1-2) edge node [below] {${\footnotesize\begin{pmatrix} 0,2\mapsto 0 \\ 1\mapsto 1  \\ 3\mapsto 2
			\end{pmatrix}}$} (m-1-3)
		(m-2-1) edge node [below] {${\footnotesize\begin{pmatrix} 0,1,2,4,5,7\mapsto 0\\3\mapsto 1 \\ 6\mapsto 2  \end{pmatrix}}$} (m-2-2)
		(m-2-2) edge node [below] {${\footnotesize\begin{pmatrix} \ne 1 \mapsto 0 \\ 1\mapsto 1  \end{pmatrix}}$} (m-2-3)
		(m-3-1) edge node [below] {${\footnotesize\begin{pmatrix} \ne4 \mapsto 0 \\ 4\mapsto 1 \end{pmatrix}}$} (m-3-2)
		(m-3-2) edge node [below] {} (m-3-3)
		(m-1-1) edge node [left] {$f$} (m-2-1)
		(m-2-1) edge node [left] {$g$} (m-3-1)
		(m-1-2) edge node [below] {$$} (m-2-2)
		(m-2-2) edge node [below] {$$} (m-3-2)
		(m-1-3) edge node [below] {$$} (m-2-3)
		(m-2-3) edge node [below] {$$} (m-3-3);
	\end{tikzpicture}
\]
(The unlabeled vertical arrows are determined by labeled vertical ones and the horizontal ones. The appearance of $[3]$ in top row,
middle column reflects the fact that the distinguished element $4$ in the target of $gf$
has preimage of cardinality $3$. The appearance of $[2]$ in top row,
right-hand column reflects the fact that the distinguished element $3$ in the target of $f$
has preimage of cardinality $2$. The appearance of $[2]$ in middle row,
middle column reflects the fact that the distinguished element $4$ in the target of $g$
has preimage of cardinality $2$. )

Both examples illustrate the following general feature. In a diagram of type $\sigma^V$ (a diagram in $\fino$),
the horizontal morphisms are always surjective; they become injective as well as order-preserving if the preimage of $0$ is discarded.
The last example illustrates another general feature, not explicit in the definitions but easy to
explain from the definitions. For all vertical arrows, except possibly some in the left-hand column,
the preimage of $\{0\}$ is always $\{0\}$.

The last example illustrates yet another general feature. If there are positions $(x,y)\in V(\sigma)$ where $x \le y$, then
they are always occupied by the object $([1],1,\textup{black})$ of $\fino$\,. Horizontal arrows induced by $(x-1,y)\le (x,y)$ in
$V(\sigma)$ where $x=y$ are of type
\emph{everything maps to 0 except the distinguished element}. Vertical arrows induced by $(x,y+1)\le (x,y)$ in $V(\sigma)$
where $x=y$ are of type \emph{everything maps to the distinguished element except 0}, as we have already noted.
}
\end{expl}

\begin{defn} \label{defn-Wr} {\rm Here we define a poset $W(\sigma)$ depending on $\sigma\co I(\sigma)^\op \to \fino$\,.
An object in $W(\sigma)$ is a triple of integers $(t,x,y)$ such that
\begin{itemize}
\item[] both $(t,y)$ and $(x,y)$ qualify as elements of $V(\sigma)$;
\item[] $t\le x$;
\item[] $x\ge 0$.
\end{itemize}
The order relation is given by $(t_0,x_0,y_0)\le (t_1,x_1,y_1)$ iff $y_0\ge y_1$ and
$x_1\ge x_0$ and $t_1\ge t_0$\,, nota bene.
}
\end{defn}

\begin{defn} \label{defn-sigmaW} {\rm Continuing in the notation of definition~\ref{defn-Wr},
we construct a functor $\sigma^W\co W(\sigma) \to\fino$. It will land
in the full subcategory of $\fino$ spanned by the white objects.

The idea is this: for an object $(t,x,y)$ of $W(\sigma)$, we define $\sigma^W(t,x,y)$ as the
\emph{kernel} of the morphism $\sigma^V(t,y)\to \sigma^V(x,y)$ induced by the relationship $(t,y)\le (x,y)$ in the
poset $V(\sigma)$.
For a more precise definition we say that there will be
a morphism
\[ j_{t,x,y}\co \sigma^W(t,x,y) \to \sigma^V(t,y) \]
in $\fino$, natural in the variable $(t,x,y)\in W(\sigma)$, such that the underlying map of sets is injective.
Now we describe or define the image of $j_{t,x,y}$~, a subset of $\sigma^V(t,y)$.
It consists of the distinguished element of $\sigma^V(t,y)$ and all elements of $\sigma^V(t,y)$
which are mapped to the base point $0$ in $\sigma^V(x,y)$.

The color of $\sigma^W(t,x,y)$
is always white by definition, and the map $j_{t,x,y}$ is order-preserving if we omit
the distinguished element in the source.

Having defined $\sigma^W$ on objects, we need to say what it does on morphisms. Briefly, we make that
decision by saying that $j=\{j_{t,x,y}\}$ shall be a natural transformation between functors on $W(\sigma)^\op$. But there are a few
things to verify. Suppose that we have a relationship
$(t_0,x_0,y_0)\le (t_1,x_1,y_1)$ in $W(\sigma)$. This leads to a commutative square
\[
	\begin{tikzpicture}[descr/.style={fill=white}, baseline=(current bounding box.base)] ]
	\matrix(m)[matrix of math nodes, row sep=2.5em, column sep=2.5em,
	text height=1.5ex, text depth=0.25ex]
	{
	\sigma^V(t_0,y_0)  & \sigma^V(x_0,y_0)  \\
	\sigma^V(t_1,y_1)  & \sigma^V(x_1,y_1)  \\
	};
	\path[->,font=\scriptsize]
		(m-1-1) edge node [auto] {} (m-1-2)
		(m-2-1) edge node [auto] {} (m-2-2)
		(m-1-1) edge node [left] {} (m-2-1)
		(m-1-2) edge node [auto] {} (m-2-2);
	\end{tikzpicture}
\]
in $\fino$\,. It follows that we get an induced map from the underlying set of $\sigma^W(t_0,x_0,y_0)$
to the underlying set of $\sigma^W(t_1,x_1,y_1)$. This takes base point to base point
and distinguished element to distinguished element. But we want it to be a morphism of white
objects, and so we need to verify that non-distinguished elements of $\sigma^W(t_0,x_0,y_0)$ are mapped to
non-distinguished elements of $\sigma^W(t_1,x_1,y_1)$. This follows from a diagram chase in the
same commutative square.
}
\end{defn}

\begin{defn} \label{defn-sigmaVW} {\rm Let $VW(\sigma)$ be the categorical mapping cylinder
of the functor
\[ (t,x,y)\mapsto(t,y) \]
from $W(\sigma)$ to $V(\sigma)$. So $VW(\sigma)$ is a poset whose object set is the
disjoint union of $V(\sigma)$ and $W(\sigma)$; these two are full sub-posets, but there are additional relations
$(t,x,y) \le (t,y)$ for every $(t,x,y)\in W(\sigma)$. (There is a more general definition of categorical mapping cylinders in
appendix~\ref{sec-ext}.) The functors $\sigma^V\co V(\sigma) \to \fino$ and $\sigma^W\co W(\sigma) \to \fino$
and the natural transformation $j$ together make up a functor $\sigma^{VW}\co VW(\sigma) \to \fino$\,.

Finally, let $VW^\circ(\sigma)$ be the full sub-poset of $VW(\sigma)$ which is the preimage of the white subcategory of $\fino$
under $\sigma^{VW}$. Informally, it consists of $W(\sigma)$ and the white rows of $V(\sigma)$.
}
\end{defn}

\begin{lem} \emph{(Shadowing lemma.)} \label{lem-shadow} Let $Y:=\varphi^*\config(D^m)$ or let $Y:=\delta^*\config(D^m)$,
any $m\ge 0$. For any simplex $\sigma$ in $N\fino$~,
the restriction maps
\[
	\begin{tikzpicture}[descr/.style={fill=white}, baseline=(current bounding box.base)] ]
	\matrix(m)[matrix of math nodes, row sep=2.5em, column sep=2.5em,
	text height=1.5ex, text depth=0.25ex]
	{
	\rmap_\fino(V(\sigma),Y)   & \rmap_\fino(VW(\sigma),Y)  \\
	\rmap_\fino(I(\sigma)^\op,Y) & \rmap_\fino(VW^\circ(\sigma),Y) \\
	};
	\path[->,font=\scriptsize]
		(m-1-2) edge node [auto] {} (m-1-1)
		(m-1-1) edge node [left] {} (m-2-1)
		(m-1-2) edge node [auto] {} (m-2-2);
	\end{tikzpicture}
\]
are all weak equivalences. The reference functor $VW(\sigma) \to \fino$ is $\sigma^{VW}$.
The other reference functors to $\fino$ are appropriate restrictions of $\sigma^{VW}$.
\end{lem}

We have suppressed the nerve symbol $N$, for example by writing $V(\sigma)$ instead of $NV(\sigma)$.
Note that $\rmap_\fino(I(\sigma)^\op,Y)$ is identified with $Y(\sigma)$, the part of $Y_r$ projecting to $\sigma\in (N\fino)_r$\,.
We give some examples before turning to the proof.

\begin{expl} {\rm If $\sigma\co I(\sigma)^\op \to \fino$ lands in the white subcategory to begin with, then
$I(\sigma)^\op$ is equal to $V(\sigma)$ and also to $VW(\sigma)$ and $VW^\circ(\sigma)$, since
$W(\sigma)$ is empty. This makes the shadowing lemma easy to verify in this special case.
(Incidentally, this is one of our main reasons for making $W(\sigma)$, $V(\sigma)$ etc.~dependent on $\sigma$.)

A slightly more challenging case is the case of a $0$-simplex $\sigma\co I(\sigma)^\op \to \fino$ given by a black object,
say $([4],3,\textup{black})$. Then $V(\sigma)$ has two elements: $(-1,0)$ and $(0,0)$. Also, $W(\sigma)$ has two elements:
$(-1,0,0)$ and $(0,0,0)$. Therefore $VW(\sigma)$ has four elements. The functor $\sigma^{VW}$ is the commutative diagram
\[
	\begin{tikzpicture}[descr/.style={fill=white}, baseline=(current bounding box.base)] ]
	\matrix(m)[matrix of math nodes, row sep=6.5em, column sep=7em,
	text height=1.5ex, text depth=0.25ex]
	{
	([4],4,\textup{white}) & ([1],1,\textup{white})\\
	([4],3,\textup{black}) &  ([1],1,\textup{black}) \\
	};
	\path[->,font=\scriptsize]
		(m-1-2) edge node [auto] {} (m-2-2)
		(m-1-1) edge node [left] {${\footnotesize\begin{pmatrix}0\mapsto 0\\ 1\mapsto 1\\ 2\mapsto 2\\
3\mapsto 4 \\ 4\mapsto 3  \end{pmatrix}}$} (m-2-1)
		(m-1-1) edge node [left] {} (m-1-2)
		(m-2-1) edge node [below] {${\footnotesize\begin{pmatrix} \ne 3\mapsto 0 \\ 3 \mapsto 1 \end{pmatrix}}$} (m-2-2);
	\end{tikzpicture}
\]
in $\fino$\,.
The maps in lemma~\ref{lem-shadow} are then: restriction of derived lifts defined on the entire square to lower row of the square,
further restriction from lower row to lower left-hand term, and restriction of derived lifts defined on the entire square
to the top row of the square. It is easy to check manually that these restriction maps are weak equivalences. But do remember
that there are two cases to check: $Y=\varphi^*\config(D^m)$ and $Y=\delta^*\config(D^m)$.
}
\end{expl}

\subsection{Lifting lemmas for configuration categories}
These lifting lemmas will be needed in the proof of lemma~\ref{lem-shadow}.

\smallskip
Let $\finplus^{[1]^\op}$ be the category of functors from the poset $[1]^\op$ to $\finplus$\,. Let
\[ \finplus^{[1]^\op,\textup{mon}}\subset \finplus^{[1]^\op} \]
be the full subcategory whose objects are the functors $[1]^\op\to \finplus$ taking $0<1$ to
an injective morphism in $\finplus$\,. For a Segal space $Z$ over $N\finplus$ let
\[ Z^{\Delta[1],\textup{mon}} \subset  Z^{\Delta[1]} \]
be the preimage of $N(\finplus^{[1]^\op,\textup{mon}})$ under the
map from $Z^{\Delta[1]}$ to $N(\finplus^{[1]^\op})$ induced by the
reference map from $Z$ to $N\finplus$\,.

\begin{lem} \label{lem-monlift} For $Z=\config(M)$, the commutative square of Segal spaces
\[
	\begin{tikzpicture}[descr/.style={fill=white}, baseline=(current bounding box.base)] ]
	\matrix(m)[matrix of math nodes, row sep=2.5em, column sep=7.5em,
	text height=1.5ex, text depth=0.25ex]
	{
	Z^{\Delta[1],\textup{mon}}  & Z  \\
	N(\finplus^{[1]^\op,\textup{mon}})  & N\finplus  \\
	};
	\path[->,font=\scriptsize]
		(m-1-1) edge node [auto] {\textup{target operator $d_1$}} (m-1-2)
		(m-2-1) edge node [auto] {\textup{target operator $d_1$}} (m-2-2)
		(m-1-1) edge node [left] {} (m-2-1)
		(m-1-2) edge node [auto] {} (m-2-2);
	\end{tikzpicture}
\]
is degreewise homotopy cartesian.
\end{lem}

\begin{proof} In simplicial degree 0, this means the following. For every injective morphism $f$
in $\finplus$ with target $\ell$~, the target operator $d_1$ restricts to a weak equivalence from the
part of $Z$ which projects to $f\in (N\finplus)_1$ to the part of $Z$ which projects
to $[\ell]\in (N\finplus)_0$~. This is easy to confirm by looking at the homotopy fibers of
the target operator $d_1$ (from $Z_1$ to $Z_0$) and using formula~(\ref{eqn-formulamorbd}).

In simplicial degree 1, we should start with a commutative square
\[
	\begin{tikzpicture}[descr/.style={fill=white}, baseline=(current bounding box.base)] ]
	\matrix(m)[matrix of math nodes, row sep=2.5em, column sep=2.5em,
	text height=1.5ex, text depth=0.25ex]
	{
	{[}k_0{]} & {[}k_1{]}  \\
	{[}\ell_0{]} & {[}\ell_1{]} \\
	};
	\path[->,font=\scriptsize]
		(m-2-2) edge node [auto] {} (m-2-1)
		(m-1-2) edge node [above] {$h$} (m-1-1)
		(m-1-1) edge node [left] {$g$} (m-2-1)
		(m-1-2) edge node [auto] {} (m-2-2);
	\path[dotted,->,font=\scriptsize]
		(m-1-2) edge node [auto] {} (m-2-1);
	\end{tikzpicture}
\]
in $\finplus$ where the vertical arrows are injective. This can be viewed as a map from a
simplicial set $L$ generated by two 2-simplices to $N\finplus$. To show: the restriction map
from the space of derived lifts $L\to Z$ of the specified map $L\to N\finplus$ to the space of derived lifts of
the lower horizontal edge is a weak equivalence. This breaks up into two steps since $L$ is the pushout of two copies
of $\Delta[2]$ along a common edge. For the lower triangle alias 2-simplex in $L$~, we can use
the case of simplicial degree 0 which has already been established to extend from the lower edge to the unique inner horn.
Then lemma~\ref{lem-innerhorn}
can be applied. For the upper triangle, we can again use
the case of simplicial degree 0 which has already been established to extend from the dotted (diagonal) edge to the
horn consisting of diagonal edge and vertical edge. Using explicit models, such as the particle model, we can say that
$g$ has been lifted to a morphism $\bar g\in Z_1$\,. Write $s$ and $t$ for source and target of the morphism $\bar g$.
It remains to show that composition with $\bar g$
is a weak equivalence from the part of $\hofiber_s[d_1\co Z_1\to Z_0]$
projecting to $h\in (N\finplus)_1$ to the part of $\hofiber_t[d_1\co Z_1\to Z_0]$
projecting to $gh\in (N\finplus)_1$. This is again easy to verify using formula~(\ref{eqn-formulamorbd}).
Finally, a commutative square of Segal spaces which is homotopy cartesian in simplicial degrees 0 and 1 is
automatically homotopy cartesian in all higher degrees. \end{proof}

\medskip
For the next lemma we introduce more terminology. A diagram of based sets and based maps
\[
A \xlongrightarrow{f} B \xlongrightarrow{g} C
\]
is \emph{skew-exact} (se) if $f$ is injective and $g$ induces a bijection from the quotient
$B/\im(f)$ to the quotient $C/\im(gf)$. (This implies that $g$ is surjective. It does not imply
that $gf$ is the zero map.) Any based map $h\co A\to C$ of based sets admits an essentially unique factorization
$h=gf$ such that $g$ and $f$ together form a skew-exact diagram. An explicit solution is
\[
A \xlongrightarrow{f} A\vee \big(C/\im(h)\big) \xlongrightarrow{g} C
\]
where $f$ is the inclusion of the first wedge summand, $g$ agrees with $h$ on the first wedge summand and
$g$ agrees with the inclusion $C\smin\im(h)\to C$ on the nonzero elements of the second wedge summand.
The content of lemma~\ref{lem-factlift} below is roughly that skew-exact factorizations
of morphisms in $\finplus$ can be lifted uniquely to factorizations in $\config(M)$; that is, for a morphism $\bar h$ in $\config(M)$
projecting to $h\in (N\finplus)_1$, any skew-exact factorization $h=fg$ has an essentially
unique lift to a factorization of $\bar h$ in $\config(M)$. But once again we need an internal formulation.

A commutative diagram of based sets and based maps
\begin{equation} \label{eqn-batchmor}
	\begin{tikzpicture}[descr/.style={fill=white}, baseline=(current bounding box.base)] ]
	\matrix(m)[matrix of math nodes, row sep=2.5em, column sep=2.5em,
	text height=1.5ex, text depth=0.25ex]
	{
	A_1 & B_1 & C_1  \\
	A_0 & B_0 & C_0 \\
	};
	\path[->,font=\scriptsize]
		(m-1-1) edge node [auto] {$f_1$} (m-1-2)
		(m-1-2) edge node [auto] {$g_1$} (m-1-3)
		(m-2-1) edge node [auto] {$f_0$} (m-2-2)
		(m-2-2) edge node [auto] {$g_0$} (m-2-3)
		(m-1-1) edge node [auto] {} (m-2-1)
		(m-1-2) edge node [auto] {} (m-2-2)
		(m-1-3) edge node [auto] {$v$} (m-2-3);
	\end{tikzpicture}
\end{equation}
with skew-exact rows is \emph{admissible} if the arrow labeled $v$ takes the complement of $\im(g_1f_1)$ to the
complement of $\im(g_0f_0)$.

\begin{expl} {\rm For $(t,x,y)\in W(\sigma)$ such that $x\ge t+1$ the diagram
\[
	\begin{tikzpicture}[descr/.style={fill=white}, baseline=(current bounding box.base)] ]
	\matrix(m)[matrix of math nodes, row sep=2.5em, column sep=2.5em,
	text height=1.5ex, text depth=0.25ex]
	{
	 & \varphi\sigma^W(t,x,y)  \\
	\varphi\sigma^V(t+1,y) & \varphi\sigma^V(t,y) \\
	};
	\path[->,font=\scriptsize]
		(m-1-2) edge node [auto] {$j$} (m-2-2)
		(m-2-2) edge node [auto] {} (m-2-1);
	\end{tikzpicture}
\]
is skew-exact. Informally, the horizontal arrow is the quotient map
\[ \frac{\sigma(y)}{ \{\textup{exit index $<t+1$}\}} \quad\longleftarrow\quad \frac{\sigma(y)}{ \{\textup{exit index $<t$}\}}
\]
while the vertical arrow is injective, with image consisting of the elements of $\sigma(y)$ having exit index $<x$, plus the
distinguished element. Similarly, for $(t,x,y)\in W(\sigma)$ such that $x\ge t+1$ the diagram
\[
	\begin{tikzpicture}[descr/.style={fill=white}, baseline=(current bounding box.base)] ]
	\matrix(m)[matrix of math nodes, row sep=2.5em, column sep=2.5em,
	text height=1.5ex, text depth=0.25ex]
	{
	 & \delta\sigma^W(t,x,y)  \\
	\delta\sigma^V(t+1,y)  & \delta\sigma^V(t,y) \\
	};
	\path[->,font=\scriptsize]
		(m-1-2) edge node [auto] {$j$} (m-2-2)
		(m-2-2) edge node [auto] {} (m-2-1);
	\end{tikzpicture}
\]
is skew-exact. Furthermore, we can let $y$ vary in these diagrams and we obtain admissible transformations
between skew-exact diagrams.
}
\end{expl}

Let $\finplus^{[2]^\op,\textup{se}}\subset \finplus^{[2]^\op}$
be the subcategory whose objects are the skew-exact functors $[2]^\op\to \finplus$
and whose morphisms are the admissible natural transformations.
For a Segal space $Z$ over $N\finplus$ let
\[  Z^{\Delta[2],\textup{se}} \subset Z^{\Delta[2]} \]
be the preimage of $N(\finplus^{[2]^\op,\textup{se}})\subset N(\finplus^{[2]^\op})$ under the
evident reference map.

\begin{lem} \label{lem-factlift} For $Z=\config(M)$, the commutative square of Segal spaces
\[
	\begin{tikzpicture}[descr/.style={fill=white}, baseline=(current bounding box.base)] ]
	\matrix(m)[matrix of math nodes, row sep=2.5em, column sep=7.5em,
	text height=1.5ex, text depth=0.25ex]
	{
	Z^{\Delta[2],\textup{se}}  & Z^{\Delta[1]}  \\
	N(\finplus^{[2],\textup{se}}) & N(\finplus^{[1]})  \\
	};
	\path[->,font=\scriptsize]
		(m-1-1) edge node [auto] {\textup{compos.~operator $d_1$}} (m-1-2)
		(m-2-1) edge node [auto] {\textup{compos.~operator $d_1$}} (m-2-2)
		(m-1-1) edge node [left] {} (m-2-1)
		(m-1-2) edge node [auto] {} (m-2-2);
	\end{tikzpicture}
\]
is degreewise homotopy cartesian.
\end{lem}

\begin{proof} In simplicial degree 0, this means the following. For every se diagram
\[
[k_0] \xlongleftarrow{f_1} [k_1] \xlongleftarrow{f_2} [k_2]
\]
in $\finplus$, the operator $d_1$ restricts to a weak equivalence
from the part of $Z_2$ which projects to $(f_1,f_2)\in (N\finplus)_2$ to the part of
$Z_1$ which projects to $f_1f_2\in (N\finplus)_1$\,. The verification can be made easier
by fixing a lift $y\in Z_0$ of $[k_0]$ and working relative to that.
It is allowed to use the particle model. So let $\Phi_s$ be
the space of pairs of morphisms $(\omega_1,\omega_2)$ in the particle model of $\config(M)$
such that the target of $\omega_1$ is $y$ and the source of $\omega_1$ is the target of $\omega_2$
(and $\omega _1, \omega_2$ lift $f_1,f_2$ respectively). Let $\Phi_t$ be the space of morphisms
$\tau$ such that the target of $\tau$ is $y$ (and $\tau$ lifts $f_1f_2$). There is a map $\Psi_s\to \Psi_t$ given by $(\omega_1,\omega_2)\mapsto
(\omega_1\omega_2)$. We need to show that it is a weak equivalence. This follows easily from
formula~(\ref{eqn-formulamorbd}).

We turn to the case of simplicial degree 1. This requires more of an effort.
Let $\kappa\co \Delta[1] \to \Delta[2]$ be the simplicial map which takes the standard generator in the
source to $d_1$ of the standard generator in the target. Write $L:=\Delta[2]\times\Delta[1]$ and write
$K\subset L$ for the union of $\im(\kappa)\times\Delta[1]$ and $\Delta[2]\times\partial\Delta[1]$, where $\partial\Delta[1]\subset\Delta[1]$)
is generated by the 0-simplices in $\Delta[1]$.
Suppose given a simplicial map
\[ e\co L\to N\finplus \]
which represents a morphism in $\finplus^{[2],\textup{se}}$.
We can think of this as a commutative diagram
\[
	\begin{tikzpicture}[descr/.style={fill=white}, baseline=(current bounding box.base)] ]
	\matrix(m)[matrix of math nodes, row sep=2.5em, column sep=2.5em,
	text height=1.5ex, text depth=0.25ex]
	{
	{[}k_0{]} & {[}k_1{]} & {[}k_2{]}  \\
	{[}\ell_0{]} & {[}\ell_1{]} & {[}\ell_2{]} \\
	};
	\path[<-,font=\scriptsize]
		(m-1-1) edge node [auto] {$f_1$} (m-1-2)
		(m-1-2) edge node [auto] {$f_2$} (m-1-3)
		(m-2-1) edge node [auto] {$g_1$} (m-2-2)
		(m-2-2) edge node [auto] {$g_2$} (m-2-3);
	\path[->,font=\scriptsize]		
		(m-1-1) edge node [auto] {$h_0$} (m-2-1)
		(m-1-2) edge node [auto] {$h_1$} (m-2-2)
		(m-1-3) edge node [auto] {$h_2$} (m-2-3);
	\end{tikzpicture}
\]
We have to show that the restriction map
\[  \rmap_\finplus(L,Z) \lra \rmap_\finplus(K,Z) \]
is a weak equivalence. (Here $L$ has reference map $e$ to $N\finplus$\,.)
We choose a lift $y\in Z_0$ of $[\ell_0]\in (N\finplus)_0$ and work relative to that. Write $\Phi_s$ and $\Phi_t$
for the homotopy fibers of $\rmap_\finplus(L,Z)$ and $\rmap_\finplus(K,Z)$ over $y$, respectively.
As in the proof of lemma~\ref{lem-monlift}  we can reduce to a situation where $M=U$ is the disjoint union of an open
collar and a tubular neighborhood of a configuration representing $y$. This redefines $Z$ as $\config(U)$, and so
$Z$ has a weakly ternminal object $y$. Let $L'\subset L$ be the simplicial set obtained
from $L$ by deleting all simplices containing
the vertex corresponding to $[\ell_0]$. Let $K'=L'\cap K$.
Then we have a commutative diagram
\begin{equation} \label{eqn-heavybat}
	\begin{tikzpicture}[descr/.style={fill=white}, baseline=(current bounding box.base)] ]
	\matrix(m)[matrix of math nodes, row sep=2.5em, column sep=2.5em,
	text height=1.5ex, text depth=0.25ex]
	{
	\Phi_s & \Phi_t  \\
	\rmap_{\finplus\downarrow[\ell_0]}(L,Z) & \rmap_{\finplus\downarrow[\ell_0]}(K,Z) \\
	\rmap_{\finplus\downarrow[\ell_0]}(L',Z) & \rmap_{\finplus\downarrow[\ell_0]}(K',Z) \\
	};
	\path[->,font=\scriptsize]
		(m-1-1) edge node [auto] {} (m-1-2)
		(m-2-1) edge node [auto] {} (m-2-2)
		(m-3-1) edge node [auto] {} (m-3-2);
	\path[->,font=\scriptsize]		
		(m-1-1) edge node [auto] {\textup{forget}} (m-2-1)
		(m-2-1) edge node [auto] {\textup{restrict}} (m-3-1)
		(m-1-2) edge node [auto] {\textup{forget}} (m-2-2)
		(m-2-2) edge node [auto] {\textup{restrict}} (m-3-2);
	\end{tikzpicture}
\end{equation}
where the four vertical arrows are weak equivalences. Therefore it remains only to
show that the restriction map
\begin{equation} \label{eqn-lightbat} \rmap_{\finplus\downarrow[\ell_0]}(L',Z) \to \rmap_{\finplus\downarrow[\ell_0]}(K',Z)
\end{equation}
is a weak equivalence.
Here is a drawing of $L'$ with labels showing the reference map to $N\finplus$~:
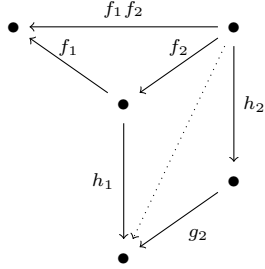
\[
	\begin{tikzpicture}[descr/.style={fill=white}, baseline=(current bounding box.base)] ]
	\matrix(m)[matrix of math nodes, row sep=1.5em, column sep=3.0em,
	text height=1.5ex, text depth=0.25ex]
	{
	\bullet & & \bullet  \\
	& \bullet & \\
	& & \bullet  \\
	& \bullet & \\
	};
	\path[->,font=\scriptsize]
		(m-1-3) edge node [above] {$f_1f_2$} (m-1-1)
		(m-1-3) edge node [above] {$f_2$} (m-2-2)
		(m-2-2) edge node [above] {$f_1$} (m-1-1)
		(m-2-2) edge node [left] {$h_1$} (m-4-2)
		(m-1-3) edge node [auto] {$h_2$} (m-3-3)
		(m-3-3) edge node [auto] {$g_2$} (m-4-2);
	\path[dotted,->,font=\scriptsize]
		(m-1-3) edge node [auto] {} (m-4-2);
	\end{tikzpicture}
	\]
The 2-simplex with edges labeled $f_1$, $f_2$ and $f_1f_2$ and the 1-simplices labeled $h_2$ and $g_2$ generate $K'$.
The restriction map~(\ref{eqn-lightbat}) is a
composition of two restriction maps, which will be described as we deal with them. (We show that both are weak equivalences.)
We start by deleting the $1$-simplex labeled $h_1$ and the 2-simplex that contains it
(but we keep the unlabeled dotted edge for now).
The associated restriction map is a weak equivalence because of admissibility, that is, because $h_1$ takes the complement of $\im(f_2)$
to the complement of $\im(g_2)$, which in turn is identified with a subset of $[\ell_0]$. (Therefore $h_1$ has an essentially unique lift on the
complement of $\im(f_2)$; the lift on $\im(f_2)$ is prescribed by the lift of the dotted arrow, and the two partial lifts don't interfere with each other.) Next we delete the unlabeled dotted edge and the 2-simplex containing it
(but we do not delete the edges $g_2$ and $h_2$).
The associated restriction map is a weak equivalence because it corresponds to a relative inner horn inclusion. \end{proof}

\subsection{Shadowing: the proofs} \label{subsec-shadowproofs} The proof of lemma~\ref{lem-shadow} comes in three parts,
corresponding to the three arrows in the diagram. We can assume that $\sigma(0)$ is a black object of $\fino$ since the
case where $\sigma$ lands in the white subcategory has already been discussed. This gives more uniformity in the drawings,
where drawings are needed. Since $\sigma$ is fixed, we shall write $I(\sigma)=I_r=\{0,1,\dots,r\}$ and $V:=V(\sigma)$,
$W:=W(\sigma)$, and so on.

\begin{proof}[Proof of lemma~\ref{lem-shadow}: restricting from $V$ to $I_r^\op$]
The nerve of $V$ has a simplicial subset $L$ which contains all $0$-simplices of $NV$
and is generated by nondegenerate $2$-simplices in $NV$~,
as indicated in the following picture.

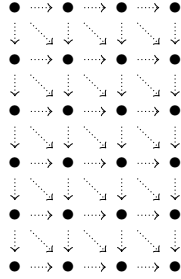
\[
\begin{tikzpicture}[descr/.style={fill=white}, baseline=(current bounding box.base)]
\matrix[column sep=3.0mm, row sep=3.0mm]{
        \node (06) {$\bullet$}; & \node (16) {$\bullet$};  & \node (26) {$\bullet$}; & \node (36) {$\bullet$}; \\
        \node (05) {$\bullet$}; & \node (15) {$\bullet$};  & \node (25) {$\bullet$}; & \node (35) {$\bullet$};  \\
        \node (04) {$\bullet$}; & \node (14) {$\bullet$};  & \node (24) {$\bullet$}; & \node (34) {$\bullet$}; \\
        \node (03) {$\bullet$}; & \node (13) {$\bullet$};  & \node (23) {$\bullet$}; & \node (33) {$\bullet$};  \\
        \node (02) {$\bullet$}; & \node (12) {$\bullet$};  & \node (22) {$\bullet$}; & \node (32) {$\bullet$};  \\
        \node (01) {$\bullet$}; & \node (11) {$\bullet$};  & \node (21) {$\bullet$}; & \node (31) {$\bullet$}; \\
    };
    \foreach \y [remember=\y as \lasty (initially 1)] in {2,...,6}{
    \foreach \x [remember=\x as \lastx (initially 0)] in {1,...,3}{
	    \path[draw, ->, densely dotted] (\lastx\y)--(\lastx\lasty);
	    \path[draw, ->, densely dotted] (\lastx\y)--(\x\lasty);
	    \path[draw, ->, densely dotted] (\lastx\y)--(\x\y);
	    }
    }
    \foreach \x [remember=\x as \lastx (initially 0)] in {1,...,3}{
	    \path[draw, densely dotted, ->] (\lastx1)--(\x1);
    }
    \foreach \y [remember=\y as \lasty (initially 1)] in {2,...,6}{
	    \path[->, densely dotted, font=\tiny ]
		(3\y) edge node [left] {} (3\lasty);
    }

\end{tikzpicture}
\]
By repeated application of corollary~\ref{cor-innerhorn}, the inclusion $L\to NV$ is a weak equivalence
in the setting of Segal spaces.
Therefore and since $Y$ is a Segal space, we may replace
$\rmap_\finplus(V,Y)$ by $\rmap_\finplus(L,Y)$ without changing the weak homotopy
type. Let
\[ K=L\cap N(I_r^\op) \subset NV \]
which corresponds to the left-hand vertical chain in the above picture.
The inclusion $K\to N(I_r^\op)$ is a weak equivalence
in the setting of Segal spaces. In this way, our restriction map simplifies to
\begin{equation}
\rmap_\fino(L,Y) \lra \rmap_\fino(K,Y)
\end{equation}
and it is induced by the inclusion $K\to L$. Next we
write the inclusion $K\to L$ as a composition of inclusions of simplicial subsets,
\[ K=K(0)\subset K(1)\subset \cdots \subset K(s-1)\subset K(s)=L
\]
so that $K(j)$ is obtained from $K(j-1)$ by attaching one more nondegenerate
$2$-simplex (gluing along an edge or along a horn, union of two edges). The recommended order in which these $2$-simplices should be attached is as follows:
 
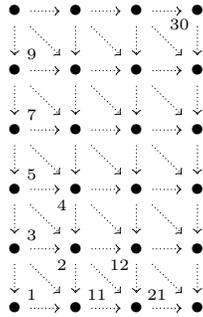
\[
\begin{tikzpicture}[descr/.style={fill=white}, baseline=(current bounding box.base)]
\matrix[column sep=4.0mm, row sep=4.0mm]{
        \node (06) {$\bullet$}; & \node (16) {$\bullet$};  & \node (26) {$\bullet$}; & \node (36) {$\bullet$}; \\
        \node (05) {$\bullet$}; & \node (15) {$\bullet$};  & \node (25) {$\bullet$}; & \node (35) {$\bullet$};  \\
        \node (04) {$\bullet$}; & \node (14) {$\bullet$};  & \node (24) {$\bullet$}; & \node (34) {$\bullet$}; \\
        \node (03) {$\bullet$}; & \node (13) {$\bullet$};  & \node (23) {$\bullet$}; & \node (33) {$\bullet$};  \\
        \node (02) {$\bullet$}; & \node (12) {$\bullet$};  & \node (22) {$\bullet$}; & \node (32) {$\bullet$};  \\
        \node (01) {$\bullet$}; & \node (11) {$\bullet$};  & \node (21) {$\bullet$}; & \node (31) {$\bullet$}; \\
    };
    \foreach \y [remember=\y as \lasty (initially 1)] in {2,...,6}{
    \foreach \x [remember=\x as \lastx (initially 0)] in {1,...,3}{
	    \path[draw, ->, densely dotted] (\lastx\y)--(\lastx\lasty);
	    \path[draw, ->, densely dotted] (\lastx\y)--(\x\lasty);
	    \path[draw, ->, densely dotted] (\lastx\y)--(\x\y);
	    }
    }
    \foreach \x [remember=\x as \lastx (initially 0)] in {1,...,3}{
	    \path[draw, densely dotted, ->] (\lastx1)--(\x1);
    }
    \foreach \y [remember=\y as \lasty (initially 1)] in {2,...,6}{
	    \path[->, densely dotted, font=\tiny ]
		(3\y) edge node [left] {} (3\lasty);
    }

    \node[text width=0.1cm] at (-1,1.4) {{\tiny 9}};
    \node[text width=0.1cm] at (-1, 0.6) {{\tiny 7}};
    \node[text width=0.1cm] at (-1,-0.2) {{\tiny 5}};
    \node[text width=0.1cm] at (-1,-1.0) {{\tiny 3}};
    \node[text width=0.1cm] at (-1,-1.8) {{\tiny 1}};

    \node[text width=0.1cm] at (-0.6,-0.6) {{\tiny 4}};
    \node[text width=0.1cm] at (-0.6,-1.4) {{\tiny 2}};

    \node[text width=0.1cm] at (-0.2,-1.8) {{\tiny 11}};
    \node[text width=0.1cm] at (0.1,-1.4) {{\tiny 12}};
    \node[text width=0.1cm] at (0.6,-1.8) {{\tiny 21}};

    \node[text width=0.1cm] at (0.9,1.8) {{\tiny 30}};

\end{tikzpicture}
\]
It suffices to show that the restriction
\[
\rmap_\fino(K(j),Y) \lra \rmap_\fino(K(j-1),Y)
\]
is a weak equivalence, for $j=1,\dots,s$. Let $Z=\config(D^m)$. Note that we must allow $Y=\varphi^*Z$ and
$Y=\delta^*Z$.

\emph{Case 1.} These are the cases where $j$ is odd and the inclusion $K(j-1)\hookrightarrow K(j)$
is a relative inner horn inclusion. We can use lemma~\ref{lem-innerhorn}.

\emph{Case 2.} These are the cases where $j$ is odd but exceptional: in the above drawing, the cases
$j=1,11,21$. Here it is a good idea to factor the inclusion $K(j-1)\hookrightarrow K(j)$ into two,
\[   K(j-1) \hookrightarrow  K'(j-1) \hookrightarrow K(j)~,  \]
where $K'(j-1)$ is obtained from $K(j-1)$ by attaching only the $d_2$-face of the new 2-simplex in $K(j)$.
This $d_2$-face is represented by a horizontal edge in the above picture. The inclusion
$K'(j-1) \hookrightarrow K(j)$ is again a relative inner horn inclusion. For the other inclusion,
$K(j-1) \hookrightarrow K'(j-1)$, we need to show that
\[  d_0\co Z_1 \lra Z_0 \]
(the source operator)
restricts to a weak equivalence between the preimages of certain elements in $(N\finplus)_1$ and $(N\finplus)_0$~, respectively.
The elements in $(N\finplus)_1$ that we should be looking at have the form
\[
[1]�\xlongleftarrow{g} [k]
\]
where $g$ is a based map such that $g^{-1}(1)$ has exactly one element. (Use the observations on general features
at the end of example~\ref{expl-Vfeat}. There is no need to make a careful distinction between the cases
$Y=\varphi^*\config(D^m)$ and $Y=\delta^*\config(D^m)$ here because we are concerned with the black part of $\fino$\,.)
The corresponding element in $(N\finplus)_0$
that we should be looking at is the source $[k]$. Since the part of $Z_0$ projecting to $[1]\in (N\finplus)_0$
is weakly contractible, formula~(\ref{eqn-formulamorbd}) gives us a description of the part of $Z_1$ projecting to
$g\in (N\finplus)_1$\,. With that description, the verification is easy.

\emph{Case 3.} These are the cases where $j$ is even. We attach a 2-simplex along an edge,
the $d_1$ face of the $2$-simplex. Therefore the task is to show that
\[   d_1\co Z_2 \lra   Z_1  \]
restricts to a weak equivalence between the preimages of certain elements in $(N\finplus)_2$ and $(N\finplus)_1$~, respectively.
More precisely, we are going to show that
\begin{equation} \label{eqn-shadowcheck}
	\begin{tikzpicture}[descr/.style={fill=white}, baseline=(current bounding box.base)] ]
	\matrix(m)[matrix of math nodes, row sep=2.5em, column sep=2.5em,
	text height=1.5ex, text depth=0.25ex]
	{
	Z_2 & Z_1  \\
	Z_0 & Z_0 \\
	};
	\path[->,font=\scriptsize]
		(m-1-1) edge node [auto] {$d_1$} (m-1-2)
		(m-1-1) edge node [auto] {$d_1d_1$} (m-2-1)
		(m-1-2) edge node [auto] {$d_1$} (m-2-2);
		\draw [double equal sign distance] (m-2-1) to (m-2-2) node [anchor=mid west] {$$};
	\end{tikzpicture}
\end{equation}
restricts to a homotopy cartesian square on the preimages of certain elements in $(N\finplus)_2$, $(N\finplus)_1$ and $(N\finplus)_0$.
The elements in $(N\finplus)_2$ that we should be looking at have the form
\[
[m] \xlongleftarrow{f_1} [k]  \xlongleftarrow{f_2} [\ell]
\]
where $f_2$ is surjective, order preserving and injective away from the preimage of the base point,
and $f_1(z)=0$ implies $z=0$. (These properties were highlighted earlier, at the end of example~\ref{expl-Vfeat}.
It may help to re-draw the arrow $f_1$ as a vertical arrow.)
The element in $(N\finplus)_1$ that we should be looking at is the composition
\[
[m] \xleftarrow{\quad f_1f_2 \quad} [\ell]
\]
The element in $(N\finplus)_0$ that we should be looking at is $[m]$. (Note that $f_1$ and $f_2$
can be reconstructed from $f_1f_2$\,.) Again, the verification is easy since we can use formula~(\ref{eqn-formulamorbd})
to describe the vertical homotopy fibers in~(\ref{eqn-shadowcheck}). \end{proof}

\medskip
\begin{proof}[Proof of lemma~\ref{lem-shadow}: restricting from $VW$ to $V$]
Let $q\co W \to V$ be the map of posets defined by $(t,x,y)\mapsto (t,y)$, as seen in
definition~\ref{defn-sigmaVW}. Let $e$ from $W \times[1]$ to $VW$ be the map of posets defined by $((t,x,y),0)\mapsto (t,x,y)$ and
$((t,x,y),1)\mapsto(t,y)$. According to corollary~\ref{cor-innerhorn},
the nerve of $VW$ is weakly equivalent (details as in that corollary) to the
mapping cylinder of
\[ NW \to NV \;. \]
Since the mapping cylinder was defined as a pushout, and the pushout also
turns out to be a homotopy pushout (with a convenient choice of model structure which has the prescribed weak equivalences: the injective one),
it follows that we have a homotopy pullback square
\[
	\begin{tikzpicture}[descr/.style={fill=white}, baseline=(current bounding box.base)] ]
	\matrix(m)[matrix of math nodes, row sep=2.5em, column sep=2.5em,
	text height=1.5ex, text depth=0.25ex]
	{
	\rmap_{\fino}(VW,Y) & \rmap_{\fino}(W\times[1],Y)  \\
	\rmap_{\fino}(V,Y) & \rmap_{\fino}(W,Y) \\
	};
	\path[->,font=\scriptsize]
		(m-1-1) edge node [auto] {$e^*$} (m-1-2)
		(m-1-1) edge node [auto] {\textup{res.}} (m-2-1)
		(m-2-1) edge node [auto] {$q^*$} (m-2-2)
		(m-1-2) edge node [auto] {$(-,1)^*$} (m-2-2);
	\end{tikzpicture}
\]
(where $W$ in the right-hand term maps to $N\fino$ by $W\to V\to N\fino$).
So it only remains to be shown that the map in the right-hand column is a weak equivalence.
Using definition~\ref{defn-rezkadj} and putting $Z=\config(D^m)$, we may write that map in the form
\[
\rmap_{\finplus^{[1]}}(W,Z^{\Delta[1],\textup{mon}}) \lra \rmap_{\finplus}(W,Z)\,.
\]
(The reference functors from $W$ to $\finplus^{[1]}$ and $\finplus$ depend on
whether we are dealing with $Y=\varphi^*Z$ or $Y=\delta^*Z$.)
It is the map given by post-composition with the target operator from $Z^{\Delta[1],\textup{mon}}$ to $Z$.
It is a weak equivalence by lemma~\ref{lem-monlift}. \end{proof}

\medskip
\begin{proof}[Proof of lemma~\ref{lem-shadow}: restricting from $VW$ to $VW^\circ$]
Let $s$ be the maximum of the $i\in \{0,1,\dots,r\}$ such that $\sigma(i)$ is black.
There is the following practical description of $VW$: it is a product $U\times I_r^\op$
where $U$ is the poset of all pairs $(t,x)$ such that $t\in \{-1,0,1,\dots,s\}$ and
$x\in \{0,1,2,\dots,s,s+1\}$ and $t\le x$. The ordering on $U$ is given by $(t_0,x_0)\le (t_1,x_1)$
if and only if $t_0\le t_1$ and $x_0\le x_1$\,. In this description, $V\subset VW$ corresponds to the
set of all triples $(t,x,y)$, or better $((t,x),y)$, where $x=s+1$. (This description of $VW$ and $U$ is practical
when it comes to drawing $U$. It would be somewhat misleading if we were interested in the
naturality properties with respect to $\sigma$ as a variable. But it was agreed that $\sigma$ is fixed
in these proofs.)

Here is a sketch of $U$ which at the same time indicates a simplicial subset
$L_U$ of $NU$ generated by nondegenerate $2$-simplices:
 
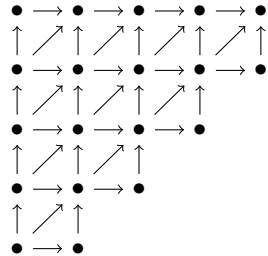
\[
\begin{tikzpicture}[descr/.style={fill=white}, baseline=(current bounding box.base)]
\matrix[column sep=4.0mm, row sep=4.0mm]{
        \node (05) {$\bullet$}; & \node (15) {$\bullet$};  & \node (25) {$\bullet$}; & \node (35) {$\bullet$}; & \node (45) {$\bullet$};  \\
        \node (04) {$\bullet$}; & \node (14) {$\bullet$};  & \node (24) {$\bullet$}; & \node (34) {$\bullet$}; & \node (44) {$\bullet$}; \\
        \node (03) {$\bullet$}; & \node (13) {$\bullet$};  & \node (23) {$\bullet$}; & \node (33) {$\bullet$}; & \\
        \node (02) {$\bullet$}; & \node (12) {$\bullet$};  & \node (22) {$\bullet$}; & & \\
        \node (01) {$\bullet$}; & \node (11) {$\bullet$};  & & & \\
    };

\foreach \x [remember=\x as \lastx (initially 0)] in {1,...,4}{
	    \path[draw, ->] (\lastx5)--(\x5);
}
\foreach \y [remember=\y as \lasty (initially 1)] in {1,...,4}{
\foreach \x [remember=\x as \lastx (initially 0)] in {1,...,\y}{
	    \path[draw, ->] (\lastx\y)--(\x\y);
}
}
\foreach \y [remember=\y as \lasty (initially 1)] in {2,...,5}{
	    \path[draw, ->] (0\lasty)--(0\y);
	    \path[draw, ->] (1\lasty)--(1\y);
}
\foreach \y [remember=\y as \lasty (initially 2)] in {3,...,5}{
	    \path[draw, ->] (2\lasty)--(2\y);
}
\foreach \y [remember=\y as \lasty (initially 3)] in {4,...,5}{
	    \path[draw, ->] (3\lasty)--(3\y);
}
\path[draw, ->] (44)--(45);
\path[draw, ->] (01)--(12); \path[draw, ->] (02)--(13); \path[draw, ->] (03)--(14); \path[draw, ->] (04)--(15);
\path[draw, ->] (12)--(23); \path[draw, ->] (13)--(24); \path[draw, ->] (14)--(25);
\path[draw, ->] (23)--(34); \path[draw, ->] (24)--(35);
\path[draw, ->] (34)--(45);
\end{tikzpicture}
\]
Let $K_U\subset L_U$ be the simplicial subset generated by the nondegenerate 1-simplices and 2-simplices
in the next sketch:
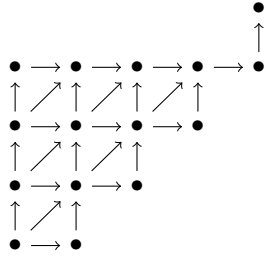
\[
\begin{tikzpicture}[descr/.style={fill=white}, baseline=(current bounding box.base)]
\matrix[column sep=4.0mm, row sep=4.0mm]{
       	& & & & \node (45) {$\bullet$};  \\
        \node (04) {$\bullet$}; & \node (14) {$\bullet$};  & \node (24) {$\bullet$}; & \node (34) {$\bullet$}; & \node (44) {$\bullet$}; \\
        \node (03) {$\bullet$}; & \node (13) {$\bullet$};  & \node (23) {$\bullet$}; & \node (33) {$\bullet$}; & \\
        \node (02) {$\bullet$}; & \node (12) {$\bullet$};  & \node (22) {$\bullet$}; & & \\
        \node (01) {$\bullet$}; & \node (11) {$\bullet$};  & & & \\
    };

\foreach \y [remember=\y as \lasty (initially 1)] in {1,...,4}{
\foreach \x [remember=\x as \lastx (initially 0)] in {1,...,\y}{
	    \path[draw, ->] (\lastx\y)--(\x\y);
}
}
\foreach \y [remember=\y as \lasty (initially 1)] in {2,...,4}{
	    \path[draw, ->] (0\lasty)--(0\y);
	    \path[draw, ->] (1\lasty)--(1\y);
}
\foreach \y [remember=\y as \lasty (initially 2)] in {3,...,4}{
	    \path[draw, ->] (2\lasty)--(2\y);
}
\foreach \y [remember=\y as \lasty (initially 3)] in {4,...,4}{
	    \path[draw, ->] (3\lasty)--(3\y);
}
\path[draw, ->] (44)--(45);
\path[draw, ->] (01)--(12); \path[draw, ->] (02)--(13); \path[draw, ->] (03)--(14);
\path[draw, ->] (12)--(23); \path[draw, ->] (13)--(24);
\path[draw, ->] (23)--(34);
\end{tikzpicture}
\]
Let $L\subset NVW$ consist of all simplices which are
taken to simplices of $L_U$ under the projection, and
let $K\subset L$ be the preimage of $K_U$\,. The first and most important thing we want to show here is that the restriction map
\begin{equation} \label{eqn-atomic} \rmap_\fino(VW,Y) \lra \rmap_\fino(K,Y) \end{equation}
is a weak equivalence. ---
By repeated use of corollary~\ref{cor-innerhorn}, the restriction map
\[
\rmap_\fino(VW,Y) \lra  \rmap_\fino(L,Y)
\]
is a weak equivalence.
Therefore we ought to show that the restriction map
\[  \rmap_\fino(L,Y) \lra \rmap_\fino(K,Y) \]
is a weak equivalence. To that end we write the inclusion $K\hookrightarrow L$ as a composition of
inclusions
\[ K=K(0)\subset K(1)\subset \cdots \subset K(v-1)\subset K(v)=L
\]
so that $K(j)$ is obtained from $K(j-1)$ by attaching the preimage of one more nondegenerate
$2$-simplex in $L_U$\,, for $j=1,2,\dots,v$. The recommended order in which these 2-simplices of $L_U$
should be called up is indicated in the following picture:
 
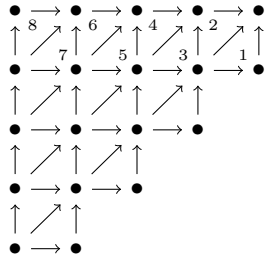
\[
\begin{tikzpicture}[descr/.style={fill=white}, baseline=(current bounding box.base)]
\matrix[column sep=4.0mm, row sep=4.0mm]{
        \node (05) {$\bullet$}; & \node (15) {$\bullet$};  & \node (25) {$\bullet$}; & \node (35) {$\bullet$}; & \node (45) {$\bullet$};  \\
        \node (04) {$\bullet$}; & \node (14) {$\bullet$};  & \node (24) {$\bullet$}; & \node (34) {$\bullet$}; & \node (44) {$\bullet$}; \\
        \node (03) {$\bullet$}; & \node (13) {$\bullet$};  & \node (23) {$\bullet$}; & \node (33) {$\bullet$}; & \\
        \node (02) {$\bullet$}; & \node (12) {$\bullet$};  & \node (22) {$\bullet$}; & & \\
        \node (01) {$\bullet$}; & \node (11) {$\bullet$};  & & & \\
    };

\foreach \x [remember=\x as \lastx (initially 0)] in {1,...,4}{
	    \path[draw, ->] (\lastx5)--(\x5);
}
\foreach \y [remember=\y as \lasty (initially 1)] in {1,...,4}{
\foreach \x [remember=\x as \lastx (initially 0)] in {1,...,\y}{
	    \path[draw, ->] (\lastx\y)--(\x\y);
}
}
\foreach \y [remember=\y as \lasty (initially 1)] in {2,...,5}{
	    \path[draw, ->] (0\lasty)--(0\y);
	    \path[draw, ->] (1\lasty)--(1\y);
}
\foreach \y [remember=\y as \lasty (initially 2)] in {3,...,5}{
	    \path[draw, ->] (2\lasty)--(2\y);
}
\foreach \y [remember=\y as \lasty (initially 3)] in {4,...,5}{
	    \path[draw, ->] (3\lasty)--(3\y);
}
\path[draw, ->] (44)--(45);
\path[draw, ->] (01)--(12); \path[draw, ->] (02)--(13); \path[draw, ->] (03)--(14); \path[draw, ->] (04)--(15);
\path[draw, ->] (12)--(23); \path[draw, ->] (13)--(24); \path[draw, ->] (14)--(25);
\path[draw, ->] (23)--(34); \path[draw, ->] (24)--(35);
\path[draw, ->] (34)--(45);

    \node[text width=0.1cm] at (1.0,1.4) {{\tiny 2}};
    \node[text width=0.1cm] at (0.2,1.4) {{\tiny 4}};
    \node[text width=0.1cm] at (-0.6,1.4) {{\tiny 6}};
    \node[text width=0.1cm] at (-1.4,1.4) {{\tiny 8}};

    \node[text width=0.1cm] at (1.4, 1) {{\tiny 1}};
    \node[text width=0.1cm] at (0.6, 1) {{\tiny 3}};
    \node[text width=0.1cm] at (-0.2, 1) {{\tiny 5}};
    \node[text width=0.1cm] at (-1, 1) {{\tiny 7}};

\end{tikzpicture}
\]
Then it remains to show that the restriction map
\[  \rmap_\fino(K(j),Y) \lra \rmap_\fino(K(j-1),Y) \]
is a weak equivalence, for $j=1,2,\dots,v$. There are two types of $j$ to be considered:
type I, which in the above picture comprises $j=2,4,6,8$, and type II, which in the above
picture comprises $j=1,3,5,7$.
For $j$ of type I we can rely on lemma~\ref{lem-factlift}. In more detail,
for $j$ of type I there is a pushout square of simplicial sets
\begin{equation} \label{eqn-typeI}
	\begin{tikzpicture}[descr/.style={fill=white}, baseline=(current bounding box.base)] ]
	\matrix(m)[matrix of math nodes, row sep=2.5em, column sep=2.5em,
	text height=1.5ex, text depth=0.25ex]
	{
	K(j-1) & K(j)  \\
	\Delta[1] \times \Delta[r] & \Delta[2] \times \Delta[r] \\
	};
	\path[->,font=\scriptsize]
		(m-1-1) edge node [auto] {} (m-1-2)
		(m-2-1) edge node [auto] {} (m-1-1)
		(m-2-1) edge node [auto] {} (m-2-2)
		(m-2-2) edge node [auto] {} (m-1-2);
	\end{tikzpicture}
\end{equation}
This leads to a homotopy pullback square
\[
	\begin{tikzpicture}[descr/.style={fill=white}, baseline=(current bounding box.base)] ]
	\matrix(m)[matrix of math nodes, row sep=2.5em, column sep=2.5em,
	text height=1.5ex, text depth=0.25ex]
	{
	\rmap_\fino(K(j-1),Y) & \rmap_\fino(K(j),Y)  \\
	\rmap_\fino(\Delta[1]\times\Delta[r],Y) & \rmap_\fino(\Delta[2]\times\Delta[r],Y) \\
	};
	\path[<-,font=\scriptsize]
		(m-1-1) edge node [auto] {} (m-1-2)
		(m-2-1) edge node [auto] {} (m-1-1)
		(m-2-1) edge node [auto] {} (m-2-2)
		(m-2-2) edge node [auto] {} (m-1-2);
	\end{tikzpicture}
\]
and it suffices to show that the lower horizontal arrow is a weak equivalence.
This is a special case of lemma~\ref{lem-factlift} because we are dealing with a
very particular simplicial map from $\Delta[2]\times\Delta[r]$ to $N\fino$.
When $j$ is of type II we can rely on proposition~\ref{prop-innerhorn}.

To recapitulate, the restriction map~(\ref{eqn-atomic}) is a weak equivalence. Let $K^\circ=K\cap NVW^\circ$. The arguments which we employed to show that the map~(\ref{eqn-atomic}) is a
weak equivalence show also that the restriction map
\begin{equation} \label{eqn-atomic2} \rmap_\fino(NVW^\circ,Y) \lra \rmap_\fino(K^\circ,Y)
\end{equation}
is a weak equivalence. In order not to interrupt the flow of ideas, we have isolated
a sketch of that in remark~\ref{rem-otherres} below. --- Let $J_U\subset K_U$ be the simplicial subset generated by the nondegenerate 1-simplex which
connects the vertices $(s,s)$ and $(s,s+1)$. Let $H_U\subset J_U$ be the simplicial subset generated by the single vertex $(s,s+1)$.
Let $J\subset K$ and $H\subset K$ be the preimages of $J_U$ and $H_U$, respectively.
Let $J^\circ=J\cap K^\circ$ and $H^\circ=H\cap K^\circ$. Then we obtain
a commutative diagram of restriction maps
\[
	\begin{tikzpicture}[descr/.style={fill=white}, baseline=(current bounding box.base)] ]
	\matrix(m)[matrix of math nodes, row sep=2.5em, column sep=2.5em,
	text height=1.5ex, text depth=0.25ex]
	{
	\rmap_\fino(VW,Y) & \rmap_\fino(VW^\circ,Y)  \\
	\rmap_\fino(K,Y) & \rmap_\fino(K^\circ,Y) \\
	\rmap_\fino(J,Y) & \rmap_\fino(J^\circ,Y) \\
	\rmap_\fino(H,Y) & \rmap_\fino(H^\circ,Y) \\	
	};
	\path[->,font=\scriptsize]
		(m-1-1) edge node [auto] {} (m-1-2)
		(m-2-1) edge node [auto] {} (m-2-2)
		(m-3-1) edge node [auto] {} (m-3-2)
		(m-4-1) edge node [auto] {} (m-4-2);
	\path[->,font=\scriptsize]		
		(m-1-1) edge node [auto] {$\simeq$} (m-2-1)
		(m-1-2) edge node [auto] {$\simeq$} (m-2-2)
		(m-2-1) edge node [auto] {} (m-3-1)
		(m-2-2) edge node [auto] {} (m-3-2)
		(m-3-1) edge node [auto] {} (m-4-1)
		(m-3-2) edge node [auto] {} (m-4-2);
	\end{tikzpicture}
\]
Since $K$ is the pushout of $K^\circ \leftarrow J^\circ\to J$, the square in the middle of the diagram is
homotopy cartesian. The vertical arrows in the lower part of the diagram are weak equivalences,
for example by lemma~\ref{lem-monlift}.
Therefore in order to finish we only have to show that the
lower horizontal arrow in that diagram is a weak equivalence. But this is true by inspection. Note that $H\cong I_r^\op$ and
$H^\circ\cong I_r^\op\smin I_s^\op$~, and that $\sigma^{VW}$ takes all vertices of $H$ not in $H^\circ$
to the object $([1],1,\textup{black})$. \end{proof}

\begin{rem} \label{rem-otherres} {\rm The arguments which
gave us that~(\ref{eqn-atomic}) is a weak equivalence are also applicable to~(\ref{eqn-atomic2}).
Namely, let $U'\subset U$ be the full sub-poset consisting of all $(t,x)$ where $x\le s$. Then
$NVW^\circ$ is the union of $N(U\times(I_r\smin I_s)^\op)$ and $N(U'\times I_r^\op)$ inside
$N(U\times I_r^\op)=NVW$. In a wisely chosen model structure on simplicial spaces, this implies
that $NVW^\circ$ is the homotopy pushout of
\[  N(U\times(I_r\smin I_s)^\op)  \longleftarrow N(U'\times(I_r\smin I_s)^\op)
\longrightarrow  N(U'\times I_r^\op). \]
A similar homotopy pushout decomposition is available for $L^\circ:=L\cap NVW^\circ$. It follows that
the restriction map in $\rmap_\fino(-,Y)$
determined by the inclusion $L^\circ\to NVW^\circ$ is a weak equivalence.
Therefore we only need to investigate the restriction map in $\rmap_\fino(-,Y)$
determined by the inclusion of $K^\circ$ in $L^\circ$.
If $s=r$, then there is nothing to prove: $K^\circ=L^\circ$.
If $s<r$, then in the diagram which replaces diagram~(\ref{eqn-typeI}), the
simplicial set $\Delta[r-s-1]$ replaces $\Delta[r]$.
}
\end{rem}

\smallskip
\begin{proof}[Proof of proposition~\ref{prop-shadowbox}] We take $q=2$. For a simplicial space $Y$ over $N\fino$
and $\sigma\in (N\fino)_r$ let the part of $F_j(Y)$ which is taken to $\sigma$ by the reference map
$F_j(Y)\to N\fino$ be defined as follows ($j=0,1,2$).
\begin{itemize}
\item[$j=0$]: $\rmap_\fino(N(I(\sigma)^\op),Y)$.
\item[$j=1$]: $\rmap_\fino(N(VW(\sigma)^\op),Y)$.
\item[$j=2$]: $\rmap_\fino(N(VW^\circ(\sigma)^\op),Y)$.
\end{itemize}
We leave it to the reader to specify the face and degeneracy operators in $F_j(Y)$. Obviously they
should be compatible with the face and degeneracy operators in $N\fino$\,.
The natural transformations $F_0\Leftarrow F_1\Rightarrow F_2$ are the restriction maps of lemma~\ref{lem-shadow}.

(i) We observe that $\rmap_\fino(N(I(\sigma)^\op),Y)$ is identified with the part
of $Y$ which projects to $\sigma\in (N\fino)_r$ since $N(I(\sigma)^\op)$ is free on one generator.
In that sense, $F_0=\id$. 

(ii) We need to argue that $F_2=E\alpha^*$ for some homotopy
functor $E$ from simplicial spaces over $N\finplus$ to simplicial spaces over $N\fino$\,. This is clear since
$\sigma^W$ maps $VW^\circ(\sigma)^\op$ to the white subcategory of $\fino$~, which is the image of the full embedding
$\alpha\co \finplus\to \fino$\,.

(iii) Let $Z$ be a simplicial space over $N\finplus$\,.
The part of $E(Z)$ over $\sigma\in (N\fino)_r$ is by construction
\[ E(Z)(\sigma)=\rmap_\finplus(N(VW^\circ(\sigma)^\op),Z) \]
where we use $\alpha^{-1}\sigma^{VW}$ as the reference functor from $VW^\circ(\sigma)^\op$ to $\finplus$. If $\sigma=\alpha\tau$
for some $\tau\in (N\finplus)_r$, then $VW^\circ(\sigma)\cong I(\sigma)$ and $\alpha^{-1}\sigma^{VW}$ agrees with $\tau$ under this
identification of the source categories. Therefore
\[ \alpha^*E(Z)(\tau)= E(Z)(\alpha\tau)=\rmap_\finplus(N(VW^\circ(\alpha\tau)^\op),Z) \cong Z(\tau). \]

(iv) We need to show that the natural transformations $F_0\Leftarrow F_1\Rightarrow F_2$
specialize to weak equivalences when evaluated on $Y=\varphi^*\config(D^m)$ or $Y=\delta^*\config(D^m)$.
But that is the content of lemma~\ref{lem-shadow}.
\end{proof}

\section{Spaces of long knots and higher analogues} \label{sec-knottynew}
\subsection{Statement of results}
Let $D^m_\circ$ be the punctured disk, $D^m_\circ=D^m\smin \{0\}$.
In this section, we prove

\begin{thm}[(Alexander trick for configuration categories)]\label{thm-alexanderk}
For any $k>0$ or $k = \infty$, and for all integers $m,n\ge 0$, the restriction map
\[
\rmap_{\finplus}(\config(D^m;k),\config(D^n)) \lra
\rmap_{\finplus}(\config(D^m_\circ;k),\config(D^n))
\]
is a weak homotopy equivalence. Therefore the space $\rmap_{\finplus}^\partial(\config(D^m;k),\config(D^n))$
is weakly contractible.
\end{thm}

Before we go into the proof, we deduce from this result some of the main theorems stated in the introduction. Let $\imap_\partial(D^m,D^n)$ be the space of injective continuous maps from $D^m$ to $D^n$, subject to the usual boundary conditions,
with the compact-open topology (cf.~remark~\ref{rem-imap}).

\begin{thm}\label{thm:disks} 
When $n - m$ is at least $3$, there is a homotopy cartesian square
\begin{equation} \label{eqn-alexanderfacto}
\begin{tikzpicture}[descr/.style={fill=white}, baseline=(current bounding box.base)],
	\matrix(m)[matrix of math nodes, row sep=2.5em, column sep=2.5em,
	text height=1.5ex, text depth=0.25ex]
	{
	\emb_\partial(D^m,D^n) & \imap_\partial(D^m,D^n) \\
	\imm_\partial(D^m,D^n) & \Omega^m \rmap_{\fin}(\config(\RR^m), \config(\RR^n)) \\
	};
	\path[->,font=\scriptsize]
		(m-1-1) edge node [auto] {$\textup{incl.}$} (m-1-2);
	\path[->,font=\scriptsize]
		(m-2-1) edge node [auto] {} (m-2-2);
	\path[->,font=\scriptsize]
		(m-1-1) edge node [left] {} (m-2-1);
	\path[->,font=\scriptsize] 		
		(m-1-2) edge node [right] {} (m-2-2);
	\end{tikzpicture}
\end{equation}
where $\imap_\partial(D^m,D^n)$ is contractible by the Alexander trick.  The lower horizontal map factors through the Smale-Hirsch map $\imm_\partial(D^m,D^n) \to \Omega^m V_{m,n}$ where $V_{m,n}$ is the space of linear injective maps from $\RR^m$ to $\RR^n$. Moreover, all the maps in the square are $m$-fold loop maps.
\end{thm}
\begin{proof}
The right vertical map is the composite
\begin{equation*}
\imap_\partial(D^m,D^n) \to \rmap_{\fin}^{\partial}(\config^\loc(D^m), \config^\loc(D^n)) \to \Omega^m \rmap_{\fin}(\config(\RR^m), \config(\RR^n)) \;.
\end{equation*}
Viewing the middle space as a functor on the variable $D^m$ (i.e. on the category of manifolds with boundary diffeomorphic to $S^{m-1}$) the second map is the  homotopy sheafification of that functor with respect to usual covers (cf. \cite[Proposition 7.6]{BoavidaWeiss}). This is easy to deduce from the formula for the homotopy sheafification, and the equivalence $\config^\loc(U;k) \simeq \config(\RR^m;k)$ which holds for every manifold $U$ of the form $C \amalg \RR^m$ where $C$ is a boundary collar. It follows that the second map in the display is a weak equivalence, as $F$ is a homotopy $\ms J_1$-sheaf by proposition \ref{prop-localconfigsheafbdry}. The first map in the display can itself be described as the composite
\[
\imap_\partial(D^m,D^n) \to \rmap_{\finplus}^\partial(\config(D^m),\config(D^n))
\to \rmap_{\fin}^{\partial}(\config^\loc(D^m),\config^\loc(D^n))
\]
where the left-hand map is a weak equivalence by theorem \ref{thm-alexanderk}.

It follows that square (\ref{eqn-alexanderfacto}) is related to the square of corollary \ref{cor-mainbdry} by termwise weak equivalences, and so it is homotopy cartesian.
\end{proof}

\begin{thm}\label{thm:disksk}
There is a homotopy fiber sequence
\begin{equation} \label{eqn-mainfacto}
\begin{tikzpicture}[descr/.style={fill=white}, baseline=(current bounding box.base)],
	\matrix(m)[matrix of math nodes, row sep=2.5em, column sep=2.5em,
	text height=1.5ex, text depth=0.25ex]
	{
	T_k \emb_\partial(D^m,D^n) & \\
	T_1 \emb_\partial(D^m,D^n) & \Omega^m \rmap_{\fin}(\config(\RR^m;k), \config(\RR^n)) \\
	};
	\path[->,font=\scriptsize]
		(m-2-1) edge node [auto] {} (m-2-2);
	\path[->,font=\scriptsize]
		(m-1-1) edge node [left] {} (m-2-1);
	\end{tikzpicture}
\end{equation}
of $m$-fold loop spaces.
\end{thm}

\begin{proof}
That this is a homotopy fiber sequence follows from theorems \ref{thm-alexanderk} and \ref{thm-mainbdry}. The horizontal
arrow is clearly a map of $m$-fold loop spaces. Therefore the term in the upper left automatically
acquires a structure of $m$-fold loop space as the homotopy fiber of the horizontal arrow. This promotes the homotopy fiber
sequence to one of $m$-fold loop spaces. (This solution is admittedly somewhat tautological and may be found disappointing. 
See however remark~\ref{rem:disksk} below.)
\end{proof}

\begin{rem} \label{rem:disksk} {\rm
Theorem \ref{thm:disksk} and its proof also imply that, for each $k$, the map $\emb_\partial(D^m,D^n) \to  T_k \emb_\partial(D^m,D^n)$ is an $E_m$-algebra map. Indeed we can think of it
as the map of horizontal homotopy fibers determined by the commutative square of $E_m$-algebras
\[
\begin{tikzpicture}[descr/.style={fill=white}, baseline=(current bounding box.base)],
	\matrix(m)[matrix of math nodes, row sep=2.5em, column sep=2.5em,
	text height=1.5ex, text depth=0.25ex]
	{
	\emb_\partial(D^m,D^n) & \imap_\partial(D^m,D^n)   \\
	T_1 \emb_\partial(D^m,D^n) & \Omega^m \rmap_{\fin}(\config(\RR^m;k), \config(\RR^n)) \; . \\
	};
\path[->,font=\scriptsize]
		(m-1-1) edge node [auto] {} (m-1-2);
	\path[->,font=\scriptsize]
		(m-2-1) edge node [auto] {} (m-2-2);
	\path[->,font=\scriptsize]
		(m-1-1) edge node [left] {} (m-2-1);
	\path[->,font=\scriptsize] 		
		(m-1-2) edge node [right] {} (m-2-2);
	\end{tikzpicture}
\]
}
\end{rem}

\subsection{The cases where $k=\infty$.} This is about the proof of theorem~\ref{thm-alexanderk} in the cases $k=\infty$.
Let $Z^{(m)}=\config(D^m)$.

\begin{lem} \label{lem-shadowboxapp} The map
\[ \rmap_\fino(\varphi^*Z^{(m)},\delta^*Z^{(n)}) \lra \rmap_\finplus(\alpha^*\varphi^*Z^{(m)},\alpha^*\delta^*Z^{(n)}) \]
induced by $\alpha^*$ is a weak equivalence.
\end{lem}

\begin{proof}
In the diagram
\begin{equation*}
\begin{tikzpicture}[descr/.style={fill=white}, baseline=(current bounding box.base)],
	\matrix(m)[matrix of math nodes, row sep=2.5em, column sep=2.0em,
	text height=1.5ex, text depth=0.25ex]
	{
	\rmap_\fino(\varphi^*Z^{(m)},\delta^*Z^{(n)}) & \rmap_\finplus(\alpha^*\varphi^*Z^{(m)},\alpha^*\delta^*Z^{(n)}) \\
	\rmap_\finplus(\alpha^* E \alpha^*\varphi^*Z^{(m)},\alpha^* E \alpha^*\delta^*Z^{(n)}) & \rmap_\finplus(E\alpha^*\varphi^*Z^{(m)},E\alpha^*\delta^*Z^{(n)}) \\
	};
	\path[->,font=\scriptsize]
		(m-1-1) edge node [auto] {} (m-1-2);
	\path[->,font=\scriptsize]
		(m-2-2) edge node [auto] {} (m-2-1);
	\path[->,font=\scriptsize] 		
		(m-1-2) edge node [right] {} (m-2-2);
	\end{tikzpicture}
\end{equation*}
the composition of the first two arrows and the composition of the last two arrows are weak equivalences by proposition \ref{prop-shadowbox}.
\end{proof}

\begin{lem} \label{lem-faceboxapp} The map
\[ \rmap_\finplus(Z^{(m)},Z^{(n)}) \lra \rmap_\fino(\delta_*\varphi^*Z^{(m)},Z^{(n)}) \]
given by pre-composition with the map $\delta_*\varphi^*Z^{(m)}\to Z^{(m)}$ of proposition~\ref{prop-genpos}
is a weak equivalence.
\end{lem}

\begin{proof} This follows from proposition~\ref{prop-genpos} since $Z^{(n)}$ is conservative over $N\finplus$\,. \end{proof}

\medskip
\begin{proof}[Proof of theorem~\ref{thm-alexanderk}, cases $k=\infty$] Noting that
\[  \rmap_\fino(\delta_*\varphi^*Z^{(m)},Z^{(n)})= \rmap_\fino(\varphi^*Z^{(m)},\delta^*Z^{(n)}) \]
we may compose the weak equivalences in lemmas~\ref{lem-shadowboxapp} and~\ref{lem-faceboxapp}. The result is a
weak equivalence
\[  \rmap_\finplus(Z^{(m)},Z^{(n)}) \lra \rmap_\finplus(\alpha^*\varphi^*Z^{(m)},Z^{(n)}). \]
If we use the identification $\alpha^*\varphi^*Z^{(m)}\simeq \config(D^m_\circ)$, over $N\finplus$\,, then we
recognize this map as the map given by restriction from $\config(D^m)$ to $\config(D^m_\circ)$. \end{proof}

\begin{rem}\label{aroneturchin} {\rm
Arone-Turchin \cite{AroneTurchin} and \cite{Turchin2} have a description of
\[
\overline{\emb}_{\partial}(D^m, D^n) := \hofiber \bigl( \emb_{\partial}(D^m, D^n) \hookrightarrow \imm_{\partial}(D^m, D^n) \bigr)
\]
in operadic terms: in their language, it is the space of derived maps between two infinitesimal bimodules over $E_m$ associated with $D^m$ and $D^n$. This works when $n - m$ is at least three. Dwyer-Hess \cite{DwyerHess1} and Turchin \cite{Turchin} show that, in the case $m = 1$, this space of derived infinitesimal bimodule maps has an $(m+1)$-fold delooping given by the space of derived maps from $E_m$ to $E_n$. Dwyer-Hess \cite{DwyerHess2} also announced an extension of this result for arbitrary $m \leq n$. The combination of these results of Arone-Turchin and Dwyer-Hess is in good agreement with theorem \ref{thm:disks}. Indeed, it shows that the vertical homotopy fibers in square (\ref{eqn-alexanderfacto}) have the same weak homotopy type. However, we have shifted the emphasis from trying to describe the homotopy fiber directly to constructing a \emph{homotopy fiber sequence} by means of a rather geometric map -- the lower horizontal map in square (\ref{eqn-alexanderfacto}). This can be important in some applications.

On the other hand, the Dwyer-Hess result is a theorem about fairly general operads and as such it has a different scope and applicability from our result. Moreover, the Arone-Turchin result admits a homological variant and one may wonder whether a homological Dwyer-Hess theorem exists to complement it.
}
\end{rem}

\subsection{The cases where $k<\infty$.}
Let $\finplusk$ be the full subcategory of $\finplus$ determined
by the objects $\uli\ell$ where $\ell\le k$. Let
\[ \iota_k : \finok \hookrightarrow \fino \]
be the inclusion of the full subcategory which is the preimage of $\finplusk\subset\finplus$ under the functor $\delta\co \fino \to \finplus$. This full subcategory of $\fino$ has objects $([\ell],a,c)$ where $\ell\le k$ if $c=$ black and $\ell\le k+1$ if $c=$ white. By construction the functors $\delta\co \fino\to \finplus$ and $\alpha\co \finplus\to \fino$ have truncated analogues:
\[ \finok \to \finplusk~, \qquad \finplusk \to \finok\,. \]
By contrast there is no \emph{forgetful} functor
(of type $\varphi$) from $\finok$ to $\finplusk$. This is not a serious problem. There is a truncated version of
proposition \ref{prop-genpos} which reads as follows: the base change of the map
\[
\delta_* \varphi^* \config(M) \to \config(M) \times M
\]
along $\finplusk \to \finplus$ is still a conservatization map. This map may also be described as
\[
\delta_* \iota_k^*\varphi^* \config(M) \to \config(M; k) \times M \; .
\]

The truncated version of the shadowing lemma holds since if $\sigma \in \finok$ then $\sigma^{VW} \in \finok$. Therefore the truncated version of proposition \ref{prop-shadowbox} also holds with $\fino$ replaced by $\finok$ and $\finplus$ replaced by $\finplusk$.

\appendix
\section{Derived mapping spaces and homotopy Kan extensions}
The standard constructions (e.g. derived functors) one performs in homotopy theory only depend on the notion of
weak equivalence. Nevertheless, it is often the case in practice that model structures are available and using them
simplifies matters considerably. They are helpful as a way of calculation, but not essential to the statements of results,
and so we tried to avoid overemphasising their role (by relocating them to this Appendix, for example).

\subsection{Derived mapping spaces}\label{section-Rmap}

If $\sC$ is a simplicial model category, then the mapping space functor $\map(-,-) : \sC^{\op} \times \sC \to \sS$ is a right
Quillen functor (it preserves fibrations and trivial fibrations), and so admits a right derived functor,
$\rmap(-,-) : \sC^{\op} \times \sC \to \sS$. To compute $\rmap(X,Y)$, take a (functorial) cofibrant
replacement $X^{c} \to X$ and a (functorial) fibrant replacement $Y \to Y^{f}$ and declare
\[
\rmap(X,Y) = \map(X^c, Y^f)
\]
However, the weak homotopy type of $\rmap(X,Y)$ only depends on the class of weak equivalences of $\sC$ and a model structure is not needed to define it \cite{DwyKa2}. As such, all simplicial model structures on $\sC$ sharing the same notion of weak equivalence compute the correct derived mapping space. We record here some of its main properties:

\begin{enumerate}
\item $\pi_0 \map(X,Y) \cong \mor_{\mathsf{Ho}(\sC)}(X,Y)$
\item $\rmap(-,-)$ respects weak equivalences in each variable
\item If $F : \sC \leftrightarrows \sD : G$  is a Quillen pair, then there is a weak equivalence
\[
\rmap(\LL F(X), Y) \simeq \rmap(X, \RR G(Y))
\]
where $\LL F$ denotes the left derived functor of $F$ and $\RR G$ the right derived functor of $G$ (cf. section \ref{A:Kanext}).
\item If $I$ is a small category and $F : I \to \sC$ a functor, then
\[
\rmap(\hocolimsub{i \in I} F(i), Y) \simeq \holimsub{i \in I} \rmap(F(i), Y)
\]
and
\[
\rmap(X, \holimsub{i \in I} F(i)) \simeq \holimsub{i \in I} \rmap(X, F(i))
\]
\end{enumerate}

\subsection{Homotopy Kan extensions}\label{A:Kanext}
Let $j\co \sC \to \sD$ be a functor between small categories (possibly enriched over spaces). It induces, by
precomposition with $j$, a functor $j^*\co\PSh(\sD) \to \PSh(\sC)$. (Here $\PSh(\sC)$ indicates
\emph{presheaves on $\sC$} and so denotes the category of contravariant functors from $\sC$ to $\sS$.)
Then $j^*$ has both a left and a right adjoint
\[
j^* : \PSh(\sD) \leftrightarrows \PSh(\sC) : j_* \quad \quad j_! : \PSh(\sC) \leftrightarrows
\PSh(\sD) : j^*
\]

The functor $\LK{j}$ (respectively, $\RK{j}$) is usually called the left (respectively, right) Kan extension along $j$.

Let us focus on the \emph{left} Kan extension. If we choose the projective model structure on
$\PSh(\sC)$ and $\PSh(\sD)$, then $(\LK{j}, j^*)$ becomes a Quillen pair. Indeed, $j^*$ preserves fibrations and $\LK{j}$ preserves (generating) cofibrations since, being a left adjoint, $\LK{j}$ satisfies
\[ \LK{j}(\mor_{\sC}(-,c) \times K)~\cong~\mor_{\sD}(-,j(c)) \times K \]
for any object $c$ of $\sC$ and object $K$ of $\sS$. As a consequence, for $F$ in $\PSh(\sC)$ and
$G$ in $\PSh(\sD)$ the map
\begin{equation}\label{eq:lkan}
\rmap(\hLK{j}F, G) \to \rmap(F,\RR j^*G)
\end{equation}
obtained by applying $\RR j^*$
and then precomposing with the derived unit
\[ \epsilon : \id \to (\RR j^*)(\hLK{j}) \]
evaluated at $F$, is a weak equivalence of spaces.
In other words, the pair $(\hLK{j}\,, \RR j^*)$ forms a derived
adjunction, i.e. a simplicial adjunction between the $\sS$-categories associated to $\PSh(\sC)$ and $\PSh(\sD)$ by Dwyer-Kan localization \cite{DwyKa2}.

Because $j^*$ preserves weak equivalences, $j^*$ coincides both with its left derived functor $\LL j^*$ and with its right derived functor $\RR j^*$. Then the map~(\ref{eq:lkan}) takes the form
\begin{equation}\label{eq:lkan2}
\rmap(\hLK{j}F, G) \to \rmap(F,j^*G)
\end{equation}
for $F$ in $\PSh(\sC)$ and $G$ in $\PSh(\sD)$.

\begin{lem}\label{lem:LK1}
Fix $E \in \PSh(\sD)$. The map
\begin{equation}\label{eqn-ff}
\rmap(E,G) \lra \rmap(j^*E,j^*G)
\end{equation}
given by composition with $j$ is a weak equivalence for any $G \in \PSh(\sD)$ if and only if the derived counit
$
\hLK{j} j^*E \to E
$
is a weak equivalence. Dually, fixing $G \in \PSh(\sD)$, the map (\ref{eqn-ff}) is a weak equivalence for every $E \in \PSh(\sD)$ if and only if the derived unit
$
G \to \hRK{j} j^*G
$
is a weak equivalence.
\end{lem}
\begin{proof}
The map in question factors as
\[ \rmap(E,G) \lra \rmap(\hLK{j} j^*E,G) \]
(composition with the derived counit $\hLK{j} j^*E\to E$) followed by the adjunction
morphism (\ref{eq:lkan}) with $F:=j^*E$. This proves that the first two statements are equivalent. The equivalence
involving the unit is dual.
\end{proof}

\section{Homotopy theory of (fiberwise) complete Segal spaces} \label{sec-model}

\subsection{Simplicial spaces}
The category $s \sS$ of simplicial spaces, being a functor category, has two standard simplicial model structures with
degreewise weak equivalences. These are the projective (with degreewise fibrations) and the injective (with degreewise cofibrations) model
structures. It is well-known that the identity functor induces a Quillen equivalence between the two. Thus whenever we refer to the
standard model structure on simplicial spaces we mean either one of the two,
unless for some particular reason we find it convenient to pick one of these.

Let $s \sS_{/B}$ denote the category of simplicial spaces over a fixed simplicial space $B$. A morphism is a simplicial
map $X \to Y$ over $B$. This category is enriched in $\sS$ by setting $\map_{B}(X, Y)$ to be the pullback (taken in $\sS$) of
\[
* \rightarrow \map(X,B) \leftarrow \map(X,Y)
\]
The left-hand map selects the reference map $X \to B$ and the right-hand map is given by post-composition with the
reference map $Y \to B$.

\begin{prop}
There is a simplicial model structure on $s \sS_{/B}$ in which a map over $B$ is a weak
equivalence/fibration/cofibration if it is so in $s\sS$.
\end{prop}
\begin{proof}
That this forms a model structure is immediate from the definitions. (And it is left proper and cellular \cite[Theorem IV.4.1.6]{Hirschhorn}.) It also satisfies (SM7) \cite[IX.9.1.5]{Hirschhorn} as one can verify by using the description of the mapping spaces above and the corresponding property for $\sS$, and so it is a simplicial model category.
\end{proof}

Given two objects $X$ and $Y$ in $s \sS_{/B}$, we denote by $\RR \map_{B}(X,Y)$ the derived mapping space with respect
to this model structure; it is weakly equivalent to the homotopy pullback of
\[
* \to \rmap(X,B) \leftarrow \rmap(Y,B)
\]
where $\RR \map$ in the diagram refers to the derived mapping spaces formed in $s\sS$ (with degreewise weak equivalences).

\subsection{Homotopy theory of fiberwise complete Segal spaces}
There are three important notions of weak equivalence for simplicial spaces over a fixed simplicial space $B$ that
we want to consider. We say that a map $f\co X \to Y$ of simplicial spaces over $B$ is
\begin{enumerate}
\item a \emph{degreewise} weak equivalence if $f_n\co X_n \rightarrow Y_n$ is a weak equivalence in spaces, for each $n \ge 0$;
\item a \emph{Dwyer-Kan} equivalence if it is fully faithful, i.e.
\[
\mor^h_X(x,y) \rightarrow \mor^h_Y(f(x), f(y))
\]
 is a weak equivalence of spaces, and $\mathsf{Ho}(f)\co \mathsf{Ho}(X) \rightarrow \mathsf{Ho}(Y)$ is essentially
 surjective (see section~\ref{sec-topcat});
\item a \emph{local} weak equivalence if $\rmap_{B}(f, Z)$ is a weak equivalence for every fiberwise complete Segal space $Z\to B$.
\end{enumerate}

\begin{thm} \label{thm-fiberwisecpl} Fix a Segal space $B$. There is a left proper, simplicial model structure on the category of
simplicial spaces over $B$ which is uniquely determined by the following data.
\begin{itemize}
\item An object $p\co X \to B$ is \emph{fibrant} if it is fiberwise complete and $p$ is a
fibration in $\sSp$.
\item A map $X \to Y$ over $B$ is a cofibration if it is a cofibration in $\sSp$\,.
\end{itemize}
A map $X \to Y$ between Segal spaces over $B$ is a weak equivalence if and only if it is a Dwyer-Kan equivalence.
More generally, a map $f\co X \to Y$ between any two simplicial spaces over $B$ is a weak equivalence if it is a local weak equivalence.
\end{thm}

In theorem~\ref{thm-fiberwisecpl} it is permitted to take $B$ equal to the terminal object, $B_n=*$ for all $n\ge 0$.
In that case the theorem describes
a model category structure on the category of simplicial spaces. This is Rezk's model category structure.
In order to prove this theorem, we use the following observation.

\begin{lem}\label{lem-cof+fib}
A model structure is uniquely determined by its class of fibrant objects and its class of cofibrations.
\end{lem}
\begin{proof}
The following argument is due to Joyal (unpublished). Suppose $\sC$ and $\sC^{\prime}$ are two model structures on the same underlying category possessing the same classes of cofibrations and fibrant objects. Recall that a cylinder object $A \otimes I$ is a factorisation $A \amalg A \hookrightarrow A \otimes I \overset{\sim}{\twoheadrightarrow} A$ of the fold map by a cofibration followed by a trivial fibration. Since $\sC$ and $\sC^{\prime}$ have the same class of cofibrations (and hence also the same class of trivial fibrations), $A \otimes I$ is a cylinder object for $\sC$ if and only if it is a cylinder object for $\sC^{\prime}$. This implies that two maps in $\sC$ are left homotopic if and only if they are left homotopic in $\sC^{\prime}$.

Assume without loss of generality that $A$ is cofibrant. Because the morphism set in the homotopy category $\Hom_{\mathsf{Ho}(\sC)}(A,X)$ can be described as the set of equivalence classes of $\Hom_{\sC}(A,X)$ under the equivalence relation of left homotopy (see, for example, \cite[7.4 - 7.5]{Hirschhorn}), it follows that the identity map induces an equivalence
$
\mathsf{Ho}(\sC) \to \mathsf{Ho}(\sC^{\prime})
$
of homotopy categories. This shows $\sC$ and $\sC^{\prime}$ also have the same classes of weak equivalences, and so the model structures coincide.
\end{proof}

We define two model structures on $\sSp_{/B}$ with the same set of cofibrations and show that they agree. In view of the lemma above, in order to do that we only need to verify that both have the same set of fibrant objects. This will give the model structure of Theorem \ref{thm-fiberwisecpl}.

\begin{prop} There is a simplicial model structure on $\sSp_{/B}$ which is uniquely determined by the following properties.
\begin{itemize}
\item An object $X \to B$ is fibrant if it is a fibration in the complete Segal space model structure
\item A map is a cofibration if it is a cofibration in $\sSp$.
\end{itemize}
Moreover, a map $X \to Y$ over $B$ such that $X$ and $Y$ are Segal spaces is a weak equivalence if and only if it is a Dwyer-Kan equivalence.
\end{prop}

We refer to this as the \emph{Rezk model structure}, since it is obtained by taking the overcategory model structure of Rezk's model structure on complete Segal spaces. A fibrant object in this model structure shall be called \emph{Rezk fibrant}.

\begin{prop}
There is a simplicial model structure on $\sSp_{/B}$ which is uniquely determined by the following properties.
\begin{itemize}
\item An object $p : X \to B$ is fibrant if $X$ is fiberwise complete over $B$ and $p$ is a fibration in the Segal space model structure.
\item A map is a cofibration if it is a cofibration in $\sSp$.
\end{itemize}
Moreover, a map $f: X \to Y$ over $B$ is weak equivalence if
\[
\map_B(f, Z) : \RR \map_{B}(Y, Z) \rightarrow \RR \map_{B}(X,Z)
\]
is a weak equivalence of spaces for every fibrant $Z$.
\end{prop}
\begin{proof}
Starting with the Segal space model structure on the category of simplicial spaces (i.e. whose fibrant objects are the Segal spaces),
take the corresponding model structure on the overcategory of simplicial spaces over $B$. The claimed model structure is
obtained as a left Bousfield localization at the set of morphisms
$$
\{ \Delta[0] \to E \xrightarrow{f} B \}_{f} \;
$$
in the category of simplicial spaces over $B$, where $E$ denotes the nerve of the groupoid with two
objects $x,y$ and two non-identity isomorphisms $x \to y$ and $y \to x$.
\end{proof}

We refer to this as the \emph{fiberwise complete Segal space} model structure. In order to make a distinction
with the other model structures, a fibrant object in this model structure shall be called \emph{fc fibrant}.

\begin{prop}\label{prop:fibrationSegal}
Let $B$ be a Segal space. A map $p : X \rightarrow B$ is a fibration in the Segal space model structure if and only if $X$ is a Segal space and $p$ is a degreewise fibration.
\end{prop}
\begin{proof}
If $f : X \rightarrow B$ is a fibration and $B$ is fibrant, then $X$ is necessarily fibrant since
$X \rightarrow B \rightarrow *$ is a composite of fibrations, hence a fibration. Thus, $X$ is fibrant in the
Segal space model structure, i.e. it is a Segal space. The statement now follows from \cite[3.3.16]{Hirschhorn}.
\end{proof}

\begin{prop}\label{prop:Rezkfibrant}
Suppose $p: X \to B$ is Rezk fibrant. Then
\begin{enumerate}
\item $X$ is a Segal space and $p$ is a degreewise fibration
\item $X$ is fiberwise complete over $B$.
\end{enumerate}
\end{prop}
\begin{proof}
Since the complete Segal space model structure is a Bousfield localization of the Segal space model structure,
a fibration in the complete model structure is in particular a fibration in the Segal space model structure. So,
by proposition \ref{prop:fibrationSegal}, X is a Segal space.

As before, let $E$ denote the nerve of the groupoid with two objects $x,y$ and two non-identity isomorphisms
$x \to y$ and $y \to x$. The map $\Delta[0] \to E$ selecting either $x$ or $y$ is by construction a trivial
cofibration in the (injective) complete Segal space model structure and $X \to B$, being a fibration,
satisfies the right lifting property with respect to $\Delta[0] \to E$. That is to say, $X$ is fiberwise
complete over $B$.
\end{proof}

It remains to show the converse, the statement that if $X \in \sSp_{/B}$ is fc fibrant then it is Rezk fibrant.

\begin{prop}\label{prop:fc=R}
Let $B$ be a Segal space. If $f\co X \to B$ is \emph{fc} fibrant, then $f$ is degreewise equivalent (over $B$) to a Rezk fibrant object $g\co Y \to B$.
\end{prop}

\begin{proof}
By the axioms of a model category, one can factor $f$ as
\[
X \xrightarrow{i} Y \xrightarrow{p} B
\]
where $i$ is a trivial cofibration and $p$ is a fibration in the complete Segal space model structure. By Proposition \ref{prop:Rezkfibrant},
the simplicial space $Y$ is necessarily a Segal space which is fiberwise complete over $B$. Thus $i$ is a Dwyer-Kan equivalence because in the complete Segal space model structure weak equivalences between Segal spaces are Dwyer-Kan equivalences.

Since $X$ and $Y$ are both fiberwise complete over $B$, it follows that $X$ is fiberwise complete over $Y$. Therefore it is enough to show that
if $i\co X\to Y$ is a Dwyer-Kan equivalence of Segal spaces and fiberwise complete, then $i$ is a degreewise equivalence. Given that $i$ is a
Dwyer-Kan equivalence, it suffices to show that the map $i_0\co X_0\to Y_0$ is a weak equivalence. \newline
For that we introduce the subspace
\[  X_0\times_\heq X_0\subset X_0\times X_0 \]
consisting of all pairs $(x,x')$ with $x,x'\in X_0$ such that $x$ is weakly equivalent to $x'$ in the Segal space $X$\,.
By definition, $X_0\times_\heq X_0$
is a union of path components of $X_0\times X_0$\,. Similarly we introduce $Y_0\times_\heq Y_0$~, a union of path components
of $Y_0\times Y_0$\,. In the commutative diagram
\[
	\begin{tikzpicture}[descr/.style={fill=white}, baseline=(current bounding box.base)] ]
	\matrix(m)[matrix of math nodes, row sep=2.5em, column sep=3.5em,
	text height=1.5ex, text depth=0.25ex]
	{
	X_1^\heq  & X_0\times_\heq X_0 & X_0  \\
	Y_1^\heq  & Y_0\times_\heq Y_0 & Y_0  \\
	};
	\path[->,font=\scriptsize]
		(m-1-1) edge node [auto] {$(d_0,d_1)$} (m-1-2)
		(m-1-2) edge node [auto] {\textup{2nd proj.}} (m-1-3)
		(m-2-1) edge node [auto] {$(d_0,d_1)$} (m-2-2)
		(m-2-2) edge node [auto] {\textup{2nd proj.}} (m-2-3)
		(m-1-1) edge node [auto] {$i_1$} (m-2-1)
		(m-1-2) edge node [auto] {$i_0 \times i_0$} (m-2-2)
		(m-1-3) edge node [auto] {$i_0$} (m-2-3);
	\end{tikzpicture}
\]
the outer rectangle is homotopy cartesian because $i$ is fiberwise complete and the left-hand square is homotopy
cartesian because $i$ is a Dwyer-Kan equivalence. Moreover the map $Y_1^\heq \to Y_0\times_\heq Y_0$ in the diagram
induces a surjection on $\pi_0$ by construction. It follows that the right-hand square of the diagram is homotopy cartesian;
see lemma~\ref{lem-hococomp} below.

Now choose $x\in X_0$ and hence $y=i_0(x)\in Y_0$.
The horizontal homotopy fiber over $x$ in the right-hand square, which we denote by $X_0 \times_{\heq} \{x\}$, is identified with the subspace of $X_0$ comprising all $x'\in X_0$ which are weakly equivalent to $x$. The horizontal
homotopy fiber over $y=i_0(x)$ in the right-hand square, which we denote by $Y_0 \times_{\heq} \{y\}$, is identified with the subspace of $Y_0$ comprising all $y'\in Y_0$ which are weakly equivalent to $y$. Since the right-hand square is homotopy cartesian, $i_0$ induces a weak equivalence $X_0 \times_\heq \{x\} \to Y_0 \times_\heq \{y\}$. So it remains only to show that the map $\pi_0(X_0)\to \pi_0(Y_0)$ induced by $i_0$ is
a bijection.

If it is not injective, then there are $x,x'$ in distinct path components of $X_0$ such that $y=i_0(x)$ and
$y'=i_0(x')$ are in the same path component of $Y_0$\,. Then $y$ and $y'$ are weakly equivalent. Therefore $x$ and $x'$ are weakly
equivalent (because $i$ is a Dwyer-Kan equivalence). Since $X_0 \times_\heq \{x\} \to Y_0 \times_\heq \{y\}$ is a weak
equivalence, $y$ and $y'$ are in distinct path components of $Y_0$ we get a contradiction. To show that the map $\pi_0(X_0)\to \pi_0(Y_0)$ induced by $i_0$
is surjective, choose $y'\in Y_0$\,. Since $i$ is a Dwyer-Kan equivalence,
there exists $x\in X_0$ such that $y=i(x)$ is weakly equivalent to $y'$. Since $X_0 \times_\heq \{x\} \to Y_0 \times_\heq \{y\}$ is a weak
equivalence, the path component of $y'$ is in the image of the map $\pi_0(X_0)\to \pi_0(Y_0)$.
\end{proof}

\begin{lem} \label{lem-hococomp} Let
\[
	\begin{tikzpicture}[descr/.style={fill=white}, baseline=(current bounding box.base)] ]
	\matrix(m)[matrix of math nodes, row sep=2.5em, column sep=2.5em,
	text height=1.5ex, text depth=0.25ex]
	{
	A & B & C  \\
	A^\prime & B^\prime & C^\prime  \\
	};
	\path[->,font=\scriptsize]
		(m-1-1) edge node [auto] {} (m-1-2)
		(m-1-2) edge node [auto] {} (m-1-3)
		(m-2-1) edge node [auto] {} (m-2-2)
		(m-2-2) edge node [auto] {} (m-2-3)
		(m-1-1) edge node [auto] {} (m-2-1)
		(m-1-2) edge node [auto] {} (m-2-2)
		(m-1-3) edge node [auto] {} (m-2-3);
	\end{tikzpicture}
\]
be a commutative diagram of spaces. Suppose that the left-hand square
 is homotopy cartesian, that the outer rectangle
 is also homotopy cartesian, and that the map from $A^{\prime}$ to $B^{\prime}$
in the diagram induces a surjection on $\pi_0$. Then the right-hand square
is also homotopy cartesian.
\end{lem}

\begin{proof} Choose $y\in B^{\prime}$ with image $z\in C^{\prime}$. It is enough to show that the map
\[  \hofiber_y\begin{bmatrix} B\\ \downarrow \\ B^{\prime} \end{bmatrix}
\lra \hofiber_z\begin{bmatrix} C\\ \downarrow \\ C^{\prime} \end{bmatrix} \]
determined by the diagram is a weak equivalence. Since $A^{\prime}\to B^{\prime}$ induces a surjection on $\pi_0$, we can assume that $y\in B^{\prime}$
is the image of some $x\in A^{\prime}$. Then our map of homotopy fibers becomes part of a larger diagram
\[  \hofiber_x\begin{bmatrix} A\\ \downarrow \\ A^{\prime} \end{bmatrix}\lra \hofiber_y\begin{bmatrix} B\\ \downarrow \\ B^{\prime} \end{bmatrix}
\lra \hofiber_z\begin{bmatrix} C\\ \downarrow \\ C^{\prime} \end{bmatrix}. \]
The composite map is a weak equivalence (since the outer rectangle is homotopy cartesian) and so is the left-hand map (since the left-hand square is homotopy cartesian). We complete the argument by invoking the two-out-of-three principle for weak equivalences. \end{proof}

\begin{prop}\cite[3.3.15]{Hirschhorn}\label{prop:HHfib}
Let $M$ be a model category and $L_{S} M$ the (left) Bousfield localization of $M$ with respect to a class $S$ of maps in $M$. If $f : X \to B$ is a fibration in $M$, $g : Y \to B$ a fibration in $L_{S} M$, and $h : X \to Y$ is a weak equivalence in $M$ that makes the triangle
\[
	\begin{tikzpicture}[descr/.style={fill=white}, baseline=(current bounding box.base)] ]
	\matrix(m)[matrix of math nodes, row sep=1.8em, column sep=1.5em,
	text height=1.5ex, text depth=0.25ex]
	{
	X & & Y \\
	  & B & \\
	};
	\path[->,font=\scriptsize]
		(m-1-1) edge node [auto] {$h$} (m-1-3);
	\path[->,font=\scriptsize]
		(m-1-1) edge node [left] { $f$ \,} (m-2-2);
	\path[->,font=\scriptsize]
		(m-1-3) edge node [right] { \, $g$} (m-2-2);
	\end{tikzpicture}
\]
commute, then $f$ is also a fibration in $L_{S} M$.
\end{prop}

\begin{cor}
An object $X \to B$ is Rezk fibrant if and only if it is \emph{fc} fibrant.
\end{cor}
\begin{proof}
A Rezk fibrant object is \emph{fc} fibrant by Proposition \ref{prop:Rezkfibrant}. The converse is a consequence of Propositions \ref{prop:fc=R} and \ref{prop:HHfib}.
\end{proof}


\section{Extension lemmas for maps to Segal spaces}\label{sec-ext}
We begin with a lemma which is well known and describes a
fundamental relationship between $\infty$-categories and Segal spaces. As usual let $\Delta[k]$ be the simplicial set
freely generated by one element $w$ in degree $k$.
For $i\in \{0,1,2,\dots,k\}$ let $\Lambda^i[k]$ be the $i$-th horn of $\Delta[k]$, the simplicial subset of
$\Delta[k]$ generated by the faces $d_jw$ for $j\ne i$. In the cases where $0<i<k$ we say that $\Lambda^i[k]$ is
an \emph{inner horn}. --- Let $g\co X\hookrightarrow X'$ be a map of simplicial sets which fits into a pushout square
\[
	\begin{tikzpicture}[descr/.style={fill=white}, baseline=(current bounding box.base)] ]
	\matrix(m)[matrix of math nodes, row sep=2.5em, column sep=2.5em,
	text height=1.5ex, text depth=0.25ex]
	{
	\Lambda^i[k] & \Delta[k]  \\
	X & X' \\
	};
	\path[right hook->,font=\scriptsize]
		(m-1-1) edge node [auto] {} (m-1-2);
	\path[->,font=\scriptsize]
		(m-1-1) edge node [auto] {} (m-2-1)
		(m-2-1) edge node [auto] {$g$} (m-2-2)
		(m-1-2) edge node [auto] {} (m-2-2);
	\end{tikzpicture}
\]
where $i\in \{1,2,\dots,k-1\}$. We may say that $g$ is a \emph{relative inner horn inclusion}.

\begin{lem} \label{lem-innerhorn} Let $g\co X\to X'$ be a relative inner horn inclusion.
Let $Z$ be a simplicial space which happens
to be a Segal space. Then the restriction map
\[  g^*\co \rmap(X',Z) \to \rmap(X,Z) \]
is a weak equivalence. \emph{(The derived mapping spaces are to be formed in a model structure
on the category of simplicial spaces where the weak equivalences are defined degreewise)}.
\end{lem}

\begin{proof} We can specify the model structure on the category of simplicial spaces in such a way that the
pushout square defining $g$ becomes a homotopy pushout square. Then the corresponding square of
derived mapping spaces (with target $Z$) becomes a homotopy pullback square. Therefore we have reduced
the proof to the situation where $g$ is the inner horn inclusion $\Lambda^i[k]\to \Delta[k]$. Next,
let $L[k]$ be the simplicial subset of $\Delta[k]$ generated by the 1-simplices $u_j^*w$ where $u_j\co [1]\to [k]$ are the monotone maps
defined near display~(\ref{eqn-nerve}). It is easy to see that the inclusion $L[k]\to \Lambda^i[k]$ is a
composition of relative inner horn inclusions. Since these are relative inner horn inclusions in which
the added simplex has dimension less than $k$,
we may assume per induction that the restriction
\[  \rmap(\Lambda^i[k],Z) \to \rmap(L[k],Z) \]
is a weak equivalence. Then it remains to show that the restriction
\begin{equation} \label{eqn-otherSegal}
\rmap(\Delta[k],Z)\to \rmap(L[k],Z)
\end{equation}
is a weak equivalence. Again, $L[k]$ can be described as the colimit and homotopy colimit of a diagram
involving $k-1$ copies of $\Delta[0]$ and $k$ copies of $\Delta[1]$. On applying $\rmap(-,Z)$ this
identifies $\rmap(L[k],Z)$ with the homotopy limit of a diagram involving $k-1$ copies of $\rmap(\Delta[0],Z)$ and $k$ copies of $\rmap(\Delta[1],Z)$.
Using this and using $\rmap(\Delta[j],Z)\simeq Z_j$ for all $j$, specifically $j=0,1$ and $j=k$, we
see that the Segal property of $Z$ is equivalent to the statement that~(\ref{eqn-otherSegal}) is a weak equivalence for
all $k\ge 2$. \end{proof}

\medskip
There is an \emph{internal} variant of lemma~\ref{lem-innerhorn}. To formulate it (proposition~\ref{prop-innerhorn}),
we require one more definition.

\begin{defn} \label{defn-rezkadj} {\rm Let $A,B,Z$ be simplicial spaces where $Z$ is a Segal space, and $B$ shall
be considered as a variable. We look for a simplicial space $Z^A$ which solves the adjunction problem
\[  \rmap(B,Z^A) \simeq \rmap(A\times B,Z). \]
(Form the derived mapping spaces in a model category structure
on the category of simplicial spaces where the weak equivalences are degreewise weak equivalences.)
Rezk \cite{Rezk} defines
\[  (Z^A)_r := \rmap(A\times\Delta[r],Z). \]
It turns out that $Z^A$ is again a Segal space, and there is indeed a zigzag of weak equivalences, natural in $B$,
relating $\rmap(B,Z^A)$ to $\rmap(A\times B,Z)$.
}
\end{defn}

\begin{prop} \label{prop-innerhorn} With the assumptions and notation of lemma~\ref{lem-innerhorn},
the restriction map $Z^{X'} \lra Z^X$ is a weak equivalence of Segal spaces.
\end{prop}

\begin{proof} Since $Z^{X'}$ and $Z^X$ are Segal spaces, it suffices to show that the restriction map is a
weak equivalence in degrees 0 and 1. The degree 0 case is
lemma~\ref{lem-innerhorn}. For the degree 1 case, we have to show that the restriction map
\[  \rmap(X'\times\Delta[1],Z) \lra \rmap(X\times\Delta[1],Z) \]
is a weak equivalence. It is an exercise to show that the inclusion
of $X\times\Delta[1]$ in $X'\times\Delta[1]$ is a composition of three relative inner horn extensions.
Therefore lemma~\ref{lem-innerhorn} can be applied one more time. \end{proof}

\medskip
Let $F\co \sA\to \sB$ be a functor between small categories. Let $\cyl(F)$ be the categorical
mapping cylinder of $F$. This is a small category whose set of objects is the disjoint union of $\ob(\sA)$ and $\ob(\sB)$. The morphisms are defined in such a way that $\sA$ and $\sB$ are full subcategories by means of the standard inclusions of
$\ob(\sA)$ and $\ob(\sB)$ into $\ob(\sA)\sqcup \ob(\sB)$, and in addition, there is one distinguished
morphism $s_x\co x\to F(x)$ for every object $x$ in $\sA$. The morphisms $s_x$ are subject to relations $s_y\circ v= F(v)\circ s_x$ whenever
$v\co x\to y$ is a morphism in $\sA$, so that the following diagram in $\cyl(F)$ is commutative by definition:
\[
	\begin{tikzpicture}[descr/.style={fill=white}, baseline=(current bounding box.base)] ]
	\matrix(m)[matrix of math nodes, row sep=2.5em, column sep=2.5em,
	text height=1.5ex, text depth=0.25ex]
	{
	x & y  \\
	F(x) & F(y) \\
	};
	\path[->,font=\scriptsize]
		(m-1-1) edge node [auto] {$v$} (m-1-2)
		(m-1-1) edge node [auto] {$s_x$} (m-2-1)
		(m-2-1) edge node [auto] {$F(v)$} (m-2-2)
		(m-1-2) edge node [auto] {$s_y$} (m-2-2);
	\end{tikzpicture}
\]
In other words, a functor from $\cyl(F)$ to another small category $\sC$ is the same thing
as a triple consisting of a functor $G_0\co \sA\to \sC$, a functor $G_1\co \sB\to \sC$ and a natural transformation from
$G_0$ to $G_1F$.

The ordinary mapping cylinder $\cyl(NF)$ of the map of nerves $NF\co N\sA\to N\sB$ is a simplicial set: the pushout of
\[ N\sA\times\Delta[1] \xleftarrow{x\mapsto (x,1)} N\sA \xrightarrow{NF} N\sB\,. \]
We can think of it as a simplicial subset of the nerve $N(\cyl(F))$, but the inclusion is typically not an isomorphism.

\smallskip
\emph{Example.} Let $F$ be the inclusion of posets $[0]\to [1]$. Then $\cyl(F)$ is isomorphic (as a category) to the
poset $[2]$, and so $N(\cyl(F))$ is identified with the simplicial set $\Delta[2]$.
But $\cyl(NF)$ is clearly a simplicial set which is generated by $1$-simplices. The inclusion
$\cyl(NF)\to N(\cyl(F))$ can be identified with the inner horn inclusion $\Lambda^i[2]\to \Delta[2]$.

\begin{cor}[to lemma~\ref{lem-innerhorn}] \label{cor-innerhorn}
Let $Z$ be a simplicial space which happens
to be a Segal space. Then the restriction map
\[  \rmap(N(\cyl(F)),Z) \to \rmap(\cyl(NF),Z) \]
is a weak equivalence. \emph{(The derived mapping spaces are to be formed in a model structure
on the category of simplicial spaces where the weak equivalences are defined degreewise)}.
\end{cor}

\begin{proof} We are going to show that the inclusion $\cyl(NF)\to N(\cyl(F))$ is an iteration of relative
inner horn inclusions. The nondegenerate
simplices of $N(\cyl(F))$ which are not in $\cyl(NF)$ come in two distinct types.
\begin{itemize}
\item[I.] Strings of $r$ composable non-identity morphisms in $\cyl(F)$, where $r\ge 1$ and precisely
one of the $r$ morphisms has source in $\sA$ and target in $\sB$, and that one is not of the form $s_x$~.
\item[II.] Strings of $r$ composable non-identity morphisms in $\cyl(F)$, where $r\ge 2$ and
precisely one of the $r$ morphisms has source in $\sA$ and target in $\sB$, and that one is of the form $s_x$\,,
and it is not the last morphism of the string in the sense of composition.
\end{itemize}
For a type II nondegenerate $r$-simplex $\sigma$, let $w(\sigma)\in\{1,2,\dots,r-1\}$ describe the position of
the target of the special morphism in the string which has the form $s_x$\,. For example $w(\sigma)=1$ in the case of the
$5$-simplex sketched here:
\[
c_0 \xleftarrow{\quad \quad} c_1 \xleftarrow{\; \textup{special} \;} c_2  \xleftarrow{\quad \quad} c_3  \xleftarrow{\quad \quad} c_4  \xleftarrow{\quad \quad} c_5
\]
There is a bijection from type II simplices to type I simplices given by
$\sigma\mapsto d_{w(\sigma)}\sigma$\,.
Therefore the idea emerges that we can attach $\sigma$ and $d_{w(\sigma)}\sigma$ in one single
relative inner horn extension. To make this work, we need to ensure that the other faces $d_i\sigma$, where
$i\ne w(\sigma)$, have already been attached in previous steps. It is clear that these other faces $d_i\sigma$ are again
of type II (or degenerate) in all cases except possibly when $i=w(\sigma)+1$. Therefore we can proceed as follows.
\begin{itemize}
\item[-] Start by attaching all pairs $(\sigma,d_{w(\sigma)}\sigma)$ where $|\sigma=2|$.
\item[-] Next attach all pairs $(\sigma,d_{w(\sigma)})$ where $|\sigma|=3$ and $w(\sigma)=2$,
then all pairs $(\sigma,d_{w(\sigma)})$ where $|\sigma|=3$ and $w(\sigma)=1$.
\item[-] Next attach all pairs $(\sigma,d_{w(\sigma)})$ where $|\sigma|=4$ and $w(\sigma)=3$,
then all pairs $(\sigma,d_{w(\sigma)})$ where $|\sigma|=4$ and $w(\sigma)=2$,
then all pairs $(\sigma,d_{w(\sigma)})$ where $|\sigma|=4$ and $w(\sigma)=1$.
\item[-] And so on. 
\end{itemize}
\end{proof}


\section{Some postponed lemmas and proofs} \label{sec-diago}
\begin{proof}[Proof of lemma~\ref{lem-diago}]
A standard formula for homotopy limits (see e.g. the last pages of \cite{DwyKa1}) identifies $\holim~F\delta$ with the homotopy limit of
$\RR\delta_*(F\delta)$ alias $\RR\delta_*\delta^*F$\,,
where $\RR\delta_*$ is the homotopy right Kan extension along $\delta$.
So it suffices to show that the unit natural transformation
\begin{equation} \label{eqn-diago1}  F  \lra \RR\delta_*\delta^*F  \end{equation}
between contravariant functors on $\prod_j \simp(Z(j))$
is a weak equivalence. To make that natural transformation explicit, let
\[  \tau=((\ell_1,\tau_1),(\ell_2,\tau_2),\dots,(\ell_m,\tau_m)) \]
be an object of $\prod_j\simp(Z(j))$. We extract a string of integers
\[  \Lambda=(\ell_1,\dots,\ell_m) \]
from it. The formula for $\RR\delta_*\delta^*F$ applied to $\tau$ is
\[  \holimsub{(\delta \downarrow \tau)} F\delta\circ\beta_\tau \]
where $\beta_\tau$ is the forgetful functor from the comma category $(\delta\!\downarrow\!\tau)$
to the source of $\delta$.
The small category $(\delta\!\downarrow\!\tau)$ is isomorphic to another small category
$\ms Q(\Lambda)$ which depends only on the string of integers $\Lambda$.
An object of $\ms Q(\Lambda)$ is a pair $(k,u)$ consisting of some $k\ge 0$ and a map
\[ u\co [k]\lra [\ell_1]\times [\ell_2]\times[\ell_3]\times\cdots\times [\ell_m] \]
whose coordinates $u_i\co [k]\to [\ell_i]$ are monotone.
A morphism from $(k,u)$ to $(k',u')$ is a monotone map $g\co [k]\to[k']$ such that $u'g=u$.
We identify $\ms Q(\Lambda)$ with $(\delta\!\downarrow\!\tau)$.
The forgetful functor $\beta_\tau$ then takes an object $(k,u)$ of $\ms Q(\Lambda)$ to the object
\[  (k,(u_1^*\tau_1,u_2^*\tau_2,\dots,u_m^*\tau_m))~ \]
of $\simp(\prod_j Z(j))$, where $u_i\co [k]\to[\ell_i]$ is the $i$-th coordinate of $u$. The maps
\[  F(\tau)=F((\ell_1,\tau_1),\dots,(\ell_m,\tau_m)) \lra F((k,u_1^*\tau_1),\dots,(k,u_m^*\tau_m))  \]
induced by $u=(u_1,\dots,u_m)$ determine a map
\begin{equation} \label{eqn-diago2} F(\tau) \lra \holimsub{(k,u)\in \ms Q(\Lambda)} F((k,u_1^*\tau_1),\dots,(k,u_m^*\tau_m))
\end{equation}
and this is the explicit form of~(\ref{eqn-diago1}). Therefore it remains to show that~(\ref{eqn-diago2}) is a weak
equivalence for every choice of $\tau$.

Let $\ms P(\Lambda)\subset \ms Q(\Lambda)$ be the full subcategory consisting of the
objects $(k,u)$ where each of the maps $u_i\co [k] \to [\ell_i]$ is onto. We show:
\begin{itemize}
\item[(1)] $|\ms P(\Lambda)|$ is contractible;
\item[(2)] the inclusion $e_\Lambda\co \ms P(\Lambda)\to \ms Q(\Lambda)$ is homotopy terminal.
\end{itemize}
Here is a sketch proof of (1). If $\Lambda=(0,0,\dots,0)$, then $\ms P(\Lambda)=\ms Q(\Lambda)$
is isomorphic to $\simp(\Delta[0])$ which is contractible. In the remaining
cases we can assume without loss of generality that the integers $\ell_1,\ell_2,\dots,\ell_m$ are \emph{all} positive (drop zero coordinates if not).

We now argue by induction on $m$, the length of $\Lambda$. There is a functor $V$ from
$\ms P(\Lambda)$ to the poset of proper subsets of $\{1,\dots,m\}$, as follows. To determine $V(k,u)$ for an object
$(k,u)$ of $\ms P(\Lambda)$, find the minimum $t_u$ of the $t\in [k]$ such that
\[ u(t)\ne (0,0,\dots,0)\in [\ell_1]\times\cdots\times[\ell_m] \]
and let $V(k,u)$ be the subset of $\{1,\dots,m\}$ consisting of elements $j$ such that $u_j(t_u)=0$. Next, for
a proper subset $S$ of $\{1,\dots,m\}$, we look at the comma category $(V\downarrow S)$. This contains
$V^{-1}(S)$, the fiber of $V$ over $S$. That fiber is isomorphic to
$\simp(\Delta[0])\times\ms P(\Lambda-\Bar\chi_S)$, where $\Bar\chi_S$ is the characteristic function of the complement of $S$. Moreover the inclusion of $V^{-1}(S)$ in $(V\downarrow S)$ admits a left adjoint.
Therefore the classifying
spaces of $\ms P(\Lambda-\Bar\chi_S)$ and $(V\downarrow S)$ are weakly equivalent, and so, by inductive assumption,
the classifying space of $(V\downarrow S)$ is contractible. Since
this holds for every $S$, we can apply Quillen's theorem A to deduce that $V$ induces a weak equivalence
from the classifying space of $\ms P(\Lambda)$ to the classifying space of the poset of proper subsets
of $\{1,2,\dots,m\}$. The latter is an $(m-1)$-simplex. ---
For the proof of (2) we observe that, for an object $(k,u)$ of $\ms Q(\Lambda)$, the
comma category $((k,u)\downarrow e_\Lambda)$ is isomorphic to a product of $k$ factors
which have the form $\ms P(\Lambda')$; more precisely, $\Lambda'$ runs through the vectors $u(j)-u(j-1)$
where $j=1,2,\dots,k$. In this way, (2) follows from (1).

Returning to the map~(\ref{eqn-diago2}), we have now earned the right to replace $\ms Q(\Lambda)$ by $\ms P(\Lambda)$
in the target. After that we can replace $F((k,u_1^*\tau_1),\dots,(k,u_m^*\tau_m))$ in the target
by the constant expression $F((\ell_1,\tau_1),\dots,(\ell_m,\tau_m))$ since the maps $u_1,\dots,u_m$ induce a natural weak
equivalence
\[ F((\ell_1,\tau_1),\dots,(\ell_m,\tau_m)) \lra F((k,u_1^*\tau_1),\dots,(k,u_m^*\tau_m)). \]
Here we use our special assumption on $F$.
After these adjustments, the homotopy limit in the target is the space of maps from $|\ms P(\Lambda)|$
to $F((\ell_1,\tau_1),\dots,(\ell_m,\tau_m))$. Since $|\ms P(\Lambda)|$ is contractible, this completes the proof. \end{proof}

\begin{lem} \label{lem-orderinfl} Let $F=N\circ\psi^T$. The map $w^*\co \holim~F \to \holim~Fw$
determined by~\emph{(\ref{eqn-unrav3})} is a weak equivalence.
\end{lem}

\begin{proof} For a simplicial set $X$, let $\nsimp(X)\subset \simp(X)$ be the full subcategory whose objects are the
nondegenerate simplices of $X$. If $X$ is regular (faces of nondegenerate simplices are nondegenerate), then
the inclusion $\nsimp(X)\to \simp(X)$ has a left adjoint $\rho$. In particular,
the inclusion $\nsimp(\ms B^T)\to \simp(\ms B^T)$ has a left adjoint $\rho$ and
$F$ takes the unit morphisms of the adjunction to weak equivalences. Let $F_1$ be the
restriction of $F$ to $\nsimp(\ms B^T)$. Then there is a
commutative diagram
\[
	\begin{tikzpicture}[descr/.style={fill=white}, baseline=(current bounding box.base)] ]
	\matrix(m)[matrix of math nodes, row sep=2.5em, column sep=2.5em,
	text height=1.5ex, text depth=0.25ex]
	{
	\holim~F & \holim~Fw   \\
	\holim~F_1 & \holim~F_1\rho w \\
	};
	\path[->,font=\scriptsize]
		(m-1-1) edge node [auto] {$w^*$} (m-1-2)
		(m-2-1) edge node [auto] {$\simeq$} (m-1-1)
		(m-2-1) edge node [auto] {$(\rho w)^*$} (m-2-2)
		(m-2-2) edge node [auto] {$\simeq$} (m-1-2);
	\end{tikzpicture}
\]
with vertical arrows induced by the unit morphisms of the adjunction.
The lower horizontal arrow is a weak equivalence
because $\rho w$ is homotopy terminal (an exercise).
\end{proof}

\begin{lem} \label{lem-stem} Let $\psi^{T,e}$ be the restriction of $\psi^T$ to $\simp(\ms B^{T,e})$.
The restriction map $\holim~\psi^T \to \holim~\psi^{T,e}$
is a weak equivalence. \emph{(See~(\ref{eqn-unrav2}) and (\ref{eqn-cutposetform}) for notation.)}
\end{lem}
\begin{proof}
We introduce two endofunctors
$E_0$ and $E_1$ of $\simp(\ms B^T)$ and two natural transformations
$\id\Rightarrow E_0 \Leftarrow E_1$\,.
The functor $E_0$ takes an $n$-simplex
\[  R_0\ge R_1\ge \cdots\ge R_n \]
to the $n+1$-simplex
$S_0\ge S_1\ge \cdots \ge S_n \ge S_{n+1}$
where $S_i=R_i$ for $i\le n$ and $S_{n+1}=\emptyset$.
The functor $E_1$ takes the same $n$-simplex $R_0\ge R_1\ge \cdots\ge R_n$ to
\[  S_r\ge S_{r+1}\ge \cdots \ge S_n\ge S_{n+1} \]
where $S_i=R_i$ for $i\le n$ and $S_{n+1}=\emptyset$ as before, while
$r\in \{0,1,\dots,n,n+1\}$ is chosen minimally so that $S_r\in \ms B^{T,e}$.
The natural transformations $\id\Rightarrow E_0\Leftarrow E_1$ are obvious. They induce weak equivalences
\[  \psi^T~\leftarrow~\psi^T E_0~\rightarrow~\psi^T E_1 \]
by inspection. The functors $E_0$ and $E_1$ take $\simp(\ms B^{T,e})$ to itself, and
$E_1$ takes all of $\simp(\ms B^T)$ to $\simp(\ms B^{T,e})$. It follows that the pullback map
$\holim~\psi^T\to \holim~\psi^{T,e}$
has a homotopy inverse, given by the pullback map
\[ \holim~\psi^{T,e}\lra \holim~\psi^{T,e}E_1 \]
where $\holim~\psi^{T,e}E_1=\holim~\psi^T E_1\simeq \holim~\psi^T E_0 \simeq \holim~\psi^T$.
\end{proof}

\emph{Acknowledgements.} 
We are indebted to Victor Turchin and Peter Teichner for helpful comments on earlier versions of this paper and to Geoffroy Horel for sharing his knowledge in model categories with us.


\begin{thebibliography}{10}
\bibitem{Andrade} R.~Andrade, Ph.D thesis, MIT 2010.
\bibitem{AFT1} D.~ Ayala, J.~ Francis, H.~Lee Tanaka, \emph{Local structures on stratified spaces}, Adv. Math. 307 (2017) 903--1028.
\bibitem{AFT2} D.~ Ayala, J.~ Francis, H.~Lee Tanaka, \emph{Factorisation homology of stratified spaces},
Selecta Math. 23 (2017), 293--362.
\bibitem{AxSing94} S.~Axelrod and I.~M.~Singer, \emph{Chern-Simons perturbation theory. II},
J.Diff.Geom. 39 (1994), 173--213.
\bibitem{AroneTurchin} G.~ Arone and V.~Turchin, \emph{On the rational homology of high-dimensional analogues of spaces of long knots}, Geom. Topol. 18 (2014), 1261--1322.
\bibitem{BergerMoerdijk2003} C.~Berger and I.~Moerdijk, \emph{Axiomatic homotopy theory for operads}, Comment. Math. Helv. 78 (2003), 805--831.
\bibitem{BergnerHackney} J. Bergner and P. Hackney, \emph{Group actions on Segal operads}, Israel J. Math. 202 (2014), 423-460.
\bibitem{BousfieldKan} A.~ Bousfield and D.~ Kan, \emph{Homotopy limits, completions and localizations}, Lecture Notes in Math. vol. 304, Springer-Verlag, Berlin-NewYork, 1972.
\bibitem{BoavidaWeiss} P.~Boavida de Brito and M.S.~Weiss, \emph{Manifold calculus and homotopy sheaves}, Homology Homotopy Appl. 15 (2013), 361--383.
\bibitem{BCSK} R. Budney, J. Conant, R. Koytcheff and D. Sinha, \emph{Embedding calculus knot invariants are of finite type}, arXiv:1411.1832 (2014)
\bibitem{CisinskiMoerdijk1} D.-C. Cisinski and I.Moerdijk, \emph{Dendroidal sets as models for homotopy operads},
J. Topol. 4 (2011), 257--299.
\bibitem{CisinskiMoerdijk2} D.-C. Cisinski and I.Moerdijk, \emph{Dendroidal Segal spaces and $\infty$-operads},
J. Topol. 6 (2013), 675--704.
\bibitem{CisinskiMoerdijk3} D.-C. Cisinski and I.Moerdijk, \emph{Dendroidal sets and simplicial operads},
J. Topol. 6 (2013), 705--756.
\bibitem{Dugger} D. Dugger, \emph{Combinatorial model categories have presentations}, Adv. Math. 164 (2001), 177-201.
\bibitem{DwyerHess1} W.~G.~Dwyer and K. Hess. \emph{Long knots and maps between operads}, Geom. Topol. 16 (2012), 919--955.
\bibitem{DwyerHess2} W.~G.~Dwyer and K. Hess. \emph{in preparation}, 2015
\bibitem{DwyKa1} W.~G.~Dwyer and D.~Kan, \emph{A classification theorem for diagrams of simplicial sets},
Topology 23 (1984), 139--155.
\bibitem{DwyKa2}  W.~G.~Dwyer and D.~Kan, \emph{Function complexes in homotopical algebra},
Topology 19 (1980), 427--440.
\bibitem{FultonMacPherson} W.~Fulton and R.~MacPherson, \emph{A compactification of configuration spaces},
Ann. of Math. 139 (1994), 183--225.
\bibitem{GalToKo} I.~Galvez-Carrillo, A.~Tonks and J.~Kock, \emph{Decomposition spaces, incidence algebras and M\"obius inversion}, arXiv:1404.3202.
\bibitem{GoThes} T.~G.~Goodwillie, \emph{A multiple disjunction lemma for smooth concordance embeddings}, Mem. Amer. Math. Soc. 86 (1990), no. 431, viii+317 pp.
\bibitem{GK1} T.~G.~Goodwillie and J.~Klein, \emph{Multiple disjunction for spaces of Poincar\'e embeddings}, J. of Topology
1 (2008), 761--803.
\bibitem{GK2} T.~G.~Goodwillie and J.~Klein, \emph{Multiple disjunction for space of smooth embeddings}, J. of Topology 8 (2015), 675--690.
\bibitem{GoWeEmb} T.~G.~Goodwillie and M.~S.~Weiss, \emph{Embeddings from the point of view of immersion theory, Part II},
Geometry and Topology 3 (1999), 103--118.
\bibitem{Hirschhorn} P.~S.~Hirschhorn, \emph{Model categories and their localizations}, Math. Surveys and Monographs vol.~99,
Amer. Math. Soc., 2002.
\bibitem{Lashof} R.~Lashof, \emph{Embedding spaces}, Illinois J. Math., 20 (1976), pp. 144-154
Graduate Texts in Mathematics, vol.5, Springer-Verlag, New York-Berlin, 1971.
\bibitem{Lurie} J.~Lurie, \emph{Higher Algebra}, preprint, 2014
\bibitem{MSS} M.~Markl, S.~Shnider and J.~Stasheff, \emph{Operads in Algebra, Topology and Physics}, Math.Surveys and Monographs
vol.96, Amer. Math. Soc., Providence, RI, 2002.
\bibitem{Miller} D.~A.~Miller, \emph{Popaths and holinks}, J. Homotopy Relat. Struct. 4 (2009), 265--273.
\bibitem{Miller2} D.~A.~Miller, \emph{Strongly stratified homotopy theory}, Trans. Amer. Math. Soc. 365 (2013), 4933--4962.
\bibitem{MoerdijkWeiss} I.~Moerdijk and I.~Weiss, \emph{Dendroidal sets}, Algebr. Geom. Topol. 7 (2007), 1441--1470.
\bibitem{Puppe} V.~ Puppe, \emph{A remark on ``homotopy fibrations''}, Manuscripta Math. 12 (1974), 113 -- 120.
\bibitem{Quinn}  F.~Quinn, \emph{Homotopically stratified sets}, J. Amer. Math. Soc. 1 (1988), 441--499
\bibitem{RandalWilliams} O.~Randal-Williams, \emph{Resolutions of moduli spaces and homological stability},
Preprint, 2009; arXiv:0909.4278
\bibitem{Rezk} C.~Rezk, \emph{A model for the homotopy theory of homotopy theory}, Trans. Amer. Math. Soc. 353 (2001),
973--1007.
\bibitem{Salvatore} P.~Salvatore, \emph{Knots, operads, and double loop spaces}, Int. Math. Research Notices (2006), Art.~ID 13628, 22 pp.
\bibitem{Sakai} K.~ Sakai, \emph{Deloopings of the spaces of long embeddings}, Fund. Math. 227 (2014), 27--34.
\bibitem{Sinha} D.~Sinha, \emph{Manifold-theoretic compactifications of configuration spaces},
Selecta Math. 10 (2004), 391--428.
\bibitem{Turchin} V.~ Turchin, \emph{Delooping totalization of a multiplicative operad}, J. Homotopy Relat. Struct. 9 (2014), 349--418.
\bibitem{Turchin2} V.~ Turchin, \emph{Context-free manifold calculus and the Fulton-MacPherson operad.} Alg. Geom. Topol. 13  (2013), 1243--1271.
\bibitem{Treumann} D.~Treumann, \emph{Exit paths and constructible stacks}, Compos. Math. 145
(2009), 1504--1532.
\bibitem{Voronov} A.~Voronov, \emph{Notes on universal algebra}, Graphs and patterns in mathematics and theoretical physics, 81--103, Proc. Symp. Pure Math. 73, Amer. Math. Soc., Providence, RI, 2005.
Proc.Sympos.Pure Math. vol.73, Amer.Math.Soc., Providence, RI, 2005.
\bibitem{Wahl} N.~Wahl, \emph{Homological stability for mapping class groups of surfaces}, in: Handbook of moduli,
Vol. III, 547--583, Adv. Lect. Math. 26, Int. Press, Somerville, MA, 2013.
\bibitem{WeissEmb} M.~S.~Weiss, \emph{Embeddings from the point of view of immersion theory, Part I}, Geometry and Topology 3 (1999), 67--101.
\bibitem{Woolf} J.~Woolf, \emph{The fundamental category of a stratified space}, J. Homotopy Rel. Structures 4 (2009), 359--387.


\end{thebibliography}
\end{document}